\documentclass[11pt]{amsart}

\usepackage[margin=1in]{geometry}

\usepackage{amsmath}
\usepackage{amssymb}
\usepackage{amsthm}
\usepackage{enumerate}
\usepackage{makecell}
\usepackage{setspace}
\usepackage{graphicx}
\usepackage{epstopdf}
\usepackage{subfig}
\usepackage{nicematrix}
\usepackage{pdfpages}
\usepackage[normalem]{ulem}
\usepackage[colorlinks=true,linkcolor=black,anchorcolor=black,citecolor=black,filecolor=black,menucolor=black,runcolor=black,urlcolor=black]{hyperref}

\newtheorem{theorem}{Theorem}

\newtheorem{proposition}[theorem]{Proposition}
\newtheorem{lemma}[theorem]{Lemma}
\newtheorem{corollary}[theorem]{Corollary}

\newtheorem{sumtheorem}{Theorem}
\newtheorem{sumlemma}[sumtheorem]{Lemma}
\newtheorem{sumcorollary}{Corollary}[sumtheorem]

\theoremstyle{definition}
\newtheorem{definition}{Definition}
\newtheorem{example}[theorem]{Example}

\newtheorem{remark}[theorem]{Remark}

\numberwithin{theorem}{section}
\numberwithin{definition}{section}

\newcommand{\nC}{{\mathbb C}}

\newcommand{\nH}{{\mathbb H}}

\newcommand{\bs}{{\backslash}}

\newcommand{\Cara}{Carath\'{e}odory }
\newcommand{\Teich}{Teichm\"{u}ller }
\newcommand{\Hol}{H\"{o}lder }
\newcommand{\Mob}{M\"{o}bius }

\newcommand{\chat}{\ensuremath{\hat{\mathbb{C}}}}

\DeclareMathOperator{\Real}{Re}
\DeclareMathOperator{\Imag}{Im}
\DeclareMathOperator{\Arg}{arg}

\DeclareMathOperator{\diam}{diam}

\DeclareMathOperator{\hcap}{hcap}
\DeclareMathOperator{\Fill}{fill}
\DeclareMathOperator{\sgn}{sgn}

\DeclareMathOperator{\Aut}{Aut}
\DeclareMathOperator{\dist}{dist}

\begingroup
\def\2#1{\ifnum#1<10 0\fi\the#1}
\xdef\isodayandtime%
{\the\year-\2\month-\2\day\space\2{\count0}:\2{\count2}}

\endgroup

\renewcommand{\tabcolsep}{0.2cm}
\renewcommand{\arraystretch}{1}

\allowdisplaybreaks
\title{A deterministic approach to Loewner-energy minimizers}
\author{Tim Mesikepp}
\address{Beijing International Center for Mathematics Research, Peking University, China}
\email{tmesikepp@gmail.com}
\date{}

\begin{document}
\begin{abstract}
    We study two minimization questions: the nature of curves $\gamma \subset \mathbb{H}$ which minimize the Loewner energy among all curves from 0 to a fixed $z_0 \in \mathbb{H}$, and the nature of $\gamma$ which minimize the Loewner energy among all curves that weld a given pair $x<0 <y$.   The former question was partially studied by Yilin Wang, who used SLE techniques to calculate the minimal energy and show it is uniquely attained \cite{WangReverse}.  We revisit the question using a purely deterministic methodology, and re-derive the energy formula and also obtain further results, such as an explicit computation of the driving function. Our approach also yields existence and uniqueness of minimizers for the welding question, as well as an explicit energy formula and explicit driving function.  In addition, we show both families have a ``universality'' property; for the welding minimizers this means that there is a single, explicit algebraic curve $\Gamma$ such that truncations of $\Gamma$ or its reflection $-\Bar{\Gamma}$ in the imaginary axis generate all welding minimizers up to scaling.  While Wang noted her minimizer is SLE$_0(-8)$, we show the welding minimizers are SLE$_0(-4,-4)$.  Our results also show sharpness of a case of the driver-curve regularity theorem of Carto Wong \cite{Wong}.
\end{abstract}
\subjclass[2020]{30C35, 30C75}
\maketitle

\section{Introduction and main results}

\subsection{Loewner energy and two minimization questions}
Consider a simple curve $\gamma = \gamma[0,T]$ in the upper half plane $\mathbb{H}$ starting from zero.  For each $0 < t \leq T$, $\mathbb{H} \bs \gamma[0,t]$ is simply connected, and we can thus ``map down'' the curve with a Riemann map $g_t:\mathbb{H}\bs \gamma[0,t] \rightarrow \mathbb{H}$ which fixes $\infty$ and takes the tip $\gamma(t)$ to some point $\lambda(t) \in \mathbb{R}$.  Since every other point $\gamma(s) \in \gamma[0,t)$ on the curve corresponds to two prime ends in $\mathbb{H} \bs \gamma[0,t]$, $g_t$ ``cuts $\gamma(s)$ in half,'' sending it to two points on either side of $\lambda(t)$.  Thus $g_t$ ``unzips'' $\gamma[0,t]$ to two intervals around $\lambda(t)$, as in Figure \ref{Fig:LoewnerFlowB}.  

When we appropriately parametrize $\gamma$ and appropriately normalize $g_t$, this flow of conformal maps for the growing curves $\gamma[0,t]$ satisfies \emph{Loewner's equation}
\begin{align}\label{Eq:LoewnerEqIntro}
    \frac{\partial g_t(z)}{\partial t} =: \dot{g}_t(z) = \frac{2}{g_t(z) -\lambda(t)}, \qquad g_0(z) = z,
\end{align}
where $\lambda(t) := g_t(\gamma(t))$ is the image of the tip, also known as the \emph{driving function} of $\gamma$ (see \S\ref{Sec:BackgroundLoewner} for details).  This differential equation uniquely encodes $\gamma$ through the real-valued function $\lambda$, allowing one to analyze $\gamma$ by means of $\lambda$.
\begin{figure}
    \centering
    \includegraphics[scale=0.16]{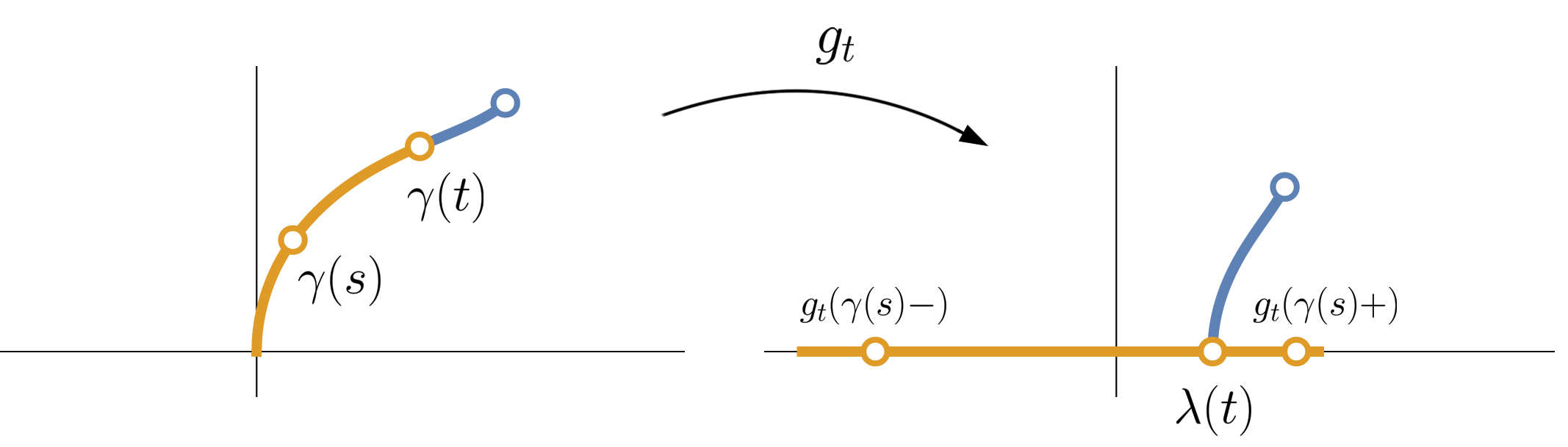}
    \caption{\small The ``mapping down'' function $g_t$ ``unzipping'' the curve $\gamma$.}
    \label{Fig:LoewnerFlowB}
\end{figure}
For example, one can consider the \emph{Loewner energy} $I(\gamma)$ of $\gamma$, which is defined as the Dirichlet energy of the driving function,
\begin{align}\label{Eq:LoewnerEnergy1}
    I(\gamma) := \frac{1}{2}\int_0^\infty \dot{\lambda}(t)^2dt,
\end{align} 
when $\lambda$ is absolutely continuous, and $I(\gamma) = +\infty$ otherwise.  One can think of the Loewner energy as a measurement of the deviation of $\gamma$ from a hyperbolic geodesic, as $I(\gamma) = 0$ if and only if $\gamma$ is a hyperbolic geodesic (see \S\ref{Sec:BackgroundLoewnerEnergy} for the trivial argument). Moreover, $I(\gamma)$ is a conformally-natural measurement, in the sense of being invariant under conformal automorphisms of $\mathbb{H}$, as shown by Yilin Wang \cite{WangReverse}.  While the energy appears to be defined by means of the parametrization of $\lambda$, Wang also showed \cite{WangEquiv} it can be expressed purely in terms of Riemann maps $g_t$ for $\gamma$.  The same work additionally shows that the loop version of Loewner energy (where the integral in \eqref{Eq:LoewnerEnergy1} becomes over all of $\mathbb{R}$ instead of just $\mathbb{R}_{\geq 0}$) characterizes the Weil-Petersson quasicircles of the universal \Teich space $T_0(1)$.  That is, finite-energy loops are the closure of smooth Jordan loops in the Weil-Petersson metric.  Bishop \cite{Bishop, Bishop2} built off this to give a slew of $\ell^2$-type geometric characterizations of finite-energy curves, giving the further intuition that finite-energy $\gamma$ are precisely those possessing ``square-summable curvature over all positions and scales.''\footnote{Here ``curvature'' refers to functionals such as the $\beta$-numbers from geometric measure theory, and not to classical notions that require second derivatives; $\gamma$ need not be in $C^2$, rather, its arc-length parametrization is $H^{3/2}$ \cite{Bishop}.}

While we discuss further background of the Loewner energy in \S\ref{Sec:BackgroundLoewnerEnergy}, it is already evident that $I(\gamma)$ is a noteworthy functional on curves. It is only natural, then, to ask what $\gamma$ minimize it.  If there are no constraints, then the answer is immediate and uninteresting: $\gamma$ is a straight line orthogonal to $\mathbb{R}$, which has zero energy.  In this paper we study the next case beyond this, considering the nature of ``one-point'' minimizers in the following two senses:
\begin{enumerate}[$(i)$]
    \item\label{Q:HMinimizer} What is the infimal energy among all $\gamma$ which pass through a given point $z=re^{i\theta} \in \mathbb{H}$, and what is the nature of minimizers, if they exist?
    \item\label{Q:WeldingMinimizer} What is the infimal energy among all $\gamma$ which start from 0 and weld a given $x<0<y$ to their base, and what is the nature of minimizers, if they exist?
\end{enumerate}
As we will see, question $(\ref{Q:HMinimizer})$ has already been partially investigated, but we revisit it using entirely different techniques, yielding fresh proofs for what has been known and also obtaining further results.  To our knowledge, the second question is thus far unexplored.   

\subsection{Answering question (\ref{Q:HMinimizer}): the one-point minimizers}\label{Intro:Wang}
Wang considered curves from 0 to some $re^{i\theta} \in \mathbb{H}$ and showed the infimal energy in this family is
\begin{align}\label{Eq:Energy1}
    I(\gamma) = -8\log(\sin(\theta)),
\end{align} 
and that it is furthermore uniquely attained \cite[Prop. 3.1]{WangReverse} ($r$'s absence in this formula reflects scale-invariance of the energy).  We call these curves the \emph{energy minimizers for one point}, or \emph{EMP curves}, for short.  Wang's argument was stochastic in nature, using Schramm's formula for the probability that SLE passes to the right of $re^{i\theta}$ \cite{SchrammPerc}, combined with a SLE$_{0+}$ large-deviations result.  While the argument is undeniably elegant, one wonders if these results could be re-derived without resorting to the probability toolbox, as they are intrinsically deterministic.  We answer here in the affirmative, and are able to derive (and re-derive) the following about EMP curves without probabilistic machinery.  See Theorem \ref{Thm:Wang} for the precise statements.

\begin{sumtheorem}\label{SumTheorem:Wang}
    \begin{enumerate}[$(i)$]
        \item \cite[Prop. 3.1]{WangReverse} EMP curves $\gamma_\theta$ exist, are unique, satisfy the energy formula \eqref{Eq:Energy1} and are (downwards) SLE$_0(-8)$.
        \item We have explicit formulas for both the driving function and the conformal welding (\eqref{Eq:WangDriver} and \eqref{Eq:WangWeld}, respectively).
        \item\label{SumThm:AUniversal} For any fixed $\theta \in (0,\pi)\bs \{\pi/2\}$, both the driver and the welding for the single EMP curve to $e^{i\theta}$ extend (using the identical formulas) to generate \emph{all} EMP curves up to scaling, translation and reflection in the imaginary axis.
        \item As $\theta \rightarrow 0$ or $\theta \rightarrow \pi$, $\gamma_\theta$ limits to an explicit algebraic curve which is part of a ``boundary geodesic pair.'' 
    \end{enumerate}
\end{sumtheorem}
\noindent We proceed in \S\ref{Intro:TheoremADescription} to summarize the content of Theorem \ref{SumTheorem:Wang} and then in \S\ref{Intro:TheoremACorollaries} to describe two corollaries.

\subsubsection{Overview of Theorem \ref{SumTheorem:Wang}}\label{Intro:TheoremADescription}
We start in part $(\ref{Thm:WangEnergy})$ of the theorem by deterministically re-deriving the minimal-energy formula \eqref{Eq:Energy1}. Using the resulting system of differential equations, we also re-derive Wang's result \cite[(3.2)]{WangReverse} that the minimizer is (downwards) SLE$_0(-8)$ with the force point $U(t)$ starting at $U(0) = e^{i\theta}$.  Recall that this means that the driving function $\lambda_\theta$ for $\gamma_\theta$ evolves according to 
\begin{align}\label{Eq:SLE0FirstLook}
    \dot{\lambda}_\theta(t) = \Real\Big(\frac{-8}{\lambda_\theta(t)-U(t)}\Big), \qquad \lambda_\theta(0) = 0,
\end{align}
while all other points $z(t)$, including $U(t)$, satisfy the standard Loewner equation \eqref{Eq:LoewnerEqIntro} determined by $\lambda_\theta$.  In particular, $\lambda_\theta$ experiences a horizontal attractive force towards $U$, while $U$ is horizontally repulsed by $\lambda_\theta$.\footnote{See \S\ref{Intro:Welding} for more discussion of SLE$_0(\rho)$, and  \S\ref{Sec:BackgroundSLE} for general background on SLE$_\kappa(\rho_1, \ldots, \rho_n)$ processes.}

By further studying the system of ODE's, we proceed in part $(\ref{Thm:WangDriver})$ of the theorem to explicitly describe SLE$_0(-8)$ by computing its driving function and conformal welding, with the latter building off calculations in \cite{MRW}.
\begin{figure}
    \centering
    \includegraphics[scale=0.13]{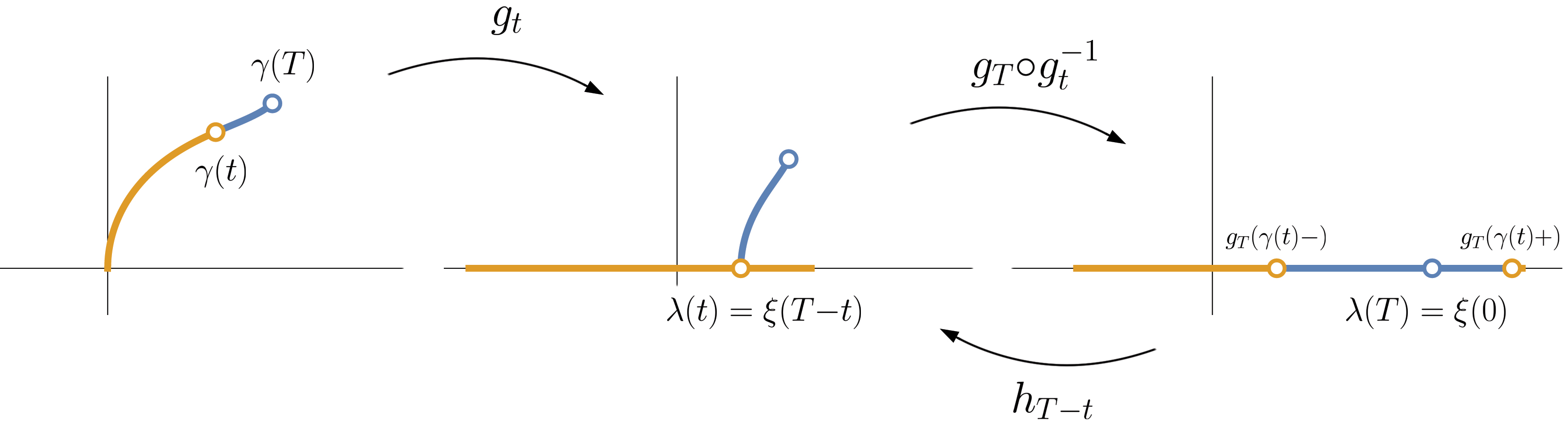}
    \caption{\small The downwards and upwards Loewner flows for $\gamma = \gamma[0,T]$.}
    \label{Fig:LoewnerFlow}
\end{figure}

Part $(\ref{Thm:WangUniversal})$ of the theorem says that the drivers and weldings have a ``universality'' property, and to help explain this, we note that it is natural to work with the ``upwards'' driver for the EMP curves.  By this we mean the function $\xi_\theta(t) := \lambda_\theta(T-t)$ arising from ``reversing the movie'' of the $g_t$ by considering the conformal maps $h_t := g_{T-t} \circ g_T^{-1}$, as in Figure \ref{Fig:LoewnerFlow}.  By \eqref{Eq:LoewnerEqIntro} the $h_t$ satisfy
\begin{align}\label{Eq:LoewnerEqIntroUp}
    \dot{h}_t(z) = \frac{-2}{h_t(z) - \xi(t)}, \qquad h_0(z) = 0,
\end{align}
and we call the dynamics given by  \eqref{Eq:LoewnerEqIntro} and \eqref{Eq:LoewnerEqIntroUp} the \emph{forwards} and \emph{backwards} Loewner flows, respectively.  We will often speak more colloquially and simply call them, respectively, the \emph{downwards} and \emph{upwards} flows, as Figure \ref{Fig:LoewnerFlow} naturally suggests.  We use $\lambda$ to denote a \emph{forwards/downwards} driving function appearing in \eqref{Eq:LoewnerEqIntro}, and $\xi$ for a \emph{backwards/upwards} driving function appearing in \eqref{Eq:LoewnerEqIntroUp}. 

Let $\xi_\theta$ be the upwards driver which generates the EMP curve $\gamma_\theta$ from 0 to $e^{i\theta}$ on the time interval $[0,\tau_\theta]$, where $0 < \theta< \pi/2$ is fixed. The point of part ($\ref{SumThm:AUniversal}$) of the theorem is that $\xi_\theta$ not only generates $\gamma_\theta$, but also \emph{all} EMP curves for angles $0 < \alpha < \pi/2$, up to translation and scaling.  That is, $\xi_\theta$ is well-defined for all $t \geq 0$, and the curve it generates on any interval $[0,t]$ is a EMP curve for some $e^{i\alpha_t}$ (after translation and scaling), where $\alpha_t$ is decreasing in $t$ and has the entire interval $(0,\pi/2)$ as its range.  By symmetry, $-\xi_\theta$ generates all EMP curves for angles $\pi/2 < \alpha <\pi$ (up to scaling and translation), and so each $\xi_\theta$ is actually universal for all the EMP curves. We also show that the conformal welding has an analogous universality property: the formula \eqref{Eq:WangWeld} for a single EMP curve is defined for all $x<0$, and welds any EMP curve $\gamma_\alpha$ (up to scaling and reflection, and excluding the trivial case $\alpha=\pi/2$) upon restriction to an appropriate interval.


In part ($\ref{Thm:Wang0+}$) we send $\theta$ tends to zero and show that taking the na\"{i}ve $\theta \rightarrow 0^+$ limit in our driving function formula yields the correct driver $\xi_0(t) = -\frac{8}{\sqrt{3}}\sqrt{t}$ for the limiting curve $\gamma_0$, and that furthermore $\gamma_0$ is a subset of an explicit cubic algebraic variety.  One subtlety here is that the limiting curve touches the real line at $x=1$ and thus has infinite energy, and so usual compactness tools for collections of curves of finite energy are not available.  

Beyond identifying the limit, we show that $\gamma_0$ and its reflection $\gamma_0^*$ across $\partial \mathbb{D}$ is a \emph{boundary geodesic pair}, in the sense that $\gamma_0$ is the hyperbolic geodesic from 0 to 1 in its component of $\mathbb{H} \bs \gamma_0^*$, and $\gamma_0^*$ is the hyperbolic geodesic from 1 to $\infty$ in its component of $\mathbb{H} \bs \gamma_0$.  See Figure \ref{Fig:Gamma0}.  

\subsubsection{Two corollaries of Theorem \ref{SumTheorem:Wang}}\label{Intro:TheoremACorollaries}
A geodesic pair in the sense of Figure \ref{Fig:Gamma0} slightly extends that defined by Marshall, Rohde and Wang \cite{MRW}, in that the point of intersection $\zeta = \gamma_0 \cap \gamma_0^*$ between the two geodesics is on the boundary of the domain.\footnote{We note that such geodesic pairs have been considered by Krusell under the name of ``fused geodesic multichords'' \cite[\S 4]{Krusell}.  See \cite{Alberts} and \cite{PeltWang} for the ``unfused'' case, where the geodesics do not intersect on the boundary.}  In summary form, our considerations show:
\begin{sumcorollary}[Corollary \ref{Cor:BoundaryGeodesic}]
    If $D \subsetneq \mathbb{C}$ is a simply-connected domain with boundary prime ends $a_1,a_2$ and $\zeta$, where $\zeta \notin \{a_1,a_2\}$, there is a unique boundary geodesic pair $\gamma_1 \cup \gamma_2$ in $(D;a_1,a_2,\zeta)$.
\end{sumcorollary}
\noindent In our example, $(D;a_1,a_2,\zeta) = (\mathbb{H}; 0,\infty,1)$ and $\gamma_1 \cup \gamma_2 = \gamma_0 \cup \gamma_0^*$.

Our second corollary of Theorem \ref{SumTheorem:Wang} is the sharpness of a theorem of Carto Wong.  In \cite{Wong}, Wong studied how the regularity of a driver $\lambda$ relates to the smoothness of the curve $\gamma = \gamma^\lambda$ it generates.  Using a novel observation about the Lipschitz continuity of the map $\lambda \mapsto \gamma$ when $\lambda$ has small $1/2$-H\"{o}lder norm, he proved $\gamma$ generally has half a degree of regularity more than $\lambda$.  That is, if $\lambda \in C^\beta$, then $\gamma \in C^{\beta +1/2}$.  There is an asterisk, however: when $\lambda \in C^{3/2}$, Wong only proved that $\gamma$ is ``weakly'' $C^{1,1}$, not $C^2$, which is to say the capacity parametrization $t \mapsto \gamma(t)$ satisfies
\begin{align}\label{Ineq:WeakC11Intro}
    |\gamma'(s) - \gamma'(t)| \leq C\max\Big\{1, \log \frac{1}{|s-t|} \Big\}|s-t|
\end{align}
nearby any $t>0$ (there is a slight technicality at $t=0$ which is irrelevant for our considerations; see \S\ref{Sec:CorollaryWong} for details and the Wong's precise theorem statement). Lind and Tran \cite[Example 7.1]{LindTran} subsequently gave an example of a $C^{3/2}$ driver whose associated curve was $C^{1,1}$ but not $C^2$, showing Wong's result was close to optimal. It has not been known, however, if the logarithm factor in \eqref{Ineq:WeakC11Intro} is necessary.  Combined with a conformal map computation from \cite{MRW}, our explicit driver formula  in Theorem \ref{SumTheorem:Wang}$(\ref{Thm:WangDriver})$ allows us to answer in the affirmative:
\begin{sumcorollary}[Corollary \ref{Cor:WongSharp}]
    Wong's theorem is sharp: for any $\theta \in (0,\pi)\bs \{\pi/2\}$, $\gamma = \gamma_\theta \cup \gamma_\theta^*$ has driver which is $C^{3/2}$ and capacity parametrization that is weakly $C^{1,1}$, where one cannot replace the logarithm in \eqref{Ineq:WeakC11Intro} with any function of slower growth.
\end{sumcorollary}
\noindent Combined with the example from Lind and Tran, this shows there are drivers of the same regularity which generate curves with capacity parametrizations of differing regularity (Corollary \ref{Cor:Regularity}).

Just as both driving functions are $C^{3/2}$, the Riemann maps for the Lind-Tran example and for the EMP-curve geodesic pair have the same boundary regularity.  The regularities of the arc-length parametrizations of the two curves are different, however (see Table \ref{Table:WongSharp}).  We classify why this happens in Lemma \ref{Lemma:ArcLengthCoeff} in terms of the coefficients of the local expansion of the conformal map.  

\begin{figure}
    \centering
    \includegraphics[scale=0.1]{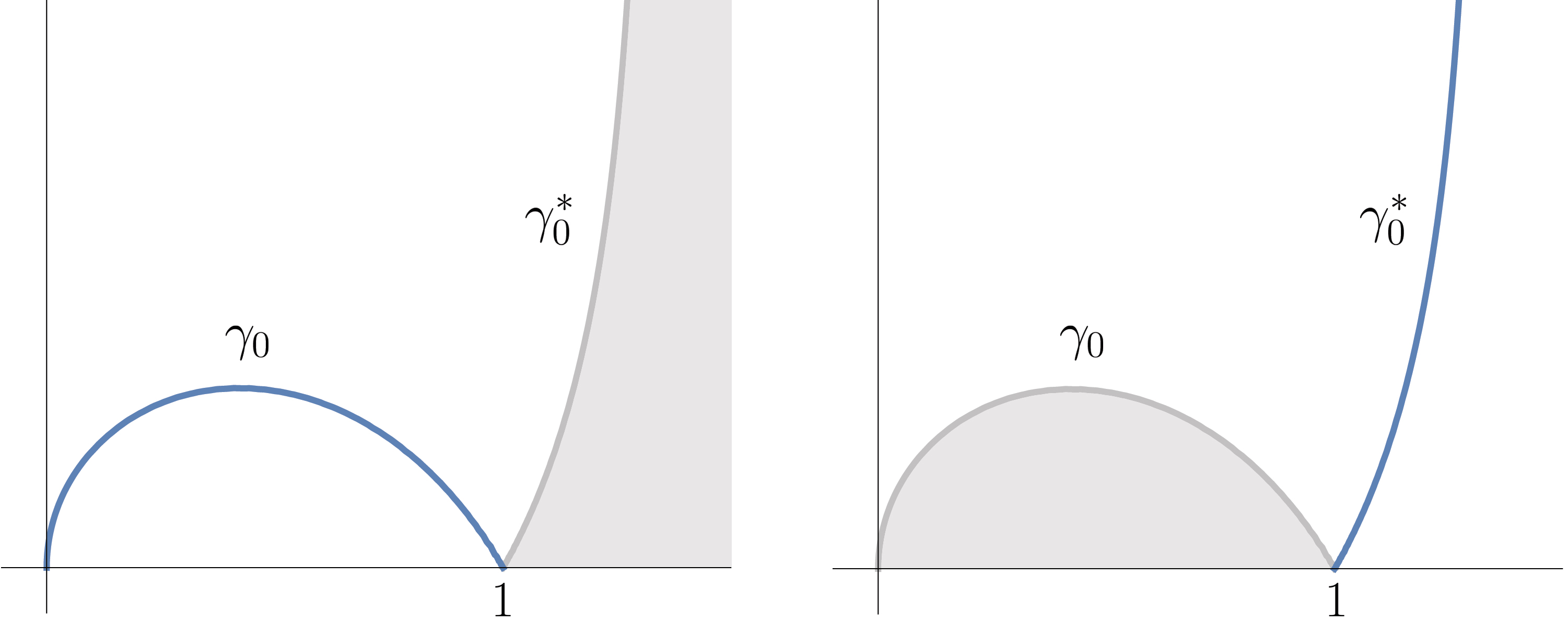}
    \caption{\small The limiting curve $\gamma_0$ for the EMP curves and its reflection $\gamma_0^*$ across $\partial \mathbb{D}$ form a \emph{boundary geodesic pair}, in the sense that $\gamma_0$ is the hyperbolic geodesic from 0 to 1 in its component of $\mathbb{H} \bs \gamma_0^*$, and $\gamma_0^*$ is a hyperbolic geodesic from 1 to $\infty$ in its component of $\mathbb{H} \bs \gamma_0$.}
    \label{Fig:Gamma0}
\end{figure}

\subsection{Answering question (\ref{Q:WeldingMinimizer}): the welding minimizers}\label{Intro:Welding}

Our deterministic techniques yield a similar scope of results for the second question.  To clarify the question statement, recall that the conformal welding $\varphi_t$ of the curve $\gamma$ is the map (determined by $\gamma$ and normalized by $g_t$) which identifies the two images of the points $\gamma(s)$ ``split in half'' by $g_t$.  That is, $\varphi_t$ satisfies $g_t^{-1}(x) = \gamma(s) = g_t^{-1}(\varphi_t(x))$, where $g_t^{-1}(x)$ is the extension of $g_t^{-1}$ to $x \in \mathbb{R}$.  For example, $\varphi_t(g_t(\gamma(s)-)) = g_t(\gamma(s)+)$ in Figure \ref{Fig:LoewnerFlowB}, while similarly $\varphi_T(g_T(\gamma(t)-)) =g_T(\gamma(t)+)$ in Figure \ref{Fig:LoewnerFlow}.  We typically normalize our weldings by considering the \emph{centered downwards flow} maps $G_t(z) := g_t(z) - \lambda_t$, in which case $\varphi_t$ maps an interval $[x_t,0]$ to the left of the origin homeomorphically to an interval $[0,y_t]$ on the right.

The precise formulation of the welding minimization question is thus: for fixed $x<0<y$, what is the infimal energy among all curves $\gamma: [0,T] \rightarrow \mathbb{H}\cup \{0\}$ which satisfy $\gamma(0)=0$ and have conformal welding $\varphi_T:[x,0]\rightarrow [0,y]$ under $G_T$?  The following theorem computes this energy, shows minimizers exist, and studies their properties.  We call these curves the \emph{energy minimizers for weldings (EMW)} family. See Theorem \ref{Thm:EMW} for precise statements.
\begin{sumtheorem}\label{SumTheorem:EMW}
    \begin{enumerate}[$(i)$]
        \item\label{SumThm:EMWSLE0} EMW curves $\gamma$ exist, are unique, satisfy the energy formula 
        \begin{align*}
            I(\gamma) = -8\log\Big(\frac{2 \sqrt{-xy}}{y-x}\Big),
        \end{align*}
        and are (upwards) SLE$_0(-4,-4)$.
        \item We have an explicit formula \eqref{Eq:EMWDriverFormula} for the driving function and an implicit formula \eqref{Eq:EMWWeldImplicit} for the conformal welding.
        \item\label{SumThm:EMWUniversal} There is an explicit algebraic curve which generates all EMW curves up to scaling and reflection in the imaginary axis.
        \item\label{SumThm:DistinctFamilies} An EMP curve $\gamma_\theta$ coincides with an EMW curve if and only if $\theta=\pi/2$. 
    \end{enumerate}
\end{sumtheorem}
Theorem \ref{SumTheorem:EMW} shows that the EMP and EMW curves share many similarities.  For instance, property $(\ref{SumThm:EMWUniversal})$ is a ``universality'' statement: we show there exists a single curve $\Gamma:[0,T] \rightarrow \mathbb{H}\cup\{0\}$ such that for every ratio $r \in (0,1)$, $\Gamma[0,t_r]$ is a welding minimizer for points $x_r < 0 < y_r$ with $-x_r/y_r=r$.  That is, up to scaling, $\Gamma$ single-handedly parametrizes all the welding minimizers with ratios between 0 and 1 (and its reflection $-\overline{\Gamma}$ across the imaginary axis does so for ratios $1<r<\infty$).  See Figure \ref{Fig:UniversalEMW}.  Furthermore, like the limiting EMP curve $\gamma_0$, $\Gamma$ is algebraic, and we explicitly identify the associated real quartic variety \eqref{Eq:EMWUniversalVariety} and show that its complex square $\Gamma^2$ is nothing other than a circle tangent to $\mathbb{R}$.

\begin{figure}
    \centering
    \includegraphics[scale=0.5]{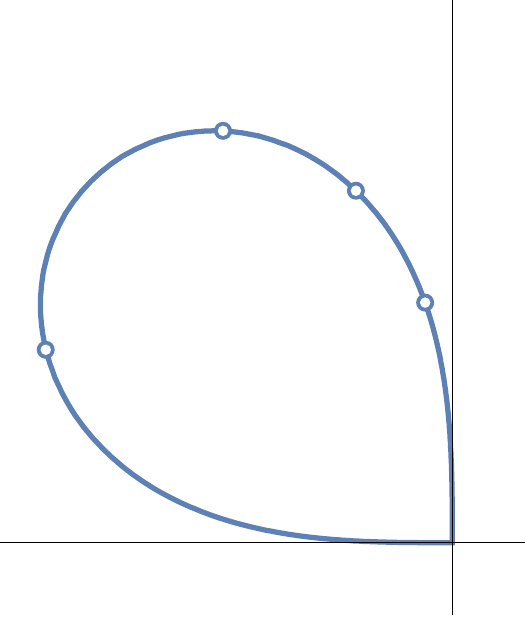}
    \caption{\small The universal curve $\Gamma$ for the welding minimizers with ratios $0 <r = -x/y <1$.  Starting from closest to the imaginary axis and moving counterclockwise, the four points show where to truncate for EMW curves with welding endpoint ratios $r=0.75, 0.5, 0.25$ and $0.01$, respectively.  The complex square of $\Gamma$ is a circle tangent at the origin to $\mathbb{R}$.}
    \label{Fig:UniversalEMW}
\end{figure}

Another parallel is that both families are SLE$_0$ with forcing.  As we saw above, the EMP curves $\gamma_\theta$ are downwards SLE$_0(-8)$ with forcing starting at the tip of the curve.  Theorem \ref{SumTheorem:EMW}($\ref{SumThm:EMWSLE0}$) says that the EMW family is \emph{upwards} SLE$_0(-4,-4)$, with force points starting at $V_1(0)=x, V_2(0)=y$, the two points to be welded together.  That is, the upwards driving function $\xi(t) := \lambda(T-t)-\lambda(T)$ satisfies
\begin{align}\label{Eq:UpSLE_0-4-4}
    \dot{\xi}(t) = \frac{4}{\xi(t)-V_1(t)} + \frac{4}{\xi(t) -V_2(t)}, \qquad \xi(0) = 0,
\end{align}
while all other points evolve according to \eqref{Eq:LoewnerEqIntroUp}. In particular, $\xi$ is repulsed by both the welding points, while both the welding points are attracted to $\xi$.\footnote{Note that another way to formulate version of SLE$_0$ is as the standard downwards SLE$_0(-4,-4)$, but with the two force points placed at the two prime ends of $\mathbb{H}\bs \gamma$ corresponding to the base of $\gamma$.  The precise meaning of this, though, is the above upwards-flow formulation. Thus one should keep in mind that upwards SLE$_0$ and downwards SLE$_0$ generate the same curve $\gamma$ so long as one handles the force points appropriately.  We choose to use the ``upwards'' terminology to match the natural picture of ``mapping up'' with a conformal map $F:\mathbb{H} \rightarrow \mathbb{H}\bs \gamma$ which welds $x$ to $y$.}

We recall that downwards SLE$_\kappa$ with forcing, introduced in \cite{LSW}, frequently appeared in the early days of SLE  (\cite{DubedatMart,DubedatComm, SchrammWilson} for example), and continues to be ubiquitous (see \cite{AngHoldenSun} and \cite{MSImaginary} for recent examples, for instance).  Its $\kappa \rightarrow 0$ limit, SLE$_0$, has also appeared elsewhere: \cite[Theorem 2.2]{Alberts}, for example, shows that the Loewner flow of a geodesic multichord in the sense of Peltola and Wang \cite{PeltWang} is downwards SLE$_0$ with forcing of $+2$ at the critical points and $-4$ at the poles of an associated real rational function.\footnote{The setup in \cite{Alberts} varies slightly from ours in several senses; one is that the Loewner flow is defined on all of $\chat$, while we only consider the upper half plane.  Recasting the result of \cite{Alberts} in terms of our setting in $\mathbb{H}$ would place forcing of $+2$ at the critical points of the rational function and $-8$ at its poles in $\mathbb{H}$ (by conjugation-symmetry of the poles).}  Furthermore, many of deterministic drivers that provided fodder for early investigations into the chordal Loewner equation also turn out to be SLE$_0$ with forcing.  For instance, Krusell \cite{KrusellComm} has shown that the driving functions $\lambda_c(t) = c - c\sqrt{1-t}$, $c \in \mathbb{R}$, initially studied in \cite{KNK,LMR,Lind4,LoewnerCurvature}, all correspond to downwards SLE$_0(\rho)$ with forcing for appropriate choices of $\rho$ and forcing starting point.  

In a similar vein, we show in the appendix that upwards SLE$_0(\rho,\rho)$ has also appeared before, albeit under different guises.  
\begin{sumlemma}[Lemma \ref{Lemma:UpwardsSLE}]
    A straight line segment is upwards SLE$_0(-2,-2)$, while an orthogonal circular arc is upwards SLE$_0(-3,-3)$.
\end{sumlemma}
\noindent Combined with the EMW curves, we thus have a cohesive picture of upwards SLE$_0(\rho,\rho)$ for $\rho \in \{-4,-3,-2\}$.


\subsection{Asymptotic energy comparisons}\label{Sec:IntroAsymp}
\def\pict2scale{0.65}
\begin{figure}
    \centering
    \includegraphics[scale=0.13]{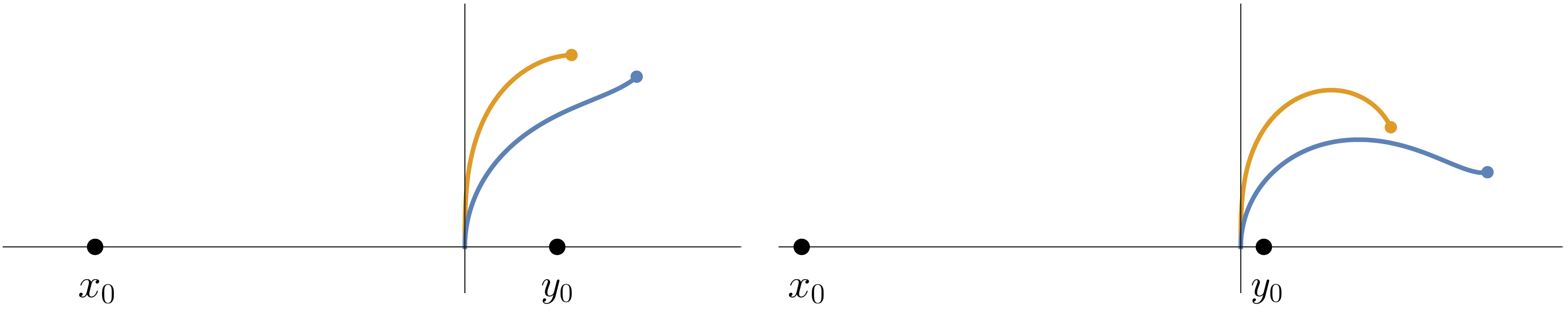}
    \caption{\small EMW and EMP curves in orange and blue, respectively, that weld the same black points $x_0<0<y_0$ to their base.  Note the increase in argument of the tip when moving from the Wang to the welding minimizer, which is expected from the monotonicity of \eqref{Eq:Energy1}.}
    \label{Fig:WangEMWSameWeld}
\end{figure}

Theorem \ref{SumTheorem:EMW}$(\ref{Thm:DistinctFamilies})$ says EMP and EMW curves never coincide, except in the trivial vertical line case.  It is then natural to ask how their energies compare. For instance, for fixed $\theta$, how far from optimal is the energy of the EMW curve with tip at $e^{i\theta}$?  Similarly, for fixed $x<0<y$, how much more energy does the EMP curve need to weld $x$ to $y$ compared to the EMW curve?  We find that the asymptotic answers are the same.
\begin{sumtheorem}[Theorem \ref{Thm:EMWWang}]
    \begin{align}\label{Lim:Asymptotics1}
    \lim_{\theta \rightarrow \pi/2} \frac{I(\text{EMW $\gamma$ to  $e^{i\theta}$})}{I(\text{EMP $\gamma$ to $e^{i\theta}$})} = \Big( \frac{4}{\pi} \Big)^2 = \lim_{-x/y \rightarrow 1} \frac{I(\text{EMP $\gamma$ welding $x$ to $y$})}{I(\text{EMW $\gamma$ welding $x$ to $y$})}.
\end{align}
\end{sumtheorem}
\noindent That both limits are the same is not obviously the case, as we discuss in Remark \ref{Remark:Reciprocals}.

We might also wonder how close generic smooth curves $\gamma$ are to minimizing energy at small scales.  For a small initial segment $\gamma[0,\delta]$, for instance, how does $I(\gamma[0,\delta])$ compare to the minimal energy needed to reach $\gamma(\delta)$?  Similarly, if $u(\delta)<0<v(\delta)$ will be welded by a smooth upwards driver $\xi$ at (small) time $\delta$, how close to energy-minimizing is $\xi$ on $[0,\delta]$?  We find that, as in \eqref{Lim:Asymptotics1}, both ratios are asymptotically the same.
\begin{sumtheorem}[Theorem \ref{Lemma:Inf98}]
If $\lambda$ and $\xi$ are sufficiently smooth with $\dot{\lambda}(0), \dot{\xi}(0) \neq 0$ (i.e. both curves are not locally hyperbolic geodesic segments), then
    \begin{align}\label{Lim:Asymptotics2}
    \lim_{\delta \rightarrow 0^+} \frac{I(\gamma[0,\delta])}{I(\text{EMP curve to }\gamma(\delta))} = \frac{9}{8} = \lim_{\delta \rightarrow 0^+} \frac{I(\xi[0,\delta])}{I(\text{EMW minimizer for }u(\delta), v(\delta))}.
\end{align}
\end{sumtheorem}
\noindent Since this result translates to any time $t$, we see smooth curves which are using energy are never locally minimizing, with constant rate of local energy inefficiency.


\subsection{Methods}\label{Intro:Methods} Our paper is philosophically similar to \cite{KNK} in that we use symmetries of the family in question to obtain systems of ODE's, that, given sufficient patience, we can explicitly solve.  The symmetries of each family also naturally yield the universality properties.  We use tools from quasiconformal mappings to obtain compactness and hence existence of minimizers. In the case of the welding minimizers, we also need a result from \cite{TVY} saying driver convergence implies welding convergence. 

We also use two ``even approach'' properties for finite-energy curves under the downwards and upwards Loewner flows.  The downwards-flow result, Lemma \ref{Lemma:WangEvenApproach}, says that the argument of the image  $G_t(\gamma(\tau))$ of the tip $\gamma(\tau)$ of a finite-energy curve under the centered downwards map $G_t$ approaches $\pi/2$ as $t \rightarrow \tau^-$.  For an example, compare the argument from the base of the curve to the blue tip in the left and center images of Figure \ref{Fig:LoewnerFlowB}.  The upwards-flow version, Lemma \ref{Lemma:EvenApproach}, says that if $x_0 < 0 <y_0$ will be welded by a finite-energy driver $\xi$ at some time $\tau$, then their images $x_t, y_t$ under the centered upwards-flow maps $h_t(z)-\xi(t)$ satisfy $-\frac{x_y}{y_t} \rightarrow 1$ as $t \rightarrow \tau^-$.  These properties enable us to see that both energy-minimization problems are local in nature, which yields ODE's that we can solve for the energy formulas.  We show in Example \ref{Eg:NotEquiv} that these properties are not equivalent.


For the asymptotic energy comparisons, we use expansions of the curve $\gamma$ in terms of $\dot{\lambda}(0)$ from \cite{LindTran}, and a new expansion \eqref{Eq:InfinitesimalWelding} of the welding in terms of $\dot{\xi}(0)$.  These enable us to say that \eqref{Lim:Asymptotics2} holds whenever $\lambda, \xi \in C^{\frac{3}{2} + \epsilon}$.  

\subsection{Outline of the paper}\label{Intro:Outline}
We recall relevant background in \S\ref{Sec:Background}, and in \S\ref{Sec:Wang} we state and prove our results for the EMP curves, culminating in Theorem \ref{Thm:Wang} (Theorem \ref{SumTheorem:Wang}).  In \S\ref{Sec:EMW} we similarly prove properties of the EMW curves, with Theorem \ref{Thm:EMW} (Theorem \ref{SumTheorem:EMW}) the main result.  We prove the energy comparisons in \S\ref{Sec:Compare}, and connect upwards SLE$_0(\rho,\rho)$ to the classical deterministic drivers in the first appendix \S\ref{Sec:Appendix}.  In the second appendix \S\ref{Sec:Appendix2} we compare the orthogonal circular arc family to the EMP and EMW families.

\subsection{Acknowledgements}
The author is grateful to Steffen Rohde for many illuminating discussions as well as for the idea to consider the welding minimizers and to deterministically re-derive Wang's formula \eqref{Eq:Energy1}.  He also thanks Don Marshall and Yilin Wang for multiple instructive discussions, as well as Ellen Krusell for introducing him to SLE$_\kappa(\rho)$ and for the suggestion that the EMW drivers may be related to SLE$_0(-4,-4)$.  He is grateful for feedback from Daniel Meyer and Joan Lind on an earlier draft, and feels indebted for the time investment made by the careful referee, whose comments and corrections have notably improved the manuscript. 

This research was partially conducted while the author was at Mathematical Sciences Research Institute during the spring 2022 semester and is thus partially supported by the National Science Foundation under Grant No. DMS-1928930.

\section{Background and notation}\label{Sec:Background}

\subsection{The Loewner equation}\label{Sec:BackgroundLoewner}
We start with a very brief review of the Loewner equation.  For other overviews, see, e.g., \cite{SLEGuide}, \cite{Lind4} or \cite[\S 2]{WangReverse}, and consult \cite{Beliaev}, \cite{Kemppainen} or \cite{Lawler} for more thorough treatments with proofs.

A simple curve $\gamma$ from $0$ to $z$ in $\mathbb{H}$ is a continuous injective map $\gamma:[0,T] \rightarrow \mathbb{H}$ such that $\gamma(0) \in \mathbb{R}$, $\gamma(0,T] \in \mathbb{H}$ and $\gamma(T) = z$.  (Note that we will often abuse notation with respect to $\gamma$ in two ways: first, by writing the image of an interval $I$ under $\gamma$ as $\gamma I$ instead of $\gamma(I)$, and secondly, by using $\gamma$ to denote its trace $\gamma([0,T])$.)  For such a curve and fixed $0 < t \leq T$, $g_t: \mathbb{H}\bs \gamma(0,t] \rightarrow \mathbb{H}$ is the unique conformal map satisfying
\begin{align}\label{Eq:LoewnerNormalize} 
    g_t(z) = z + \frac{\hcap(\gamma[0,t])}{z} + O\Big(\frac{1}{z^2}\Big), \qquad z \rightarrow \infty,
\end{align}
and $g_0(z):= z$.  Here $\hcap(\gamma[0,t])$ is, by definition, the coefficient of $1/z$ in this expansion.  All our curves $\gamma$ will be \emph{parametrized by capacity}, which means that 
\begin{align}\label{Eq:hcapNormalize}
    \hcap(\gamma[0,t]) =2t.
\end{align}
When we refer to the ``capacity time'' or ``Loewner time'' of a curve segment $\gamma$, we will always mean the half-plane capacity of $\gamma$ is $2t$, as above.   

Note that such a ``mapping down'' function $g_A$ as in \eqref{Eq:LoewnerNormalize} can be uniquely defined for a set $A$ whenever $\mathbb{H} \bs A$ is a simply-connected domain, and the \emph{half-plane capacity} $\hcap(A)$ is similarly defined.  We also recall that $\hcap$ is strictly monotone: if $A_1 \subsetneq A_2$ are such sets, then $\hcap(A_1) < \hcap(A_2)$ \cite[\S A.4]{SLEGuide}, and thus it is always possible to parametrize growing curves $\gamma$ by capacity.  

It is known that, for fixed $z$, $g_t(z)$ satisfies \emph{Loewner's differential equation}
\begin{align}\label{Eq:BackgroundLoewnerEq}
    \dot{g}_t(z) = \frac{2}{g_t(z) -\lambda(t)}, \qquad g_0(z) = z,
\end{align}
where $\lambda(t)$ is \emph{driving function} of $\gamma$, which is to say, the continuous, real-valued image of the tip $\gamma(t)$ by the (extension) of $g_t$.  Each point $z \in \mathbb{H}$ has a supremal time $\tau_z$ such the flow of $z$ under \eqref{Eq:BackgroundLoewnerEq} is defined on $[0,\tau_z]$.  We recover $\gamma(0,t]$ by taking the \emph{hull of the Loewner flow} $\{\, z \in \mathbb{H} \;: \; \tau_z \leq t\, \}$.    

There are several useful variations of the Loewner flow.  We write $G_t(z) := g_t(z) -\lambda(t)$ for the centered mapping-down function (where the tip $\gamma(t)$ always maps to zero), and define $f_t := g_t^{-1}$ and $F_t := G_t^{-1}$.  We reverse the direction of the flow via the maps $h_t$ which satisfy
\begin{align}\label{Eq:LoewnerEqUp}
    \dot{h}_t(z) = \frac{-2}{h_t(z) -\xi(t)}, \qquad h_0(z) = z.
\end{align}
If $\xi(t) = \lambda(T-t)$ is the reversal of $\lambda$, then it is easy to see that $h_t^{-1} \circ g_{T-t} \equiv g_T$; that is, $h_t = g_{T-t}\circ g_T^{-1}$ is the conformal map from $\mathbb{H}$ to $\mathbb{H}\bs g_{T-t}(\gamma[T-t,T])$ satisfying $h_t(z) = z - 2t/z + O(1/z^2)$ as $z \rightarrow \infty$.  We frequently use the backwards driver which is shifted to start at zero, $\xi(t) = \lambda(T-t) - \lambda(T)$.

If $\gamma$ is not simple or $\gamma(0,T] \cap \mathbb{R} \neq \emptyset$, then $g_t$ is defined as the map $\mathbb{H}\bs \Fill(\gamma(0,t]) \rightarrow \mathbb{H}$ satisfying the same normalization \eqref{Eq:LoewnerNormalize}, where $\mathbb{H}\bs \Fill(\gamma(0,T])$ is the unbounded connected component of the complement of $\gamma(0,T]$.

The simplest example of a driving function $\lambda(t)$ is for the vertical line segment $\gamma_y = [0,iy]$, $y>0$, where a map $\nH \bs \gamma_y \rightarrow \nH$ which is the identity at infinity to first order is
\begin{align*}
    \sqrt{z^2+y^2} = z+\frac{y^2/2}{z} + O(z^{-3}), \qquad z \rightarrow \infty.
\end{align*}
Hence the capacity parametrization for the imaginary axis is $\gamma(t) = 2i\sqrt{t}$, and the driving function is $\lambda(t) \equiv 0$, the image of $\gamma(t)$.  For further examples of driving functions, see \cite{KNK}, \cite[\S4.1]{Lawler}, \cite{LoewnerCurvature}, or Theorems \ref{Thm:Wang}($\ref{Thm:WangDriver}$) and \ref{Thm:EMW}($\ref{Thm:EMWDriver}$) below.

If we replace $\gamma$ by a scaled copy $r\gamma$ of itself, $r>0$, then to maintain capacity parametrization $r\gamma$ must be parametrized as $t \mapsto r\gamma(t/r^2)$, as one can see from \eqref{Eq:LoewnerNormalize}.  Thus the driving function of $r\gamma$ is 
\begin{align}\label{Eq:ScaledDriver}
    r\lambda(t/r^2).
\end{align}

\subsection{The Loewner energy}\label{Sec:BackgroundLoewnerEnergy}

The Loewner energy of a curve $\gamma$ is the Dirichlet energy of its driving function $\lambda$, which we formally define through the following difference quotient.  Let $\Pi[0,T]$ be the collection of all partitions of $[0,T]$.
\begin{definition}
    The \emph{Loewner energy} $I(\gamma)$ of a curve $\gamma$ on $[0,T]$ (or $[0,T)$ if $T = \infty$) with downwards driving function $\lambda$ is
    \begin{align}\label{Def:LoewnerEnergy}
        I(\gamma) := \sup_{\mathcal{P} \in \Pi[0,T]} \sum_{j=1}^n \frac{(\lambda(t_j)-\lambda(t_{j-1}))^2}{2(t_j-t_{j-1})}.
    \end{align}
\end{definition}
\noindent We may alternatively write $I(\lambda)$ for $I(\gamma)$, or even $I(\xi)$, as it is evident that the supremum in \eqref{Def:LoewnerEnergy} is not changed if we replace $\lambda$ with its reversal $\xi(t) = \lambda(T-t) -\lambda(T)$.  We write $I_T$ if we need to emphasize the interval $[0,T]$ under consideration.

The factor of 2 in the denominator in \eqref{Def:LoewnerEnergy} is a normalization choice by Wang \cite{WangReverse} in order to have the Loewner energy be the good-rate function for SLE$_\kappa$ as $\kappa \rightarrow 0^+$ \cite{PeltWang}.  The supremum in \eqref{Def:LoewnerEnergy} is because the sum is monotonic in the partition, and from analysis we have that a driver $\lambda$ with $I_T(\lambda)<\infty$ belongs to the \emph{Dirichlet space} on $[0,T]$, which is to say, $\lambda$ is absolutely continuous, $\dot{\lambda}\in L^2([0,T])$ and
\begin{align}\label{Eq:LoewnerEnergy}
    I_T(\lambda) = \frac{1}{2}\int_0^T \dot{\lambda}(t)^2dt.
\end{align}
See, for instance, \cite[\S 1.4]{Peres}.  By absolute continuity, $I(\lambda) = 0$ if and only if $\lambda \equiv 0$, and thus we may also view the Loewner energy as a measurement of the deviation of the curve from a hyperbolic geodesic (i.e., the curve generated by the zero driver).  
The Loewner energy enjoys a number of properties which will be of use to us, including:
\begin{enumerate}[$(i)$]
    \item Conformal and anti-conformal invariance: if $r >0$ and $x \in \mathbb{R}$, then $I(r\gamma+x) = I(\gamma)$, as one readily sees from \eqref{Eq:ScaledDriver}.  If $\gamma$ is a curve from 0 to $\infty$, then its image $-1/\gamma$ under the automorphism $z \mapsto -1/z$ of $\mathbb{H}$ also has the same energy, $I(\gamma) = I(-1/\gamma)$ \cite{WangReverse}.  That is, $I$ is invariant under $PSL(2,\mathbb{R})$, the automorphism group of $\mathbb{H}$.
    
    We also have $I(\gamma) = I(-\overline{\gamma})$, where $-\overline{\gamma}$ is the reflection of $\gamma$ across the imaginary axis, since if $\gamma$ is driven by $\lambda$, then $-\overline{\gamma}$ is driven by $-\lambda$.  While trivial, this observation is still useful, in that by symmetry it allows us to only consider minimizers with argument tip $\theta \in (0,\pi/2]$, for instance.
    
    \item $I(\cdot)$ is lower semi-continuous with respect to the sup norm on drivers: If $\lambda_n \rightarrow \lambda$ uniformly on $[0,T]$, then
\begin{align}\label{Ineq:LoewnerEnergyLSC}
    \liminf_{n \rightarrow \infty} I_T(\lambda_n) \geq I_T(\lambda),
\end{align}
as follows from the difference quotient expression \eqref{Def:LoewnerEnergy} (see \cite[\S 2.2]{WangReverse}).

    \item\label{LE:Compactness} The collection of drivers on $[0,T]$ with energy bounded by $C<\infty$ is compact.  Indeed, for any such driver $\lambda$,
\begin{align}\label{Eq:LoewnerHolder}
    |\lambda(t_2) - \lambda(t_1)| \leq \int_{t_1}^{t_2} |\dot{\lambda}(t)|dt \leq \sqrt{2C} \sqrt{t_2-t_1}
\end{align}
by the Cauchy-Schwarz inequality.  In particular, the family is bounded and equicontinuous on $[0,T]$, and so by Arzela-Ascoli is precompact in the uniform norm on $[0,T]$ (recall another name for the embedding $W^{1,2}([0,T]) \hookrightarrow C^{1/2}([0,T])$ is Morrey's inequality \cite[\S4.5.3]{FineProp}). Lower semi-continuity \eqref{Ineq:LoewnerEnergyLSC} then yields compactness.

    \item It follows that finite-energy drivers $\lambda$ generate curves $\gamma$ via the Loewner equation that are quasi-arcs that do not meet $\mathbb{R}$ tangentially.  Indeed, from \eqref{Eq:LoewnerHolder} we see that the $1/2$-\Hol norm of $\lambda$ is locally small and so by \cite[Proof of Thm. 4.1]{LMR}, $\lambda$ generates such quasi-arcs on small scales, and thus also on finite intervals.  Wang then extended this argument to include infinite time intervals \cite[Prop. 2.1]{WangReverse}.  

    \item\label{Lemma:CurvesCompact} One can upgrade the compactness of drivers of bounded energy to compactness of \emph{curves}.  That is, for any $M<\infty$ and $T<\infty$, the collection
    \begin{align}\label{Set:MBoundedEnergy}
        \{ \gamma:[0,T] \rightarrow \mathbb{H}\cup\{0\} \, : \, \gamma(0)=0, \gamma((0,T]) \subset \mathbb{H}, I_T(\gamma)\leq M \}
    \end{align}
    of capacity-parametrized curves, equipped with the sup norm, is compact in $C([0,T])$.  See \cite[Lemma 3.4]{Guskov} or consider the following argument: for a given sequence $\{\gamma_n\}$, the driving functions have a convergence subsequence in $C([0,T])$ by property ($\ref{LE:Compactness}$).  The corresponding subsequence of curves is also equicontinuous and bounded \cite[Theorem 2$(iii)$]{FrizShekhar}, and so has a convergence subsequence.  Since both the drivers and the curves converge, the limiting curve must correspond to the limiting driver \cite[Lemma 4.2]{LMR}.  Thus the collection \eqref{Set:MBoundedEnergy} is sequentially compact, as claimed.
\end{enumerate}

Much more could be said about the Loewner energy, but we offer the following, undoubtedly incomplete, concluding remarks.  Rohde and Wang \cite{RohdeWang} generalized the energy to loops via a limiting procedure, and this appears to be the most natural setting.  Indeed, Wang subsequently showed \cite{WangEquiv} that the loop Loewner energy can be explicitly computed via $L^2$-norms of the pre-Schwarzian derivatives of the Riemann maps to the two sides of the loop.  That description led to Bishop's ``square-summable curvature'' characterizations \cite{Bishop}, and also to connections to \Teich theory and geometry, as Takhtajan and Teo had earlier shown that Wang's conformal-map formula is (a multiple of) the K\"{a}hler potential for the Weil-Petersson metric on universal \Teich space, and that Weil-Petersson quasicircles are precisely those for which the expression is finite \cite{TT}.  That is, finite-energy loops are precisely the closure of smooth loops in the Weil-Petersson metric.  See \cite{WangEquiv} for details.  
Initial interest in studying the energy was stochastic in nature; Wang showed it is the large-deviations good-rate function for SLE$_\kappa$ as $\kappa \rightarrow 0^+$ \cite{WangReverse} (see also \cite{PeltWang} for an extension to the SLE multichord setting), and Tran and Yuan \cite{Yizheng} also showed the closure of finite-energy curves is the topological support of SLE$_\kappa$, for any $\kappa \geq 0$.  In brief, the Loewner energy occupies a fascinating crossroads between probability theory, univalent function theory, \Teich theory, geometric measure theory, and hyperbolic geometry.\footnote{See also \cite{Krusell, PeltWang} for generalizations to ensembles of multiple curves, and \cite{Krusell} for a generalization which, informally, is to single curves $\gamma$ ``conditioned to pass through a given point'' $z_0 \in \mathbb{H}$.}


\subsection{SLE$_\kappa(\rho_1,\ldots, \rho_n)$ and its reversal.}\label{Sec:BackgroundSLE}
We recall that (chordal, downwards) \emph{SLE with forcing SLE$_\kappa(\rho_1,\ldots, \rho_n)$ starting from $(\lambda_0, U^1_0, \ldots, U^n_0)$} is the Loewner flow generated by driver $\lambda_t$ whose motion is defined by Brownian motion and interactions from particles $U^1_t, \ldots, U^n_t \in \overline{\mathbb{H}}$, the closure of $\mathbb{H}$, via the system of SDE's
\begin{align}\label{Eq:SLErhoSDE}
    d\lambda_t = \sqrt{\kappa}dB_t + \sum_{j=1}^n \Real \frac{\rho_j}{\lambda_t - U^j_t}dt, \qquad dU^j_t = \frac{2}{U^j_t - \lambda_t} dt, \qquad j=1,\ldots, n,
\end{align}
with initial conditions given by the $(n+1)$-tuple $(\lambda_0, U^1_0,\ldots, U^n_0)$.  Here each $\rho_j \in \mathbb{R}$, $B_t$ is a standard Brownian motion, and the process is defined until the first time $\tau$ such that, for some $j$, 
\begin{align*}
    \inf_{0 \leq t <\tau} |\lambda_t - U^j_t| = 0.
\end{align*}
So in SLE with forcing the dynamics of $\lambda_t$ is influenced by the location of the \emph{force points} $U^1_t, \ldots, U^n_t$, but all other points evolve according to the normal Loewner flow generated by $\lambda_t$.  The connection to energy-minimizers is when $\kappa =0$, in which case \eqref{Eq:SLErhoSDE} becomes a deterministic system of ODE's.

For our purposes it will be convenient to occasionally reverse the direction of the flow, and we define \emph{upwards SLE$_\kappa(\rho_1,\ldots, \rho_n)$ starting from $(\xi_0, V^1_0, \ldots, V^n_0)$} to be the process given by flow \eqref{Eq:LoewnerEqUp} where $\xi_t$ and the force points $V^1_t, \ldots, V^n_t$ satisfy
\begin{align*}
    d\xi_t = \sqrt{\kappa}dB_t + \sum_{j=1}^n \Real \frac{-\rho_j}{\xi_t - V^j_t}dt, \qquad dV^j_t = \frac{-2}{V^j_t - \xi_t} dt, \qquad j=1,\ldots, n,
\end{align*}
with initial conditions $(\xi_0, V^1_0, \ldots, V^n_0) \in \mathbb{R} \times \overline{\mathbb{H}}^{n}$.  

In this paper we only consider deterministic $\kappa=0$ processes, and in this case we can exchange the downwards and  upwards points of view when convenient by reversing the direction of the flow and keeping track of the location of the forcing points.


\subsection{Quasiconformal mappings}

We recall several properties of two-dimensional quasiconformal mappings.  For proofs and more details, see \cite{Astala}, \cite{LehtoUFT} or \cite{LehtoQC}. Let $\Omega_1, \Omega_2$ be domains in the Riemann sphere $\chat$.  A \emph{$K$-quasiconformal map} $f: \Omega_1 \rightarrow \Omega_2$ is an orientation-preserving homeomorphism which is absolutely continuous on a.e. line parallel to the axes and differentiable at Lebesgue-almost every $z \in \Omega_1$, and whose ``complex directional derivatives''
\begin{align*}
	\partial_\alpha f (z) := \lim_{r \rightarrow 0} \frac{f(z + re^{i \alpha})-f(z)}{re^{i\alpha}}
\end{align*}
satisfy
\begin{align}\label{Eq:QCDirDerivBounds}
	\max_\alpha |\partial_\alpha f(z)| \leq K \min_\alpha |\partial_\alpha f(z)|
\end{align}
at points $z$ of differentiability.  More succinctly, $f$ is $K$-quasiconformal if $f \in W^{1,2}_{\text{loc}}(\Omega_1)$ with derivatives satisfying \eqref{Eq:QCDirDerivBounds} almost everywhere.  Conformal maps are $K$-quasiconformal with $K=1$.  

Our interest in quasiconformal maps stems from the fact that finite-energy arcs $\gamma \subset \mathbb{H} \cup \{0\}$ are images $f([0,i])$ of the imaginary axis segment $[0,i]$ under a $K=K(I(\gamma))$-quasiconformal map $f:\mathbb{H} \rightarrow \mathbb{H}$ which fixes $0$ and $\infty$ \cite[Prop. 2.1]{WangReverse}. Furthermore, families of quasiconformal maps have nice compactness properties, as expressed in the following two propositions.   We will use these in tandem with the lower semi-continuity of energy \eqref{Ineq:LoewnerEnergyLSC} to obtain Loewner-energy minimizers.

\begin{proposition}\cite[Thm. 2.1]{LehtoUFT}
    A family $F$ of $K$-quasiconformal mappings of $\Omega_1 \subset \chat$ is normal in the spherical metric if there exists  three distinct points $z_1,z_2,z_3 \in \Omega_1$ and $\epsilon >0$ such that for any $f \in F$,  $d_{\hat{\nC}}(f(z_j),f(z_k)) > \epsilon$, whenever $j \neq k$, $j,k=1,2,3$.
\end{proposition}
\noindent Furthermore, the subsequential locally-uniform limits are also either $K$-quasiconformal or constant, a generalization of the Hurwitz theorem for conformal mapping.
\begin{proposition}\cite[Thm. 2.2, 2.3]{LehtoUFT}
	Let $f_n: \Omega_1 \rightarrow \Omega_2$ be a sequence of $K$-quasiconformal maps which converges locally uniformly to $f$.  Then $f$ is either $K$-quasiconformal or $f$ maps all of $\Omega_1$ to a single boundary point of $\Omega_2$.
\end{proposition}  

\section{The energy minimizers to a point (EMP) curves}\label{Sec:Wang}



We begin by considering the minimization question ($\ref{Q:HMinimizer}$): what is the infimal energy needed for a curve to go from 0 to $r_0e^{i\theta_0}$ in $\mathbb{H}$, and what is the nature of $\gamma$ which achieve the minimum, if this is possible?  Our deterministic answer is found below in Theorem \ref{Thm:Wang}, but is prefaced by the following two lemmas.  The latter respectively state that minimizers exist, and that, under the centered downwards Loewner flow, the tip of a finite-energy curve tends towards the imaginary axis.
\begin{lemma}\label{Lemma:AngleEnergyMinimizerExists}
    For any $z \in \mathbb{H}$, there exists a simple curve $\gamma_z$ from 0 to $z$ in $\mathbb{H}$ such that 
    \begin{align}\label{Eq:AngleEnergyInf}
        I(\gamma_z) = \inf_{\gamma \in \Gamma_z} I(\gamma),
    \end{align}
    where $\Gamma_z$ is the collection of all curves in $\mathbb{H}$ from 0 to $z$.
\end{lemma}
The argument is a standard application of the compactness of quasiconformal mappings and the lower semi-continuity of energy \eqref{Ineq:LoewnerEnergyLSC}. We will show below in Theorem \ref{Thm:Wang}(i) that the minimizer is unique.
\begin{proof}
    By scale invariance of energy we may assume $z = e^{i\theta}$.  In Lemma \ref{Lemma:CircularArc} below we will see that the orthogonal circular arc segment from 0 to $e^{i\theta}$ has energy \begin{align}\label{Eq:CAEnergy}
        -9\log(\sin(\theta)),
    \end{align}
    and so the infimum in \eqref{Eq:AngleEnergyInf} is not $+\infty$ and it suffices to consider those Jordan arcs in $\Gamma_z$ with energy bounded by \eqref{Eq:CAEnergy}.  By \cite[Prop. 2.1]{WangReverse}, each such $\gamma$ is a $K$-quasislit half-plane for some fixed $K = K(\theta)$, and so there exists a $K$-quasiconformal self-map $q$ of $\mathbb{H}$ fixing $0$ and $\infty$ with $q([0,i]) = \gamma$.  
    
    Let $\{\gamma_n\}$ be a sequence such that $I(\gamma_n)$ tends to the infimum in \eqref{Eq:AngleEnergyInf}, and let $\{q_n\}$ be corresponding $K$-quasiconformal maps.  The family $\{q_n\}$ is normal in the spherical metric, and any limiting function is either a constant in $\partial_{\chat}\nH$ or a $K$-quasiconformal self-map of $\mathbb{H}$ \cite[\S2.2f]{LehtoUFT}.  The former cannot happen because all $q_n$ map $i$ to $e^{i\theta}$, and hence, by moving to a subsequence, which we relabel as $q_n$ again, we have a locally uniform limit $q_n \rightarrow q$ in $\mathbb{H}$ and a limiting curve $\gamma:=q([0,i])$.  In fact, the convergence is uniform on $[0,i]$ by Schwarz reflection of the $q_n$ and $q$.  We argue that the convergence is also uniform in the capacity parametrizations of $\gamma_n$ and $\gamma$.
    
    All the curves $\gamma, \gamma_n$ are uniformly bounded in $\mathbb{H}$, and so they are also bounded in half-plane capacity.  By extending each $\tilde{\gamma} \in \{\gamma, \gamma_n\}$ by an appropriate-length segment of the hyperbolic geodesic from $e^{i \theta}$ to $\infty$ in $\nH \backslash \tilde{\gamma}$, we can consider all the $\tilde{\gamma}$ to be defined on the same interval $[0,T]$ of capacity time. By the proof of Theorem 4.1 in \cite{LMR}, the modulus of continuity of a $K$-quasiarc in its capacity parametrization depends only upon $K$, and hence $\{\gamma, \gamma_n\}$ is a bounded and equicontinuous family.\footnote{We have changed the curves with the hyperbolic geodesic segments, but this does not change their Loewner energy, and so the augmented curves are still all $K$-quasiarcs.}  Hence by Arzela-Ascoli we move to a further subsequence, if necessary, and obtain a uniform capacity-parametrization limit $\gamma_n \rightarrow \gamma'$ on $[0,T]$.  But clearly $\gamma'$ must be $\gamma$: since $\mathbb{H}\bs q_n([0,si]) \rightarrow \mathbb{H} \bs q([0,si])$ in the \Cara sense for any $0 \leq s \leq 1$,  $\hcap(q_n([0,si])) \rightarrow \hcap(q([0,si]))$, and it readily follows from the uniform convergence $q_n \rightarrow q$ that the two limits are identical.  
    
    Since $\gamma_n \rightarrow \gamma$ uniformly and all the curves are simple, we also have the uniform convergence $\lambda_n \rightarrow \lambda$ of their associated drivers on $[0,T]$ by \cite[Thm. 1.8]{YuanTop}. The lower semicontinuity of energy then yields 
    \begin{align*}
        I(\gamma) \leq \liminf_{n \rightarrow \infty}I(\gamma_n) = \inf_{\gamma \in \Gamma_z} I(\gamma),
    \end{align*}
    and as $e^{i\theta} \in \gamma$, we have that $\gamma$ is a minimizer.
\end{proof}

\begin{lemma}[``Orthogonal approach'']\label{Lemma:WangEvenApproach}
    Suppose $\lambda$ is the driver for the simple curve $\gamma:[0,T] \rightarrow \nH \cup \{0\}$ with  $I(\gamma) < \infty$.  Then under the centered downward Loewner flow generated by $\lambda$, the image $z_t:=G_t(\gamma(T))$ of the tip of the curve satisfies
    \begin{align}\label{Lim:FiniteEnergyAngle}
        \lim_{t \rightarrow T^-} \arg(z_t) = \frac{\pi}{2}.
    \end{align}
\end{lemma}
For example, note the increase of the argument of $g_t(\gamma(T)) - \lambda_t$ towards $\pi/2$ from the left to the middle image in Figure \ref{Fig:LoewnerFlow}.  Of course, if $\gamma$ in its entirety has finite energy, then so does $\gamma([0,\tau])$ for any $\tau <T$, and so \eqref{Lim:FiniteEnergyAngle} also holds as $t \rightarrow \tau^-$ for any point $\gamma(\tau)$ on $\gamma$. 

This lemma is very similar to \cite[Lemma B]{PeltWang}, but the difference here is that we do not assume the minimal energy formula for curves through a point $e^{i\theta} \in \mathbb{H}$.  Indeed, we will use Lemma \ref{Lemma:WangEvenApproach} in our deterministic proof of this formula.  While the lemma could be proven through direct analysis of the Loewner equation, it is also a simple consequence of Lemma \ref{Lemma:AngleEnergyMinimizerExists}.\footnote{The author is grateful to Don Marshall for suggesting the following elementary proof, which replaced an earlier more complicated argument.}
\begin{proof}
  The function
  \begin{align}\label{Def:MinimalEnergy}
      m(\theta) := \min_{\gamma \in \Gamma_\theta} I(\gamma),
  \end{align}
  where $\Gamma_\theta$ is the collection of all curves from 0 to $e^{i \theta}$ in $\mathbb{H}$, is well defined by the previous lemma.  We first show that $m$ is non-increasing on $(0,\pi/2]$  and non-decreasing on $[\pi/2, \pi)$; by symmetry it suffices to consider the former case.  Indeed, fix $0 < \theta < \pi/2$ and a curve $\gamma_\theta \in \Gamma_\theta$ such that $I(\gamma_\theta) = m(\theta)$.  Now, beginning with $\gamma_\theta$ in $\mathbb{H}$, apply the \emph{upwards} Loewner flow described by \eqref{Eq:LoewnerEqUp} using the zero driver $\xi_t \equiv 0$.  Under this flow, the argument of the image of the tip of $h_t(\gamma_\theta)$ increases to $\frac{\pi}{2}^-$ as $t \rightarrow \infty$, as one can see from the explicit conformal map $h_t(z) = \sqrt{z^2-4t}$.  Since the zero driver does not add energy, the constructed curve $h_t(\gamma_\theta) \cup L_t$, where $L_t$ is the line segment from $0$ to $h_t(0) = 2i\sqrt{t}$, has identical energy to $\gamma_\theta$, thus showing $m(\theta') \leq m(\theta)$ for $\theta < \theta' < \pi/2$ as claimed.
  
  We also see that $m(\theta) > 0$ if $\theta \neq \pi/2$, for if $m(\theta)$ vanished, then any minimizer $\gamma_{\theta}$ would be driven by the zero driver, which corresponds to the imaginary axis, not a curve through $e^{i \theta}$.   Thus $m$ is non-increasing on $(0,\pi/2]$.
  
  Now, if there are times $t_n \rightarrow T^-$ such that $|\arg(z_{t_n}) - \pi/2| > \epsilon$, then $\lambda$ must expel at least some $m(\pi/2-\epsilon) = m(\pi/2+\epsilon)>0$ amount of energy on each $[t_n,T]$ by the above, contradicting $\lim_{t \rightarrow T^-} \int_t^T \dot{\lambda}^2(s)ds =0$. 
\end{proof}

Our main results about the EMP curves are in the following theorem.  We recall that the statement of $(i)$ is from Wang and is included because we provide an alternative, deterministic proof.

\begin{theorem}\label{Thm:Wang}
    Let $\theta \in (0,\pi)$.
    \begin{enumerate}[$(i)$]
        \item \label{Thm:WangEnergy}\cite[Proposition 3.1]{WangReverse} There exists a unique $\gamma_\theta$ from 0 to $e^{i\theta}$ in $\mathbb{H}$ which minimizes the Loewner energy among all such curves.  Furthermore, $\gamma_\theta$ is given by downwards SLE$_0(-8)$ starting from $(\lambda(0), U(0)) = (0,e^{i\theta})$, and satisfies
        \begin{align}\label{Eq:WangEnergy}
            I(\gamma_\theta) = -8\log(\sin(\theta)).
        \end{align}
        The driving function $\lambda_\theta$ is $C^\infty([0,\tau_\theta))$ and monotonic.
        \item\label{Thm:WangDriver} (Driver and welding) For $0<\theta <\pi/2$, the upwards-flow driver $\xi_\theta(t) = \lambda_\theta(\tau-t) - \lambda_\theta(\tau)$ is explicitly
        \begin{align}\label{Eq:WangDriver}
        \xi_\theta(t) &= - \frac{4\sqrt{2}}{\sqrt{3}}\left( \sqrt[3]{\sqrt{\frac{\sin^6(\theta)}{36\cos^2(\theta)}+t^2}+t} -\sqrt[3]{\sqrt{\frac{\sin^6(\theta)}{36\cos^2(\theta)}+t^2}-t} \right)^{3/2}
    \end{align}
    for $0 \leq t \leq \frac{1}{6}(1 - \frac{1}{2}\cos(2\theta)) =:\tau = \tau_\theta$.  In particular, $\xi_\theta(0)=0$ and $\xi_\theta(\tau) = -\frac{4}{3}\cos(\theta)$.  For $\pi/2 < \theta < \pi$, $\xi_\theta = -\xi_{\pi-\theta}$.  
    
    For any $0 < \theta < \pi$, the conformal welding $\varphi_\theta$ corresponding to the Loewner-flow normalization is
    \begin{align}\label{Eq:WangWeld}
        \varphi_\theta(x) = \frac{-x}{\sqrt{1 + \pi \frac{\cos(\theta)}{\sin^3(\theta)}x^2}},
    \end{align}
    where $\varphi_\theta: [x_\theta,0] \rightarrow [0,y_\theta]$, with \begin{align}\label{Eq:WangWeldInterval}
        x_\theta = \frac{-\sqrt{\sin^3(\theta)}}{\sqrt{\sin(\theta) - \theta \cos(\theta)}} \qquad \text{ and } \qquad y_\theta = \frac{\sqrt{\sin^3(\theta)}}{\sqrt{\sin(\theta) + (\pi-\theta)\cos(\theta)}}.
    \end{align}
        \item \label{Thm:WangUniversal} (Universality of driver and welding) For each fixed $\theta \neq \pi/2$, the driver \eqref{Eq:WangDriver} and welding \eqref{Eq:WangWeld} are universal, in the sense that they extend (using the identical formulas) to all $t \geq 0$ and $x \leq 0$, respectively, and these extensions generate every EMP curve to $e^{i\alpha}$, $\alpha \neq \pi/2$, up to scaling, translation and reflection in the imaginary axis.  
        
        More precisely, for every $t >0$, $\xi_\theta |_{[0,t]}$ generates the scaled and translated EMP curve $r_t\gamma_{\alpha(t)} + \xi_\theta(t)$, where the range of $t \mapsto \alpha(t)$ is the connected subinterval of $(0,\pi)\bs \{\pi/2\}$ containing $\theta$.  Explicitly, the curve generated by $\xi_\theta$ on the interval
        \begin{align*}
            [0,t_\alpha] := \Big[0, \frac{\sin^3(\theta)\cos(\alpha)}{6\cos(\theta)\sin^3(\alpha)}\Big( 1 - \frac{1}{2}\cos(2\alpha)\Big) \Big]
        \end{align*}
        is (a translation of) the scaled EMP curve $r \gamma_\alpha$, where 
        \begin{align}\label{Eq:WangUniversalr}
            r = \sqrt{\frac{\sin^3(\theta)}{\cos(\theta)} \cdot \frac{\cos(\alpha)}{\sin^3(\alpha)}}.
        \end{align}
        
        Similarly, for every $u<0$, $\varphi_\theta|_{[u,0]}$ is the welding function of the scaled EMP curve $r_u \gamma_{\alpha(u)}$, where the range of $u \mapsto \alpha(u)$ is the connected subinterval of $(0,\pi)\bs \{\pi/2\}$ containing $\theta$.  Explicitly, the curve welded by $\varphi_\theta$ on the interval 
        \begin{align*}
            [u_\alpha,0] := \Big[ -\sqrt{\frac{\sin^3(\theta)}{\cos(\theta)} \cdot \frac{\cos(\alpha)}{\sin(\alpha)-\alpha \cos(\alpha)}},0 \Big]
        \end{align*}
        by the upwards centered Loewner map is the scaled EMP curve $r \gamma_\alpha$, with $r$ again given by \eqref{Eq:WangUniversalr}.

        
        \item\label{Thm:Wang0+} (Limiting curve) As $\theta \rightarrow 0^+$, $\gamma_\theta$ converges pointwise to $\gamma_0$ for $0 \leq t \leq \tau_0 = 1/12$, where $\gamma_0$ is the curve with backwards driver
        \begin{align*}
            \xi_0(t) = -\frac{8}{\sqrt{3}}\sqrt{t}
        \end{align*}
        on $[0,\tau_0]$. Furthermore, the Loewner mapping-down functions $g_\theta:\mathbb{H}\bs \gamma_\theta \rightarrow \mathbb{H}$ converge locally uniformly to $g_0: \mathbb{H} \bs \Fill(\gamma_0) \rightarrow \mathbb{H}$ on $\mathbb{H} \bs \Fill(\gamma_0)$.  Also, $\gamma_0$ and its reflection $1/\overline{\gamma}_0 =:\gamma_0^*$ over $\partial \mathbb{D}$ form a boundary geodesic pair in $(\mathbb{H}; 0, \infty, 1)$ and are a subset of the algebraic variety
        \begin{align}\label{Eq:Variety}
            (4-3x)y^2 = 3x(x-1)^2.
        \end{align}
        In particular, $\gamma_0(t)$ meets $\mathbb{R}$ at $x=1$ when $t=1/12$, and the angle between $\mathbb{R}$ and $\gamma$ with respect to the bounded component of $\mathbb{H}\bs \gamma_0$ is $\pi/3$.
    \end{enumerate}
\end{theorem}
\noindent See Figure \ref{Fig:Gamma0} for an illustration of $\gamma_0$ and the geodesic pairing property of part $(\ref{Thm:Wang0+})$ of the theorem.
\begin{remark}\label{Remark:RemainingWang} We preface the proof with several comments.
\begin{figure}
    \centering
    \includegraphics[scale=0.1]{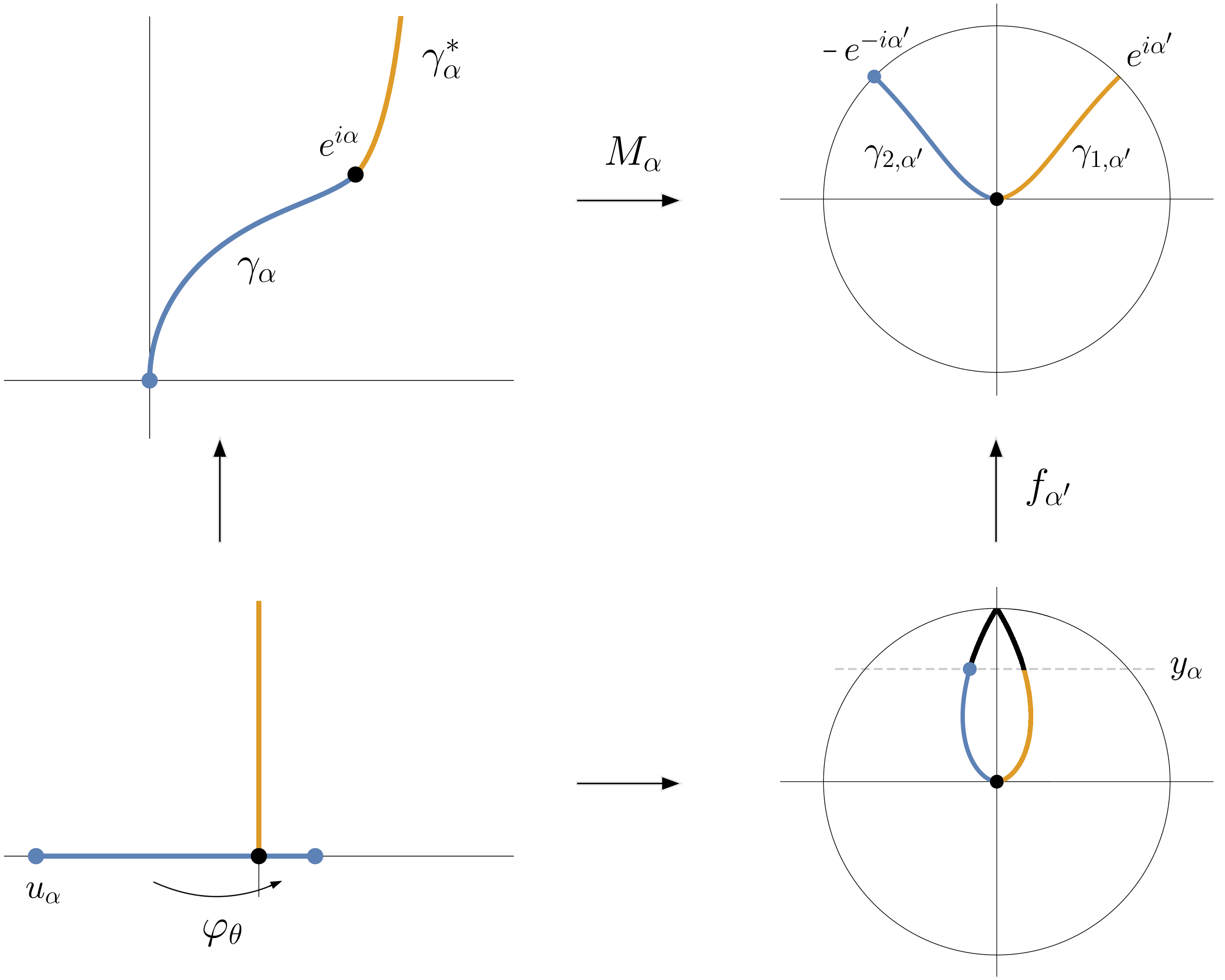}
    \caption{\small The function $\varphi_\theta$ in \eqref{Eq:WangWeld} is the welding for the blue half of the universal curve $\gamma$ of \cite[Remark 2.4]{MRW} in the sense of providing the welding for the bottom arrow of this diagram, for any $\alpha \in (0,\pi/2)$.  Here $\alpha' = \frac{\pi}{2} - \alpha$.}
    \label{Fig:UniversalCurveWang}
\end{figure}
\begin{enumerate}[$(a)$]
    \item\label{Remark:RestOfWang} When we map down an initial portion $\gamma_\theta([0,t])$ of a EMP curve, the remainder $G_t(\gamma_\theta([t,\tau]))$ is the EMP curve for the angle of $G_t(e^{i\theta})$, since if not, we could replace $G_t(\gamma_\theta([t,\tau]))$ with the minimizer and lower the energy.  
    
    In contrast, the initial portion $\gamma_\theta([0,t])$ mapped down is \emph{not} also a EMP curve: we see below in \eqref{Eq:WangSystem} that if $\theta \neq 0$, $\dot{\lambda}_\theta \neq 0$, and so Theorem \ref{Lemma:Inf98}$(\ref{Thm:LocalCompareWang})$ below applies to the initial segments of the EMP curves themselves.  Hence the symmetry of the Wang family is with respect to the ``top,'' or the portion remaining after mapping down, not with respect to the ``base'' or the portion mapped down, and it is therefore natural to express the driving function \eqref{Eq:WangDriver} in terms of the upwards flow $\xi$ rather than the downwards flow $\lambda$.  That is, the curve $\eta[0,t]$ generated by the upwards flow under $\xi$ on $[0,t]$, $t < \tau$, is always a minimizer.
    
    This is the opposite of what we will see for the EMW family, where what we map down is an EMW curve, but what remains is not.  See Remark \ref{Remark:DualFam}.
    
    \item Note that all quantities inside roots in \eqref{Eq:WangDriver} are non-negative, and so there is no ambiguity about branch cuts.
    
    \item We see in part $(\ref{Thm:Wang0+})$ that the curve $\gamma_0 \cup \gamma_0^*$ obtained in the $\theta \rightarrow 0^+$ limit is algebraic.  It is natural to ask if $\Gamma_\theta := \gamma_\theta \cup \gamma_\theta^*$ is also algebraic for other values of $\theta$.  The conjectural answer is that $\Gamma_\theta$ is algebraic if and only if $\theta \in \{0, \pi/2, \pi\}$.  Indeed, the downwards driver for $\Gamma_\theta$ is $\lambda_\theta = \xi_\theta(\tau_\theta-t) - \xi_\theta(\tau_\theta)$ on $[0,\tau_\theta]$, followed by the constant driver $\lambda_\theta(\tau_\theta)$ thereafter.  The expansion of $\xi_\theta$ at $t=0^+$, corresponding to the intersection point between $\gamma_\theta$ and $\gamma_\theta^*$, is
    \begin{align}\label{Eq:WangDriverExpansion}
        \xi(t) = -\frac{32}{3} \frac{\cos(\theta)}{\sin^3(\theta)} t^{3/2} + O(t^{7/2}),
    \end{align}
    and so the driver has global $C^{3/2}$-regularity.  If $\Gamma_\theta$ were algebraic (and non-singular at $e^{i\theta}$, which is natural), then it would be smooth at $e^{i\theta}$, say at least $C^{2+\epsilon}$.  By the correspondence between curve regularity and driver regularity \cite{RohdeWang}, this would mean the driver would be globally $C^{3/2+\epsilon}$, a contradiction.  
    
    The reason this is not a proof is that the curve-driver regularity result of \cite{RohdeWang} is only proven for curves up to $C^2$.  One expects it to hold for higher regularity as well (see \cite[Comment 4.1]{RohdeWang}), in which case the above argument becomes a proof.  Note that the ``if'' direction of the conjecture is proven by $(\ref{Thm:Wang0+})$ and the fact that $\gamma_{\pi/2}$ is a line segment orthogonal to $\mathbb{R}$. 

    \item\label{Remark:WangUniversalCurve} In \cite[Remark 2.4]{MRW}, a universal curve $\gamma$ in $\mathbb{D}$ was noted for the EMP family, which is similar in spirit to the universality properties of the welding and driver in part $(\ref{Thm:WangUniversal})$ of the theorem.  The connection is that our $\varphi_\theta$ in \eqref{Eq:WangWeld} welds half of \cite{MRW}'s universal curve $\gamma$ in the sense pictured in Figure \ref{Fig:UniversalCurveWang}. Readers interested in further details may consider the following more thorough description.
    
    Let $\gamma_\alpha^*$ be the hyperbolic geodesic in $\mathbb{H} \bs \gamma_\alpha$ from $e^{i\alpha}$ to $\infty$ and first note that $\gamma_\alpha \cup \gamma_\alpha^*$ is a geodesic pair in $(\mathbb{H}; 0, \infty, e^{i\alpha})$.  Indeed, applying $z \mapsto -1/\bar{z}$ to $\gamma_\alpha \cup \gamma_\alpha^*$ gives a curve from 0 through $e^{i \alpha}$ with the same minimizing energy by invariance of energy under reversal and reflection \cite{WangReverse}.  Hence by uniqueness of minimizers, $-1/\gamma_\alpha^* = \gamma_\alpha$, which shows $\gamma_\alpha$ is the hyperbolic geodesic from 0 to $e^{i \alpha}$ in $\mathbb{H} \bs \gamma_\alpha^*$.
    
    A change of coordinates puts us in the setting of \cite[Figure 1]{MRW}: the M\"{o}bius map $M_\alpha:\mathbb{H} \rightarrow \mathbb{D}$ given by 
    \begin{align}\label{Eq:GeodesicMobius}
        M_\alpha(z) = ie^{-i\alpha}\frac{z-e^{i\alpha}}{z-e^{-i\alpha}}
    \end{align}
    takes $\gamma_\alpha \cup \gamma_\alpha^*$ to a geodesic pair $\gamma_{1,\alpha'} \cup \gamma_{2,\alpha'}$ in $(\mathbb{D}; e^{i\alpha'}, -e^{-i \alpha'}, 0)$, where $\alpha' = \frac{\pi}{2}-\alpha$.  The universal curve $\gamma$ is a Jordan curve $\gamma$ in $\mathbb{D} \cup \{i\}$ that has the property that for each $\alpha'$, there exists $y_{\alpha'}$ such that the conformal map $f_{\alpha'} : \mathbb{D} \bs (\gamma \cap \{\Imag(z)\geq y_{\alpha'}\}) \rightarrow \mathbb{D}$ with $f_{\alpha'}(0) = 0$, $f_{\alpha'}'(0)>0$ satisfies $f_{\alpha'}(\gamma \cap \{ \Imag(z) < y_{\alpha'} \}) = \gamma_{1,\alpha'} \cup \gamma_{2,\alpha'}$ (in the notation of \cite{MRW}, $\gamma = G_{\pi/2}^{-1}(I)$ for $I = (\pi/2, \infty)$ and $G_\theta$ defined there).
    
    Now, $\varphi_\theta|_{[u_\alpha,0]}$  welds $r_u \gamma_\alpha$, the minimizer $\gamma_\alpha$ scaled by some $r_u >0$, whose image under $f_{\alpha'}^{-1} \circ M_\alpha \circ z/r_u$ is $\gamma \cap \{\Imag(z) < y_{\alpha'}\} \cap \{\Real(z) \leq 0\}$, and in this sense we say $\varphi$ is the welding for $\gamma$. 
\end{enumerate}
\end{remark}
\begin{proof}
    Minimizing drivers exist by Lemma \ref{Lemma:AngleEnergyMinimizerExists}, and we proceed to show that any such $\lambda$ satisfies a certain differential equation at all points of differentiability, which will give uniqueness and the formula \eqref{Eq:WangEnergy} for $m(\theta)$, with $m$ as in \eqref{Def:MinimalEnergy}.  
    
    Indeed, let $\lambda$ be a minimizer, and let $z(t) = x(t) + iy(t) := G_t(e^{i\theta}) = g_t(e^{i\theta}) - \lambda(t)$ be the image of $e^{i\theta}$ under the centered downwards flow generated by $\lambda$.  Note that by the Loewner equation \eqref{Eq:BackgroundLoewnerEq}, $G_t$ satisfies
    \begin{align}\label{Eq:LoewnerEqCentered}
        \dot{G}_t(z) = \frac{2}{G_t(z)} - \dot{\lambda}(t)
    \end{align}
    at points of differentiability of $\lambda$. Since $I(\lambda)<\infty$, $\lambda$ is absolutely continuous and so differentiable for a.e. $t$. We claim that for a.e. $t$, 
    \begin{align}\label{Eq:WangLambdaPrimeFormula}
        \dot{\lambda}(t) = \frac{8x(t)}{|z(t)|^2}.
    \end{align}
    By scale invariance the energy depends only on the angle, and by Lemma \ref{Lemma:WangEvenApproach} the angle $\theta(t):= \Arg(z(t))$ must tend towards $\pi/2$ if $\lambda$ minimizes energy.  And indeed, it must move strictly monotonically: if the angle ever decreases and then returns to the same value, energy is wasted, while if the angle is constant over some interval, then $\lambda$ cannot be constant and we have also wasted energy.  Hence $z(t)$ should traverse through the angles as efficiently as possible; the change in angle to the change in energy, $d\theta/dm$, must be optimal.  That is, $d\theta/dm$ must be maximized when $\theta <\pi/2$ and minimized when $\theta > \pi/2$.
    
    Suppose first that $\lambda$ is right-differentiable at $t=0$ and that $t=0$ is a Lebesgue point for $\dot{\lambda}^2$.  Then the energy expelled on a small interval $[0,\Delta t]$ is $\frac{1}{2}\dot{\lambda}^2(0)\Delta t + o(\Delta t)$.  Furthermore, $\theta$ is right-differentiable at $t=0$, and
    \begin{align*}
        \Delta \theta = \dot{\theta}(0)\Delta t + o(\Delta t) = \Imag\Big(\frac{\dot{z}(0)}{z(0)}\Big)\Delta t + o(\Delta t) = (-4x(0)y(0)+\dot{\lambda}(0)y(0))\Delta t + o(\Delta t)
    \end{align*}
    by \eqref{Eq:LoewnerEqCentered} and the fact that $|z(0)|=1$.  We thus have
    \begin{align*}
        \frac{\Delta \theta}{\Delta I} = \frac{(-4xy+\dot{\lambda}y)\Delta t + o(\Delta t)}{\frac{1}{2}\dot{\lambda}^2\Delta t + o(\Delta t)} \rightarrow \frac{-8xy}{\dot{\lambda}^2} + \frac{2y}{\dot{\lambda}}
    \end{align*}
    as $\Delta t \rightarrow 0$, where $x,y$ and $\dot{\lambda}$ are evaluated at $t=0$, and $\dot{\lambda} =\dot{\lambda}(0)$ is the right derivative of $\lambda$.  This expression is optimized with respect to $\dot{\lambda}$ when $\dot{\lambda}=8x$, which yields a local max when $x>0$ and a local min when $x<0$, as needed.  Thus any minimizer for which $\dot{\lambda}(0)$ exists and where $t=0$ is a Lebesgue point of $\dot{\lambda}^2$ must satisfy  \eqref{Eq:WangLambdaPrimeFormula} at $t=0$ (recall $|z(0)|=1$).
     
     More generally, let $t_0$ be a point of differentiability of $\lambda$ and a Lebesgue point of $\dot{\lambda}^2$.  Note that the remaining curve $\tilde{\gamma}:=G_{t_0}(\gamma([t_0,\tau]))$ must be an energy minimizer through $z(t_0)$, as discussed above in Remark \ref{Remark:RemainingWang}. Thus $\tilde{\gamma}/|z(t_0)|$ is a minimizer as in the previous paragraph, and so its driver $\tilde{\lambda}$ has initial right derivative $8\tilde{x}(0) = 8x(t_0)/|z(t_0)|$.  Recalling the scaling relation \eqref{Eq:ScaledDriver}, we therefore have 
     \begin{align*}
         |z(t_0)|\dot{\lambda}(t_0) = \dot{\tilde{\lambda}}(0) = \frac{8x(t_0)}{|z(t_0)|},
     \end{align*}
     as in \eqref{Eq:WangLambdaPrimeFormula}.  Since $\lambda$ is differentiable at a.e. $t$ and a.e. $t$ is a Lebesgue point of the integrable function $\dot{\lambda}^2$, \eqref{Eq:WangLambdaPrimeFormula} holds as claimed.  
     
    
    By \eqref{Eq:LoewnerEqCentered} we thus obtain the system of differential equations
    \begin{align}\label{Eq:WangSystem}
        \dot{\lambda}(t) = \frac{8x}{x^2 + y^2}, \qquad \dot{x}(t) = \frac{-6x}{x^2+y^2}, \qquad \dot{y}(t) = \frac{-2y}{x^2+y^2}
    \end{align}
    for which the triple $(\lambda,x,y)$ generated by $\lambda$ is an a.e.-$t$ solution, and where each component is absolutely continuous.  Since we can thus recover each of $\lambda, x$ and $y$ through integration and the three right-hand sides in \eqref{Eq:WangSystem} are continuous, we have that \eqref{Eq:WangSystem} actually holds for \emph{all} $t$, and hence each of $\lambda, x$ and $y$ is $C^1$ on $[0,\tau)$.  By bootstrapping in \eqref{Eq:WangSystem}, then, each is $C^2$, and continuing, each is $C^\infty([0,\tau))$.\footnote{We note that Wang \cite[equation (3.2)]{WangReverse} also obtained this ODE for $\dot{\lambda}$, but only by means of using the minimal-energy formula \eqref{Eq:WangEnergy}, whereas we go the opposite direction, using \eqref{Eq:WangSystem} to derive \eqref{Eq:WangEnergy}.}
    
    Classical solutions to \eqref{Eq:WangSystem} are also unique: starting at any point $(x_0,y_0)$, both $x(t)$, $y(t)$ are bounded away from 0 on a small time interval, and so the function $f(t,\lambda,x,y) = \big(8x(x^2+y^2)^{-1}, -6x(x^2+y^2)^{-1}, -2y(x^2+y^2)^{-1}\big)$ is Lipschitz.  Thus we have smoothness, uniqueness, as well as monotonicity of $\lambda$ from \eqref{Eq:WangSystem}.

    We also note that \eqref{Eq:WangSystem} immediately gives that the EMP curve is downwards SLE$_0(-8)$ starting from $(\lambda_0, U_0) = (0, e^{i\theta})$, as
    \begin{align*}
        \Real \frac{8}{U(t)-\lambda(t)} = \Real \frac{8}{g_t(e^{i\theta})-\lambda(t)} = \Real \frac{8\overline{z(t)}}{|z(t)|^2} = \dot{\lambda}(t).
    \end{align*}
    
    For the energy formula \eqref{Eq:WangEnergy} for $m(\theta)$, note that if we flow down starting from a fixed $\theta_0 = \theta(0)$, we know the remaining curve is always the minimizer for the angle 
    \begin{align}\label{Eq:WangMinimizerAngle}
        \theta(t) = \arg(z(t)),
    \end{align}
    the argument of the image of the tip (see the remark before the proof).  Hence through the composition $m(\theta(t))$ we may regard $m$ as a function of $t$, and we have
    \begin{align*}
        \frac{dm}{d\theta} = \frac{\dot{m}}{\dot{\theta}} =  \frac{-\frac{1}{2}\dot{\lambda}^2}{\frac{-4xy}{(x^2+y^2)^2} + \frac{\dot{\lambda}y}{x^2+y^2}} = -8\frac{x}{y} = -8\cot(\theta),
    \end{align*}
    and therefore
    \begin{align*}
        m(\pi/2) - m(\theta_0) = -m(\theta_0) = \int_{\theta_0}^{\pi/2} -8\cot(\theta)d\theta = 8\log(\sin(\theta_0)),
    \end{align*}
    as claimed, completing the proof of $(i)$.
    \bigskip
    
    Our formulas in $(\ref{Thm:WangDriver})$ for the driving function  and the capacity time now are exercises in ODE.  We note from \eqref{Eq:WangSystem} that $x$ is monotonically decreasing (recall we are assuming $0 < \theta < \pi/2$), and so we may reparametrize $\lambda$ as a function of $x$ and note $\frac{d\lambda}{dx} = -\frac{4}{3}$ from \eqref{Eq:WangSystem}, and hence 
    \begin{align}\label{Eq:WangLambdaStep1}
        \lambda(t) = \lambda(x(t)) - \lambda(x(0)) = \frac{4}{3}\cos(\theta) - \frac{4}{3}x(t).
    \end{align}
    In particular,
    \begin{align}\label{Eq:WangTerminalLambda}
        \lambda(\tau) = \frac{4}{3}\cos(\theta).    
    \end{align}
    To determine $x(t)$, we note from \eqref{Eq:WangSystem} that
    \begin{align}\label{Eq:WangLambdaStep2}
        \frac{dx}{dy} = 3 \frac{x}{y}, \qquad \text{implying} \qquad x(t) = \frac{\cos(\theta)}{\sin^3(\theta)}y(t)^3
    \end{align}
    since $(x(0),y(0)) = (\cos(\theta),\sin(\theta))$.  Writing $b = b(\theta):=\frac{\cos(\theta)}{\sin^3(\theta)}$ and substituting back into the equation for $\dot{y}(t)$ yields
    \begin{align*}
        \dot{y}(t) = \frac{-2}{b^2y^5 + y}
    \end{align*}
    which has implicit solution
    \begin{align}\label{Eq:WangMinimizerYtImplicit}
        \frac{b^2}{6}y(t)^6 + \frac{1}{2}y(t)^2 = -2t + \frac{1}{6}\cos^2(\theta) + \frac{1}{2}\sin^2(\theta).
    \end{align}
    We thus see that $u(t) := y(t)^2$ satisfies the cubic
    \begin{align*}
        0 = u^3 + 3\frac{\sin^6(\theta)}{\cos^2(\theta)}u + 6 \frac{\sin^6(\theta)}{\cos^2(\theta)}\big(2t - \frac{1}{6}\cos^2(\theta) -\frac{1}{2}\sin^2(\theta)  \big) =: u^3 + pu + q.
    \end{align*}
    As the discriminant $4p^3 + 27q^2$ is manifestly positive, by Cardano's cubic formula the real root is
    \begin{multline}\label{Eq:Wangy}
        y(t)^2 = \sqrt[3]{\frac{\sin^6(\theta)}{\cos^2(\theta)}6(\tau-t) + \sqrt{\frac{\sin^{12}(\theta)}{\cos^4(\theta)}36(\tau-t)^2 + \frac{\sin^{18}(\theta)}{\cos^6(\theta)}}}\\
        + \sqrt[3]{\frac{\sin^6(\theta)}{\cos^2(\theta)}6(\tau-t) - \sqrt{\frac{\sin^{12}(\theta)}{\cos^4(\theta)}36(\tau-t)^2 + \frac{\sin^{18}(\theta)}{\cos^6(\theta)}}}.
    \end{multline}
    Pulling out the trig functions and substituting back into \eqref{Eq:WangLambdaStep2}, and then into \eqref{Eq:WangLambdaStep1}, yields
    \begin{multline}\label{Eq:WangDriverDown}
    \lambda_\theta (t) = \\\frac{4}{3}\cos(\theta) - \frac{4}{3}\left( \sqrt[3]{6(\tau-t)+\sqrt{36(\tau-t)^2+\frac{\sin^6(\theta)}{\cos^2(\theta)}}} +\sqrt[3]{6(\tau-t)-\sqrt{36(\tau-t)^2+\frac{\sin^6(\theta)}{\cos^2(\theta)}}} \right)^{3/2},
    \end{multline}
    and the claimed formula for $\xi_\theta$ then follows from reversal and \eqref{Eq:WangTerminalLambda}.
    
    The formula for $\tau_\theta$ comes from sending $t \rightarrow \tau_\theta^-$ in \eqref{Eq:WangMinimizerYtImplicit}, which yields 
    \begin{align}\label{Eq:WangMinimizerHcapTime}
        \tau_\theta = \frac{1}{12}\cos^2(\theta) + \frac{1}{4}\sin^2(\theta) = \frac{1}{6}(1 - \frac{1}{2}\cos(2\theta)).
    \end{align}

The welding formula \eqref{Eq:WangWeld} follows from conjugating the welding $\omega$ on $\mathbb{R}$ constructed in \cite{MRW} for a smooth geodesic pair by the coordinate change to the chordal Loewner setting.  We start by considering $\gamma_{\theta}$ for fixed $0 < \theta < \pi$.  As noted above in Remark \ref{Remark:RemainingWang}($\ref{Remark:WangUniversalCurve}$), $\gamma_{\theta}^* := 1/\bar{\gamma}_{\theta}$ is the hyperbolic geodesic in $\mathbb{H}\bs \gamma_{\theta}$ from $e^{i\theta}$ to $\infty$.  We wish to say that $\Gamma = \Gamma_{\theta} := \gamma_{\theta} \cup \gamma_{\theta}^*$ is the conformal image of the $C^1$-geodesic pair $\gamma_{1,\theta'} \cup \gamma_{2,\theta'}$ in $(\mathbb{D}; e^{i\theta'}, -e^{-i\theta'}, 0)$ of \cite[Corollary 2.3]{MRW}, where $\theta' = \frac{\pi}{2}-\theta$.\footnote{So note $\gamma_{j,\beta}$ refer to the curves in $\mathbb{D}$, while $\gamma_{\alpha}$ to the curve in $\mathbb{H}$.  We use the former notation to stay close to the nomenclature used in \cite{MRW}.} See Figure \ref{Fig:UniversalCurveWang}.  By the uniqueness of smooth geodesic pairs \cite[Theorem 3.9]{MRW}, it suffices to show that $\Gamma$ is at least $C^1$ in its arc-length parametrization.  Indeed, note that the downwards driving function for $\Gamma$ is
    \begin{align*}
        \lambda_\theta(t) = \begin{cases}
            \xi_{\theta}(\tau_{\theta} -t) - \xi_{\theta}(\tau_{\theta}) & 0 \leq t \leq \tau_{\theta} \\
            - \xi_\theta(\tau_\theta) & \tau_{\theta} < t.
        \end{cases}
    \end{align*}
    In particular, $\lambda_\theta$ is smooth away from $t = \tau_{\theta}$, and furthermore has the same $C^{3/2}$-regularity at $t=\tau_{\theta}$ that $\xi_{\theta}$ does as $t=0$ (recall \eqref{Eq:WangDriverExpansion}).  Hence by the correspondence between driver and curve regularity \cite{Wong}, $\Gamma$ is $C^{2-\epsilon}$ near $e^{i\theta}$ in its capacity parametrization.\footnote{In fact, $\Gamma$ is actually weakly $C^{1,1}$ there, as we discuss below in \S\ref{Sec:CorollaryWong}. The point here, however, is that $\Gamma$ is at least $C^1$.}  In particular, $\Gamma$'s unit tangent vector varies continuously, and so $\Gamma$ is the claimed image of the $C^1$ geodesic pair $\gamma_{1,\theta'} \cup \gamma_{2,\theta'}$ in $(\mathbb{D};e^{i\theta'}, -e^{-i\theta'},0)$ for some $\theta' \in (-\pi/2,\pi/2)$.  Noting that the \Mob transformation $M_\alpha$ in \eqref{Eq:GeodesicMobius} with $\alpha=\theta$ sends the triple $(0, e^{i\theta}, \infty)$ to $(-e^{-i(\frac{\pi}{2}-\theta)},0,e^{i(\frac{\pi}{2}-\theta)})$, we see we may take $\theta' = \frac{\pi}{2}-\theta$, as claimed. In particular, $\gamma_{\theta} = \Gamma([0,\tau_\theta])$ corresponds to $\gamma_{2,\theta'}$ under $M_\theta$, as in Figure \ref{Fig:UniversalCurveWang} (with $\alpha$ replaced by $\theta$).  
    
    In what follows we use the notation of \cite[Figure 1]{MRW}.  Set $B = 
    \sin(\theta) + (\frac{\pi}{2}-\theta)\cos(\theta)$ and post-compose by the unique conformal map $G:\mathbb{D} \bs \gamma_{2,\theta'} \rightarrow \mathbb{C} \bs (-\infty, B]$ which maps the triple $(-i, e^{i\theta'},0)$ to $(-\frac{\pi}{2}\cos(\theta), B, \infty)$.  Thus $G \circ M_\theta$ sends the two sides of $\gamma_{\theta}$ to (portions of) two sides of the slit $(-\infty, B]$.  By the explicit construction in \cite[Ex. 3.1]{MRW}, the welding in the latter slit-plane setting is simply the shift $\omega(x) = x + 2\pi \cos(\theta)$ with
    \begin{align*}
        \omega: \Big(-\infty, -\sin(\theta) - \Big(\frac{3\pi}{2}-\theta\Big)\cos(\theta) \Big] \rightarrow \Big(-\infty,-\sin(\theta) + \Big(\frac{\pi}{2} + \theta \Big)\cos(\theta)\Big].
    \end{align*}
   That is, $M_\theta^{-1} \circ G^{-1}(x) = M_\theta^{-1}\circ G^{-1}(\omega(x)) \in \gamma_\theta$ for $x\leq-\sin(\theta)-(3\pi/2 - \theta)\cos(\theta)$.  As we seek the welding $\varphi_{\theta}:[a_{\theta}, 0] \rightarrow [0,b_{\theta}]$ giving the identifications generated by the chordal Loewner flow, we apply the conformal map $f: \mathbb{C} \bs (-\infty, B] \rightarrow \mathbb{H}$ given by $f(z) = -1/(i\sqrt{z-B})$, where $\log(z)$ is chosen so that $-\pi/2 \leq \Arg(z) \leq \pi/2$, and thus obtain $\tilde{\varphi} = \tilde{\varphi}_{\theta} := f \circ \omega^{-1} \circ f^{-1}$, 
    \begin{align*}
        \tilde{\varphi}(x) = \frac{-x}{\sqrt{1+2\pi \cos(\theta)x^2}},
    \end{align*}
    mapping $[-1/\sqrt{2 \sin(\theta) - 2\theta \cos(\theta)},0]$ to $[0,1/\sqrt{2\sin(\theta) + 2 (\pi-\theta) \cos(\theta)}]$.  However, $\tilde{\varphi}$ is not the Loewner-normalized welding if $f \circ G \circ M_\theta$ is not $z + O(1)$ as $z \rightarrow \infty$, corresponding to the hydrodynamic normalization \eqref{Eq:LoewnerNormalize}.  The map $G$ is obtained as the Schwarz reflection of the map $G_{\theta'}$ across the imaginary axis, where 
     \begin{align*}
        G_{\theta'}(z) = \frac{1}{2}\Big(z + \frac{1}{z} \Big) - i \cos(\theta)\log(z)
    \end{align*}    
    \cite[Lemma 2.2]{MRW}.  By noting
    \begin{align*}
        G_{\theta'}'(z) = \frac{(z-e^{i\theta'})(z+e^{-i\theta'})}{2z^2},
    \end{align*}
    we see
    \begin{align*}
        G(z) = B +  \frac{e^{i\theta'}+e^{-i\theta'}}{4e^{2i\theta'}}(z-e^{i\theta'})^2 + O(z-e^{i\theta'})^3, \qquad \mathbb{D} \ni z \rightarrow e^{i\theta'},
    \end{align*}
    and thus find 
    \begin{align*}
        f \circ G \circ M_\theta(z) = \frac{z}{\sqrt{2\sin^3(\theta)}} + O(1), \qquad z \rightarrow \infty.
    \end{align*}
    Hence after post-composing by $z \mapsto \sqrt{2\sin^3(\theta)}\,z$, we arrive at
    \begin{align*}
        \varphi_{\theta}(z) := \sqrt{2\sin^3(\theta)}\, \tilde{\varphi}_{\theta}\big(x/\sqrt{2\sin^3(\theta)}\big) = \frac{-x}{\sqrt{1+\pi \frac{\cos(\theta)}{\sin^3(\theta)}x^2}},
    \end{align*}
    which maps $[-\sin^{3/2}(\theta)/\sqrt{\sin(\theta) - \theta \cos(\theta)},0]$ to $[0,\sin^{3/2}(\theta)/\sqrt{\sin(\theta) +  (\pi-\theta) \cos(\theta)}]$.


    \bigskip
    
    We show part $(\ref{Thm:WangUniversal})$ for a welding and driver corresponding to fixed $0 < \theta < \pi/2$; the argument for $\pi/2 < \theta < \pi$ is similar.  Note that the $\varphi_\theta$ in \eqref{Eq:WangWeld} is defined for all $x<0$, and we claim that for any $u<0$, $\varphi_\theta|_{[u,0]}$ welds a scaled EMP curve $r_u \gamma_\alpha$ for some angle $0 < \alpha = \alpha_u < \pi/2$ (here the scale factor $r_u$ corresponds to welding $\varphi_\theta|_{[u,0]}$ with the centered upwards Loewner flow map).  Indeed, note that if we re-scale $\gamma_\theta$ by $c_\theta := \sqrt{\pi\cos(\theta)/\sin^3(\theta)}$, the corresponding welding (in the Loewner normalization) is
    \begin{align}\label{Eq:WangRescale}
        c_\theta \varphi_\theta(x/c_\theta) = \frac{-x}{\sqrt{1+x^2}},
    \end{align}
    which is independent of $\theta$.  Thus \eqref{Eq:WangRescale} is universal in the sense of the theorem statement for generating EMP curves with tip at angle $0 < \alpha <\pi/2$, and as \eqref{Eq:WangRescale} is a fixed re-scaling of $\varphi_\theta$, we see that $\varphi_\theta$ is also universal.  Hence $\varphi_\theta|_{[u,0]}$ generates a scaled EMP curve as claimed.
    
    By \eqref{Eq:WangWeldInterval}, $\theta \mapsto -x_\theta/y_\theta$ is strictly monotonic, and so the $\alpha$ of the EMP curve $r_u \gamma_\alpha$ generated by $\varphi_\theta$ on $[u,0]$ is entirely determined by the ratio $-u/\varphi_\theta(u)$. Solving $-u/\varphi_\theta(u) = -x_\alpha/y_\alpha$ yields
    \begin{align*}
        u_\alpha = -\sqrt{\frac{\sin^3(\theta)}{\cos(\theta)} \cdot \frac{\cos(\alpha)}{\sin(\alpha)-\alpha \cos(\alpha)}},
    \end{align*}
    and we observe that the scale factor $r_u$ is determined by $u_\alpha = r_u x_\alpha$, yielding
    \begin{align*}
        r_u = \sqrt{\frac{\sin^3(\theta)}{\cos(\theta)} \cdot \frac{\cos(\alpha)}{\sin^3(\alpha)}},
    \end{align*}
    as claimed.
       
    The same argument also applies to the driving function: $\xi_\theta$ is defined for all $t \geq 0$, and $c_\theta \xi_\theta(t/c_\theta^2)$ is independent of $\theta$, showing $\xi_\theta|_{[0,t]}$ generates a scaled (and translated) EMP curve $r_{u(t)}\gamma_\alpha + \xi_\theta(t)$ for any $t>0$ under its upwards Loewner flow.  To determine the $t$ corresponding to a given $\alpha$, we note that $\varphi_\theta$ and $\xi_\theta$ generate the same curves, and that the Loewner time for $r_u \gamma_\alpha$ is 
    \begin{align*}
        r_u^2 \tau_\alpha = \frac{\sin^3(\theta)\cos(\alpha)}{6\cos(\theta)\sin^3(\alpha)}\Big( 1 - \frac{1}{2}\cos(2\alpha) \Big) =t_\alpha.
    \end{align*}
    by the scaling relation for half-plane capacity and \eqref{Eq:WangMinimizerHcapTime}.  Thus $\xi_\theta|_{[0,t_\alpha]}$ generates $r_u \gamma_\alpha + \xi_\theta(t_\alpha)$.

    \bigskip
    For part $(\ref{Thm:Wang0+})$, we start by proving uniform convergence of the drivers $\xi_\theta$ to
    \begin{align}\label{Eq:theta0driver}
        \xi_0(t) = -\frac{8}{\sqrt{3}}\sqrt{t}
    \end{align}
    as $\theta \rightarrow 0^+$.  As we will need to attach hyperbolic geodesics from $\gamma_\theta(\tau_\theta)$ and $\gamma_0(\tau_0)$ to $\infty$, we actually show uniform convergence $\tilde{\xi}_{\theta,T} \rightarrow \tilde{\xi}_{0,T}$ as $\theta \rightarrow 0^+$, where, for $\alpha \geq 0$ and fixed $T > \tau_\alpha$,
    \begin{align*}
        \tilde{\xi}_{\alpha,T}(t) := \begin{cases}
            0 & 0 \leq t \leq T-\tau_\alpha,\\
            \xi_\alpha\big(t-(T-\tau_\alpha) \big) & T - \tau_\alpha < t \leq T.
        \end{cases}
    \end{align*}
    Thus $\tilde{\xi}_{\alpha,T}$ is the upwards driver $\lambda_{\alpha}(T-t) - \lambda_\alpha(T)$ for the hull which is $\gamma_\alpha$ followed by $T-\tau_\alpha$ units of time of the hyperbolic geodesic from $\gamma_\alpha(\tau_\alpha)$ to $\infty$ in $\mathbb{H} \bs \Fill(\gamma_\alpha)$.  Now, since $\tau_\theta \rightarrow \tau_0$ and $\theta \mapsto \xi_\theta(t)$ is continuous at $\theta=0^+$, point-wise convergence $\tilde{\xi}_{\theta,T}(t) \rightarrow \tilde{\xi}_{0,T}(t)$ is clear.  Furthermore, a calculation shows that 
    \begin{align*}
        \partial_\theta \tilde{\xi}_{\theta,T}(t) = \frac{\partial \xi_\theta}{\partial \theta}(t - (T-\tau_\theta)) \tau_\theta' >0    
    \end{align*}
    when $0 < \theta < \pi/2$, $T-\tau_\theta \leq t \leq T$, and hence the point-wise limit is monotone, allowing us to upgrade to the claimed uniform convergence on the compact interval $[0,T]$ by the classical Dini theorem.
    
    
    Write $\tilde{\gamma}_\alpha$ for the curve which is $\gamma_\alpha = \gamma_\alpha[0,\tau_\alpha]$ followed by the hyperbolic geodesic $\eta_\alpha$ from $\gamma_\alpha(\tau_\alpha)$ to $\infty$ in $\mathbb{H}\bs \Fill(\gamma_\alpha[0,\tau_\alpha])$, and $g_\alpha(t,z)$ for $\tilde{\gamma}_\alpha$'s mapping-down function.  By \cite[Prop. 4.7]{Lawler}, the above driver convergence yields that, for any $T>0$ and $\epsilon>0$, one has
    \begin{align}\label{Conv:DoubleUniform}
        g_\theta(t,z) \xrightarrow[\theta \rightarrow 0^+]{unif} g_0(t,z) \qquad \text{ on } \qquad [0,T] \times \{\, z \; : \; \dist\big(z, \Fill(\tilde{\gamma}_0[0,T])\big)>\epsilon  \,\}.
    \end{align}
    We wish to say that $g_\theta(\tau_\theta, \cdot) \rightarrow g_0(\tau_0,\cdot)$ uniformly on compacts of $\mathbb{H} \bs \Fill(\gamma_0)$, which easily follows (note that we are comparing the maps at the different times $\tau_\theta$ and $\tau_0$).  Indeed, for a fixed compact $K$ of $\mathbb{H} \bs \Fill(\gamma_0)$, $\tilde{\gamma}_0([\tau_0,\tau_\theta])\cap K = \emptyset$ when $\theta$ is small, and so $g_\theta(\tau_\theta, \cdot)$ is defined on $K$ for all sufficiently-small $\theta$ by \eqref{Conv:DoubleUniform}. For $z \in K$ and $\theta$ close to zero,
    \begin{align}
        |g_\theta(\tau_\theta, z) - g_0(\tau_0,z)| &\leq |g_\theta(\tau_\theta, z) - g_0(\tau_\theta,z)| + |g_0(\tau_\theta, z) - g_0(\tau_0,z)| \notag\\
        &\leq \epsilon_1 + C\sqrt{\diam(\tilde{\gamma}_0[0,\tau_\theta])\text{osc}(\tilde{\gamma}_0,\tau_\theta-\tau_0,\tau_\theta)} \label{Ineq:OneEstimate}\\
        &\leq \epsilon_1 + \epsilon_2\notag
    \end{align}
    for all small $\theta$, where the two estimates in \eqref{Ineq:OneEstimate} are by \eqref{Conv:DoubleUniform} and \cite[Lemma 4.1]{Lawler}, respectively, and $\text{osc}(\eta, \delta, T)$ is the $\delta$-modulus of continuity of the curve $\eta$ up to time $T$,
    \begin{align*}
        \text{osc}(\eta, \delta, T):= \sup\{\, |\eta(t)-\eta(s)| \; : \; 0 \leq s,t \leq T, \; |s-t| \leq \delta \,\}.
    \end{align*}
    We conclude that we have the claimed locally-uniform convergence, and hence also, recalling \eqref{Eq:WangTerminalLambda}, the locally-uniform convergence of the centered mapping-down functions $g_\theta(\tau_\theta, \cdot) - \lambda_\theta(\tau_\theta)$ to $g_0(\tau_0, \cdot) - \lambda_0(\tau_0)$.  Call $F_\theta(\cdot)$ and $F_0(\cdot)$ the inverses of the latter two maps.  Then $F_\theta \rightarrow F_0$ locally-uniformly on $\mathbb{H}$, and in particular, writing a given segment $\eta_0(I_0)$ of the hyperbolic geodesic $\eta_0$ from $1$ to $\infty$ in $\mathbb{H} \bs \Fill(\gamma_0)$ as $F_0([iy_1,iy_2])$ for some $0 < y_1 < y_2$, we have the Hausdorff convergence
    \begin{align}\label{Conv:HausSegment}
        F_\theta([iy_1,iy_2]) \xrightarrow[]{Haus} F_0([iy_1,iy_2]),
    \end{align}
    where we observe that $F_\theta([iy_1,iy_2]) \subset \gamma_\theta^*$, the hyperbolic geodesic from $\gamma_\theta(\tau_\theta)$ to $\infty$ in $\mathbb{H}\bs \gamma_\theta$, which, as we have seen, is the reflection  $1/\overline{\gamma}_\theta$.
    
    We claim that it follows that $\eta_0 = \gamma_0^*$, i.e. $\gamma_0$'s reflection in $\partial \mathbb{D}$ is the hyperbolic geodesic.  To this end, we first show uniform convergence of the capacity-parametrized curves $\gamma_\theta$ to $\gamma_0$ on $[0,\tau_0-\epsilon]$ for any $\epsilon>0$.  Indeed, consider the Loewner energy
    \begin{align}\label{Eq:ChoppedEnergy}
        I_\epsilon(\theta) := \frac{1}{2}\int_\epsilon^{\tau_0} \dot{\xi}_\theta(s)^2ds,
    \end{align}
    which we claim is continuous in $\theta \geq 0$.  Noting that $\theta \mapsto \sin^6(\theta)/\cos^2(\theta)$ is increasing and that we have excised the singularity in $\dot{\xi}$ when $t$ and $\theta$ are both zero, through coarse bounds one readily obtains a function $g \in L^1([\epsilon,\tau_0])$ such that $\dot{\xi}_\theta(s)^2 \leq g(s)$ for all $0 \leq \theta \leq \pi/4$.  Dominated convergence then shows that $I_\epsilon(\cdot)$ is continuous, and hence bounded, on $0\leq \theta \leq \pi/4$, say.  As $I_\epsilon(\theta)$ is the energy of $\gamma_\theta([0,\tau_0-\epsilon])$, by property ($\ref{Lemma:CurvesCompact}$) of the Loewner energy in \S\ref{Sec:BackgroundLoewnerEnergy}, we thus see that for any sequence $\theta_n \rightarrow 0$, $\gamma_{\theta_n}$ has a uniform limit $\gamma = \gamma(\{\theta_n\})$ on $[0,\tau_0-\epsilon]$.  As above for the drivers $\tilde{\xi}_{\theta,T}(t)$, we have
    \begin{align*}
        \lambda_\theta(t) = \xi_\theta(\tau_\theta -t) - \xi_\theta(\tau_\theta) \xrightarrow[]{unif} \xi_0(\tau_0-t) - \xi_0(\tau_0) = \lambda_0(t)
    \end{align*}
    on $[0,\tau_0-\epsilon]$, and so by \cite[Lemma 4.2]{LMR}, $\gamma = \gamma_0|_{[0,\tau_0-\epsilon]}$, and we conclude that the limit is unique and therefore that $\gamma_\theta \rightarrow \gamma_0$ uniformly on $[0,\tau_0-\epsilon]$, as claimed.  
    
    From \eqref{Conv:HausSegment} we know that the reflections
    \begin{align*}
        1/\overline{F_\theta([iy_1,iy_2])} \xrightarrow[]{Haus} 1/\overline{F_0([iy_1,iy_2])}
    \end{align*}
    for any $0 < y_1 < y_2$, while we now see from the uniform convergence that $1/\overline{F_\theta([iy_1,iy_2])} = \gamma_\theta([t_{1,\theta},t_{2,\theta}])$ also converges in the Hausdorff sense to a segment $\gamma_0([t_1,t_2])$ of $\gamma_0$.  By uniqueness of the limit we conclude that $1/\overline{\gamma_0([t_1,t_2])} = F_0([iy_1,iy_2])$, and thus that $\eta_0 \subset \gamma_0^*$, and therefore that $\eta_0 = \gamma_0^*$, as claimed.  Furthermore, since $\gamma_0$'s reflection $\gamma_0^*$ is a hyperbolic geodesic, $\gamma_0$ itself is the hyperbolic geodesic from 0 to 1 in its component of $\mathbb{H} \bs \gamma_0^*$.  That is, $\gamma_0 \cup \gamma_0^*$ is a boundary geodesic pair in $(\mathbb{D}; 0, \infty, 1)$.

    We can now show the algebraic formula \eqref{Eq:Variety} by constructing an explicit rational function $R$ such that
    \begin{align*}
        R^{-1}(\hat{\mathbb{R}}) = \gamma_0 \cup \gamma_0^* \cup \hat{\mathbb{R}} \cup \overline{\gamma}_0 \cup \overline{\gamma}_0^*,
    \end{align*}
    where $\hat{\mathbb{R}} = \mathbb{R} \cup \{ \infty\}$, and the bar continues to denote complex conjugation.  The argument is essentially identical to that for \cite[Prop. 4.1]{PeltWang}, although our context is slightly different than the ``geodesic multichord'' setting, as our two geodesics share a common boundary point (i.e. we are working with Krusell's ``fused multichords'' in \cite{Krusell}).  We sketch the details for the convenience of the reader, and refer to Figure \ref{Fig:Gamma0} for a visualization.  
    
    Write $\Omega_0$ for the component of $\mathbb{H}\bs \gamma_0^*$ containing $\gamma_0$, and take a conformal map $R_0$ from the bounded component of $\mathbb{H} \bs \gamma_0$ to the lower half plane $\overline{\mathbb{H}} := \{\, \overline{z} \,:\, z \in \mathbb{H}\,\}$. Since $\gamma_0$ is the hyperbolic geodesic from $0$ to $1$ in $\Omega_0$, we can Schwarz-reflect across $\gamma_0$ in the domain and across $\mathbb{R}$ in the codomain to extend $R_0$ to map the unbounded component of $\Omega_0 \bs \gamma_0$ to the upper half plane $\mathbb{H}$.  Note that reflecting across $\gamma_0$ in $\mathbb{H}$ sends a unique point $0 < x_0 <1$ to $\infty$.  By post-composing with an element of $\Aut(\nH)=PSL_2(\mathbb{R})$, we may assume that $R_0(x_0) = \infty$ and $R_0(0) = R_0(1)-1=0$.  
    
    Since $\gamma_0^*$ is the hyperbolic geodesic in $\mathbb{H} \bs \Fill(\gamma_0)$, we can again reflect $R_0$ to map the right-most component of $\mathbb{H} \bs \gamma_0^*$ back to the lower half plane $\overline{\mathbb{H}}$, taking $\gamma_0^*$ onto $[1,\infty)$. The map thus constructed sends $\mathbb{R} \cup \gamma_0 \cup \gamma_0^*$ to $\hat{\mathbb{R}}$, permitting us to once more Schwarz-reflect across $\mathbb{R}$ to obtain a meromorphic, degree-three branched cover $R$ of the sphere, which is therefore a rational function.  By construction, $R$ fixes $0,1$ and $\infty$ (and so is uniquely normalized), and has orders 2, 3 and 2 at these points, respectively.  Recalling the simple pole at $x_0$, we conclude $R(z) = \frac{z^2(az+b)}{z-x_0}$, and the constraints $R(1)=1$ and $R(z)-1 = \frac{c(z-1)^3}{z-x_0}$ then yield
    \begin{align*}
        R(z) &= \frac{z^2(z-3)}{1-3z} = 1 + \frac{(z-1)^3}{1-3z} = \Real(R(z)) + i \frac{2y\big(y^2(4-3x)-3x(x-1)^2\big)}{(1-3x)^2+9y^2}.
    \end{align*}
    We thus obtain \eqref{Eq:Variety} for the pre-image $R^{-1}(\mathbb{R}) = \{z \; : \; \Imag(R(z)) = 0\}$.


    Lastly, we note that the $\pi/3$ intersection of $\gamma_0(\tau_0-)$ with $\mathbb{R}$ follows from either the explicit formula for the intersection angle $\pi \theta$ of such drivers \cite[Prop. 3.2]{LMR} using $\kappa = 8/\sqrt{3}$, or from observing that $R(z)-1 = (z-1)^3/(1-3z)$ pulls back an interval $(-\epsilon, \epsilon)$ into six arcs that meet at equal angles around $x=1$, two of which lie on $\mathbb{R}$.
\end{proof}

\begin{remark}\label{Remark:FurtherComputations}
We note one other possible explicit calculation: for $0 \leq t \leq \tau$, the Loewner energy of $\gamma_\theta([0,t])$ is
        \begin{align*}
            I(\gamma_\theta([0,t])) = -4\log\Big(\sin^2(\theta) + \frac{\cos^2(\theta)}{\sin^4(\theta)} y(t)^4 \Big),
        \end{align*}
        where $y(t)^4 = \big(\Imag(g_t(e^{i\theta})-\lambda(t))\big)^4$ is explicitly given by the square of \eqref{Eq:Wangy}. To see this, recall that the curve remaining after mapping down $\gamma_\theta([0,t])$ is the minimizer through angle $\alpha(t) := \arctan(y(t)/x(t))$ (by symmetry we may assume that $0 < \theta < \pi/2$, and thus likewise that $\alpha \in (0,\pi/2)$), and so has energy $-8\log(\sin(\alpha(t)))$.  Recalling \eqref{Eq:WangLambdaStep2}, we thus see the energy of the first portion $\gamma_\theta([0,t])$ of the curve is
    \begin{align*}
        -8\log(\sin(\theta)) + 8\log(\sin(\alpha(t))) &=  -8\log(\sin(\theta)) - 4\log\Big(1 + \frac{\cos^2(\theta)}{\sin^6(\theta)}y(t)^4\Big),
    \end{align*}
    as claimed.

\end{remark}
We close section \S\ref{Sec:Wang} by noting two corollaries of Theorem 3.3.
\subsection{Corollary: boundary geodesic pairs}
Let $D \subsetneq \mathbb{C}$ be a simply-connected domain with boundary prime ends $a$ and $b$ and let $\zeta \in D$.  In \cite{MRW}, Marshall, Rohde and Wang defined a \emph{geodesic pair} in $(D; a,b,\zeta)$ as a simple curve $\gamma \subset D$, continuously parametrized on $(0,\infty)$, say, such that $\gamma(\tau)=\zeta$ for some $\tau \in (0,\infty)$, and where $\gamma(0,\tau)$ is the hyperbolic geodesic in $D\bs \gamma[\tau,\infty)$ from $a$ to $\zeta$, and $\gamma(\tau,\infty)$ is the hyperbolic geodesic in $D\bs \gamma(0,\tau]$ from $\zeta$ to $b$.  In \cite{MRW}, $\zeta$ always lies in the interior of $D$, but as a corollary to Theorem \ref{Thm:Wang}$(\ref{Thm:Wang0+})$ we may extend this definition to include $\zeta \in \partial D$.  Indeed, when $a_1, a_2$ and $\zeta$ are boundary prime ends of $D$, we define a \emph{boundary geodesic pair} in  $(D;a_1,a_2,\zeta)$ to be two simple curves $\gamma_1, \gamma_2 \subset D$ which connect $a_1$ to $\zeta$ and $\zeta$ to $a_2$, respectively, do not intersect in $D$, and have the property that $\gamma_j$ is the hyperbolic geodesic from $a_j$ to $\zeta$ in its component of $D \bs \gamma_{3-j}$, $j=1,2$. We require that $\zeta$ is distinct from the $a_j$ but allow $a_1=a_2$.

The following is the boundary geodesic pair version of \cite[Thm. 2.5]{MRW}.  
\begin{corollary}\label{Cor:BoundaryGeodesic}
    If $D \subset \mathbb{C}$ is a simply-connected domain with boundary prime ends $a_1,a_2$ and $\zeta$, where $\zeta \notin \{a_1,a_2\}$, there is a unique boundary geodesic pair $\gamma_1 \cup \gamma_2$ in $(D;a_1,a_2,\zeta)$.  If $a_1 \neq a_2$ and $\partial D$ has a tangent at $a_j$, then $\gamma_j$ intersects $\partial D$ orthogonally at $a_j$.  If $\partial D$ has a tangent at $\zeta$, then $\gamma_1$, $\gamma_2$ and $\partial D$ form three angles of $\pi/3$ at $\zeta$, and the hyperbolic geodesic $\eta_j$ in $D$ from $a_j$ to $\zeta$ bisects the angle between the $\gamma_j$.  If $a_1=a_2$, then $\gamma_1$, $\gamma_2$ and $\partial D$ form three angles of $\pi/3$ when $\partial D$ has a tangent at either $a_1$ or $\zeta$.
\end{corollary}
In the case of $\zeta \in D$, the tangent at $\zeta$ to the geodesic pair bisects the angle between the hyperbolic geodesic segments $\eta_j$ from $a_j$ to $\zeta$ in $D$ \cite[Thm. 2.5]{MRW}.  When $\zeta \in \partial D$ (and $\partial D$ has a tangent at $\zeta$), then $\eta_1$ and $\eta_2$ intersect tangentially at $\zeta$, and this no longer holds.  Corollary \ref{Cor:BoundaryGeodesic} says that the roles flip and that each of the $\eta_j$ now bisect the angle between the $\gamma_j$.  See Figure \ref{Fig:BoundaryGeodesics}.
\begin{figure}
    \centering
    \includegraphics[scale=0.1]{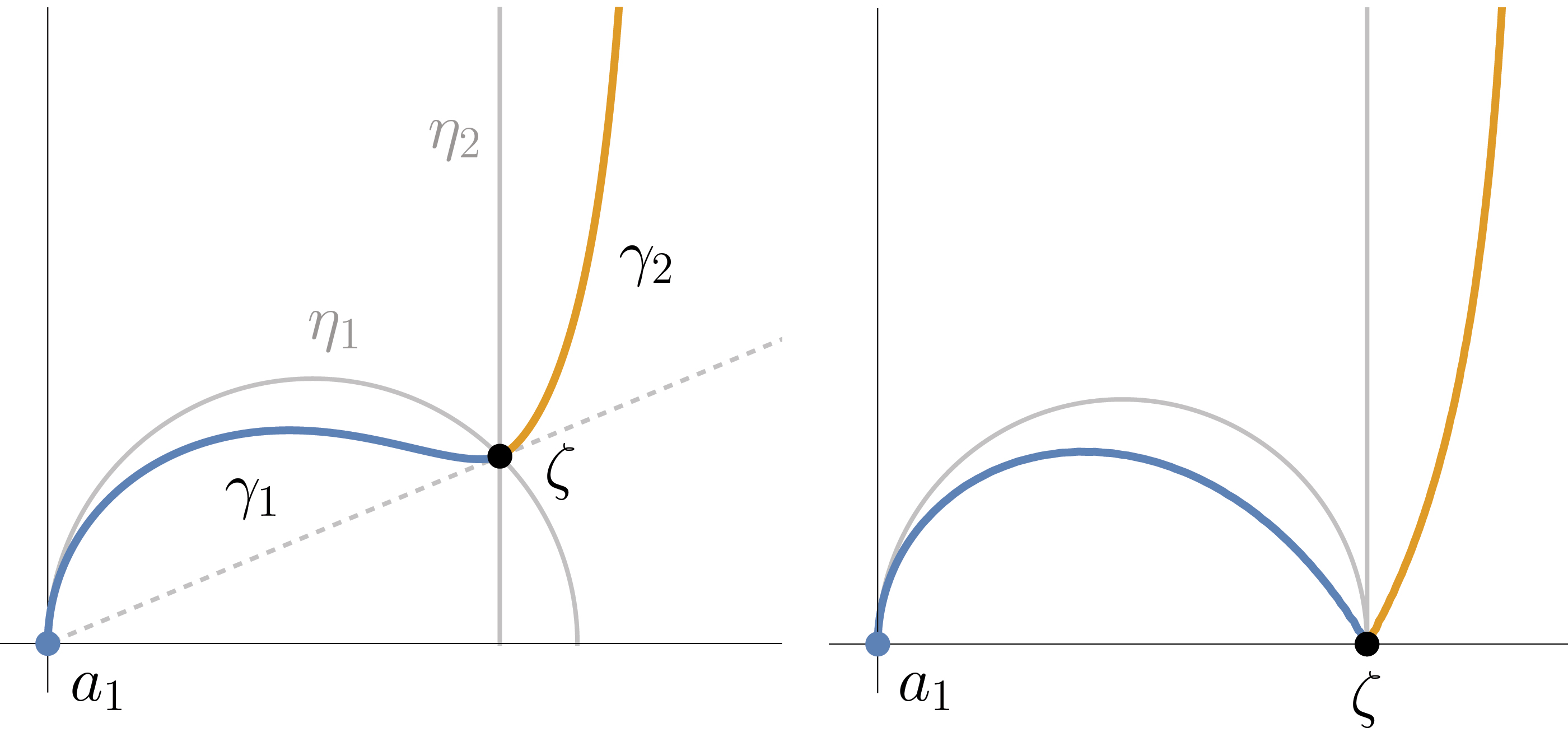}
    \caption{\small When $\zeta$ lies in the interior of $D$, the smooth geodesic pair bisects the angle between the hyperbolic geodesics $\eta_j$ connecting the $a_j$ to $\zeta$.  When $\zeta \in \partial D$, the roles are reversed.}
    \label{Fig:BoundaryGeodesics}
\end{figure}
\begin{proof}
    The proof of existence is nearly identical to that for \cite[Thm. 2.5]{MRW}; in the case $a_1 \neq a_2$, we simply apply conformal invariance to transport the geodesic pair $\gamma_0 \cup \gamma_0^*$ in $(\mathbb{H};0,\infty,1)$ constructed in Theorem   \ref{Thm:Wang}$(\ref{Thm:Wang0+})$ to the domain in question.  Uniqueness follows from the fact that, given a geodesic pair $(\eta_1 \cup \eta_2)$ in $(\mathbb{H};0,\infty,1)$, we can use it as in the proof of Theorem \ref{Thm:Wang}$(\ref{Thm:Wang0+})$ to build a degree-three analytic branched cover $R$ of $\chat$, resulting in the same variety \eqref{Eq:Variety} after post-composition by an element of $PSL(2,\mathbb{R})$. 
    
    When $a_1=a_2$, we can without loss of generality take $(D;a_1,a_1,\zeta) = (\mathbb{H};0,0,\infty)$, and we see that the two rays $\gamma_1 = \{\, re^{\pi i/3} \; : \; r>0 \,\}$ and $\gamma_2 = \{\, re^{2\pi i/3} \; : \; r>0 \,\}$ form a geodesic pair satisfying the stated properties.  Given any geodesic pair in $(\mathbb{H}; 0,0,\infty)$, the rational function $R$ we construct is again a third-degree branched cover of $\chat$, with ramification only above $0$ and $\infty$, each with degree three.  It follows that $R(z) =cz^3$ for some $c \in \mathbb{R} \bs \{0\}$, and we obtain the same pair of rays $\gamma_1 \cup \gamma_2$. 
\end{proof}

\subsection{Corollary: Wong's driver-curve regularity theorem is sharp}\label{Sec:CorollaryWong}
We recall the driver-curve regularity correspondence theorem of Carto Wong, which says that the (inverse) Loewner transform $\lambda \mapsto \gamma$ generally increases regularity by half a degree.
\begin{proposition}\cite[Thm. 4.7, Thm. 5.2, Thm. 6.2]{Wong}\label{Prop:Wong}
    Let $t \mapsto \gamma(t)$ be the capacity parametrization of a curve $\gamma \subset \mathbb{H}\cup\{0\}$ with driver $\lambda:[0,T]\rightarrow \mathbb{R}$, and set $\Gamma(t) := \gamma(t^2)$ for $0 \leq t \leq \sqrt{T}$.  If $\lambda \in C^\beta([0,T])$ for $\beta \in (1/2,3/2)\cup(3/2,2]$, then $\Gamma \in C^{\beta + 1/2}([0,\sqrt{T}])$.  If $\lambda \in C^{3/2}([0,T])$, then $\Gamma$ is weakly $C^{1,1}$ on $[0,\sqrt{T}]$.
\end{proposition}
\noindent Note that the parametrization of $\Gamma$ in terms of the square of the capacity time is only to deal with the technicality of how smooth curves necessarily ``jump'' off the real line like $2i\sqrt{t} + O(t)$ (see the expansion \eqref{Eq:InfinitesimalCurve} below), and so $t \mapsto \gamma(t)$ cannot be smooth at $t=0^+$.  For fixed positive $t$, however, there is no issue and $\gamma$ and $\Gamma$ have identical classes of regularity near $t$ and $\sqrt{t}$, respectively.   

From work on $\Gamma \mapsto \lambda$ direction of the Loewner transform by Rohde-Wang \cite[Thm. 1.1]{RohdeWang}, Wong's result gives optimal \Hol regularity for $\Gamma$ when $\beta \in (1/2,1) \cup (1,3/2)$.\footnote{Presumably this also holds for $\beta \in (3/2,2)$, cf. \cite[\S4.1]{RohdeWang}.  What we precisely mean is that $\lambda \in C^\beta$ if and only if $\Gamma \in C^{\beta +1/2}$ for $\beta \in (1/2,1) \cup (1,3/2)$.  For $\Gamma \in C^{1+1/2}$, \cite{RohdeWang} only shows that $\lambda \in C^{0,1}$, so we omit this value.  As we focus on the $\lambda \mapsto \Gamma$ direction, we do not further discuss the $\Gamma \mapsto \lambda$ map, other than to say that questions also remain in that direction about sharpness for certain values of $\beta$; see \cite[\S4.1]{RohdeWang}.} Wong was not sure, however, if the ``weak'' qualifier was needed in the case $\lambda \in C^{3/2}$ (see the paragraph after Theorem 5.2 in \cite{Wong}).  We recall again that the weak-$C^{1,1}$ class $C^{1,1}_w$ consists of those functions $\Gamma: (a,b) \rightarrow \mathbb{C}$ with Lipschitz derivative up to a logarithmic factor.  That is, $\Gamma \in C_w^{1,1}$ if $\Gamma'$ exists and there is $C < \infty$ such that 
\begin{align}\label{Ineq:WeakC11}
    |\Gamma'(s) - \Gamma'(t)| \leq C\Big(1 \vee \log \frac{1}{|s-t|} \Big)|s-t|
\end{align}
for all $s,t \in (a,b)$ with $s \neq t$, where $A \vee B := \max\{A,B\}$.  We say that $\Gamma$ is $C^{1,1}_w$ at a fixed $t$ if there is some $\delta>0$ such that $\Gamma'$ exists on $(t-\delta, t+\delta)$ and \eqref{Ineq:WeakC11} holds for all $s \in (t-\delta, t+\delta) \bs \{t\}$.  

As noted in the introduction, Lind and Tran \cite[Example 7.1]{LindTran} gave an example of a $C^{3/2}$ driver whose trace $\Gamma$ was $C^{1,1}$ but not $C^2$, and so one cannot always na\"{i}vely add half a degree of regularity and obtain the correct result.  It has not been known, however, if there exists a curve with $C^{3/2}$ driver where the logarithm in \eqref{Ineq:WeakC11} is the best one can do.  The EMP curves show that this is, indeed, the case.  We adopt the following terminology to state the result. 
\begin{definition}
For fixed $t$, we say a function $f$ which is $C^{1,1}_w$ at $t$ is \emph{sharply $C^{1,1}_w$} at $t$ if
\begin{align*}
    \limsup_{s \rightarrow t} \frac{|f'(s)-f'(t)|}{g(|s-t|)|s-t|} = +\infty
\end{align*}
whenever $g$ satisfies $g(x) = o(\log\frac{1}{x})$ as $x \rightarrow 0^+$.  That is, the logarithm in \eqref{Ineq:WeakC11} cannot be substituted with any function of slower growth.
\end{definition}
Equivalently, $f$ is sharply $C^{1,1}_w$ at $t$ if $f \in C^{1,1}_w$ and 
\begin{align*}
    \limsup_{s\rightarrow t} \frac{|f'(s)-f'(t)|}{|s-t|\log \frac{1}{|s-t|}} \neq 0.
\end{align*}
Indeed, the $\limsup$ collapses to zero if and only if $f'$ actually has a smaller order of growth.  

We can now state the main result of this section, which we phrase as a corollary to Theorem \ref{Thm:Wang}.

\begin{corollary}\label{Cor:WongSharp}
    The $\beta=3/2$ case of Wong's theorem is sharp.  In fact, for any $\theta \in (0,\pi)\bs \{\pi/2\}$, the geodesic pair in $(\mathbb{H}; 0,\infty, e^{i\theta})$ formed by the EMP curve $\gamma_\theta$ and its reflection $\gamma_\theta^*$ in $\partial \mathbb{D}$ has driver which is $C^{3/2}$ but trace whose capacity parametrization $t \mapsto \gamma(t)$ is sharply $C^{1,1}_w$ at $\tau_\theta$, the time yielding $\gamma(\tau_\theta)=e^{i\theta}$. 
\end{corollary}

\noindent Given \cite[Example 7.1]{LindTran}, Corollary \ref{Cor:WongSharp} also yields:

\begin{corollary}\label{Cor:Regularity}
    There exist simple curves with drivers of the same regularity ($C^{3/2}$), but where the half-plane capacity parametrizations have different regularity ($C^{1,1}$ versus sharply $C^{1,1}_w$). 
\end{corollary}

\begin{figure}
    \centering
    \includegraphics[scale=0.10]{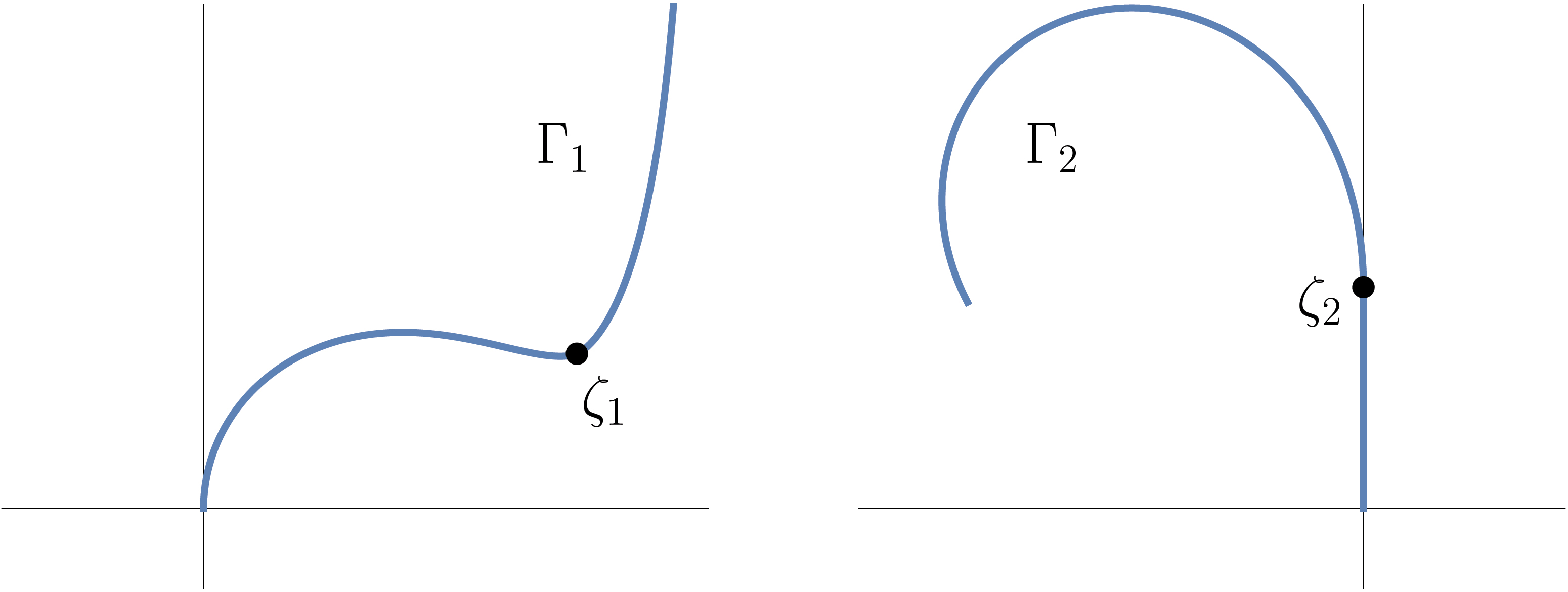}
    \caption{\small Curves $\Gamma_j$ whose drivers and conformal maps have the same regularity, but whose half-plane capacity parametrizations $\gamma_j$ have different regularity.  Owing to the non-smooth point $\zeta_j$ on each curve, both drivers are $C^{3/2}$ and the conformal maps $F_j$ from $\mathbb{H}$ satisfying $F_j(0)=\zeta_j$ have expansions $F_j'(0) = a_{j1} + a_{j2}z\log(z) + o\big(z\log(z)\big)$ for some $a_{jk} \neq 0$ as $z \rightarrow 0$ in the closed upper half plane.  While \cite[Example 7.1]{LindTran} showed the half-plane capacity parametrization of $\Gamma_2$ is $C^{1,1}$, we show that of $\Gamma_1$ is slightly worse, namely, $C^{1,1}_w$.  See Corollary \ref{Cor:WongSharp}, Example \ref{Eg:WongSharp} and Table \ref{Table:WongSharp}.
    }
    \label{Fig:WongSharp}
\end{figure}


\noindent For further discussion of these two examples see Figure \ref{Fig:WongSharp} and Table \ref{Table:WongSharp} and Example \ref{Eg:WongSharp} below.  

Our proof of Corollary \ref{Cor:WongSharp} will be aided by the following technical yet elementary result.
\begin{lemma}\label{Lemma:C11w} Let $I_j=(a_j,b_j)$, $j\in\{1,2\}$, be open intervals in $\mathbb{R}$.
    \begin{enumerate}[$(i)$]
       \item\label{Lemma:C11wComposition} Suppose $\gamma:I_2 \rightarrow \mathbb{C}$ is $C^{1,1}_w$ and $\sigma:I_1 \rightarrow I_2$ is $C^{1,1}_w$, with $\sigma'>0$ on $I_1$.  If $\gamma$ is not sharply $C^{1,1}_w$ at $t \in (a_2,b_2) \cap \sigma(I_1)$ and $\sigma$ is not sharply $C^{1,1}_w$ at $\sigma^{-1}(t)$, then $\gamma \circ \sigma$ is not sharply $C^{1,1}_w$ at $\sigma^{-1}(t)$.
        \item\label{Lemma:C11wInverse} If $\sigma:I_1 \rightarrow \mathbb{R}$ is $C^{1,1}_w$ with positive derivative, then $\sigma$ is sharply $C^{1,1}_w$ at $v \in (a_1,b_1)$ if and only if $\sigma^{-1}$ is sharply $C^{1,1}_w$ at $\sigma(v)$.
    \end{enumerate}
\end{lemma}
The proof of the lemma is a simple exercise in expansions, the chain rule, and the triangle inequality.  For completeness we include the argument for $(\ref{Lemma:C11wComposition})$.

\begin{proof}[Proof of $(i)$]
    Pick $\epsilon >0$ such that $t \in (a_2+\epsilon, b_2-\epsilon)$ and $v := \sigma^{-1}(t) \in (a_1+\epsilon, b_2-\epsilon)$, and write 
    \begin{align*}
        M_\gamma := \max_{[a_2+\epsilon,b_2-\epsilon]}|\gamma'| \qquad \text{ and } \qquad M_\sigma := \max_{[a_1+\epsilon,b_1-\epsilon]}|\sigma'|.
    \end{align*}  
    By assumption we have functions $g_1, g_2$ such that
    \begin{align*}
        |\gamma'(s) - \gamma'(t)| \leq C_1 g_1(|s-t|)|s-t| \qquad \text{ and } \qquad |\sigma'(u)-\sigma'(v)| \leq C_2 g_2(|u-v|)|u-v|
    \end{align*}
    for $s$ near $t$ and $u$ near $v$, respectively, where $g_j(x) = o(\log\frac{1}{x})$ as $x \rightarrow 0^+$.  We observe
    \begin{align*}
        |\gamma'(\sigma(u))\sigma'(u) - \gamma'(\sigma(v))\sigma'(v)| &\leq |\gamma'(\sigma(u))\sigma'(u) - \gamma'(\sigma(v))\sigma'(u)| + |\gamma'(\sigma(v))\sigma'(u) - \gamma'(\sigma(v))\sigma'(v)|\\
        &\leq M_\sigma |\gamma'(\sigma(u)) - \gamma'(\sigma(v))| + M_\gamma |\sigma'(u) - \sigma'(v)|\\
        &\leq M^2_\sigma C_1g_1\big(|\sigma(u)-\sigma(v)| \big)|u-v| + M_\gamma C_2g_2(|u-v|)|u-v|.
    \end{align*}
    Using the Taylor expansion of $\sigma$ at $v$, we see 
    \begin{align}\label{Lim:StillLittleOh}
        \frac{g_1\big( |\sigma(u)-\sigma(v)| \big)}{\log \frac{1}{|u-v|}} = \frac{g_1\big( |\sigma(u)-\sigma(v)| \big)}{\log \frac{1}{|\sigma(u)-\sigma(v)|}-\log \frac{1}{|\sigma'(v)|} + o(1)} \rightarrow 0
    \end{align}
    as $u \rightarrow v$, and so the estimate
    \begin{align*}
        |(\gamma \circ \sigma)'(u) - (\gamma\circ \sigma)'(v)| \leq C_3 \max\big\{g_1\big( |\sigma(u)-\sigma(v)| \big), g_2(|u-v|)  \big\}|u-v|
    \end{align*}
  shows the composition is not sharply $C^{1,1}_w$ at $v$, as claimed.
\end{proof}

\begin{proof}[Proof of Corollary \ref{Cor:WongSharp}]
    By Theorem \ref{Thm:Wang}($\ref{Thm:WangDriver}$), $\gamma := \gamma_\theta \cup \gamma_\theta^*$ has downwards driver
    \begin{align*}
        \lambda_\theta(t) = \begin{cases}
            \xi_\theta(\tau_\theta-t) - \xi_\theta(\tau_\theta) & 0 \leq t \leq \tau_\theta \\
            - \xi_\theta(\tau_\theta) & \tau_\theta < t,
        \end{cases}
    \end{align*}
    and expanding the derivative yields
    \begin{align*}
        \dot{\lambda}_\theta(t) =  -16\frac{\cos(\theta)}{\sin^3(\theta)} (\tau_\theta-t)^{1/2} + O(\tau_\theta - t)^{5/2}, \qquad t \nearrow \tau_\theta.
    \end{align*}
    Thus $\lambda_\theta$ is $C^{3/2}$ at $t = \tau_\theta$ and smooth elsewhere, and so by Proposition \ref{Prop:Wong} (and our comment in the subsequent paragraph), the capacity parametrization $t \mapsto \gamma(t)$ of our curve is at least $C^{1,1}_w$ at $t= \tau_\theta$.  We show that this membership is sharp in three steps:  $(i)$ first, we show the conformal parametrization of $\gamma$ is sharply $C^{1,1}_w$ at the point corresponding to $e^{i\theta}$, and then $(ii)$ that the arc-length parametrization likewise is, and finally $(iii)$ that this also holds for the half-plane capacity parametrization.
    
    \bigskip
    \noindent $(i)$ As noted above in Remark \ref{Remark:RemainingWang}$(\ref{Remark:WangUniversalCurve})$, the \Mob transformation $M_\theta(z) = ie^{-i\theta}\frac{z-e^{i\theta}}{z-e^{-i\theta}}$ takes $\gamma = \gamma_\theta \cup \gamma_\theta^*$ to the $C^1$ geodesic pair $\tilde{\gamma} = \gamma_{1,\tilde{\theta}} \cup \gamma_{2,\tilde{\theta}}$ in $(\mathbb{D}; e^{i\tilde{\theta}}, -e^{-i\tilde{\theta}},0)$, where $\tilde{\theta} = \frac{\pi}{2}-\theta$ (see also Figure \ref{Fig:UniversalCurveWang}).  The conformal map
    \begin{align*}
        G_{\tilde{\theta}}(z) := \frac{1}{2}\Big( z + \frac{1}{z} \Big) -i \sin(\tilde{\theta}) \log(z),
    \end{align*}
    with $\arg(z) \in (0,\pi)$, then maps the component $\tilde{\Omega}_1$ of $\mathbb{D} \bs \tilde{\gamma}$ with larger imaginary values to the lower half plane, sending $0$ to $\infty$ (see \cite[Figure 1]{MRW}, and Schwarz-reflect over the green line $i \mathbb{R}_{\geq 0} \cap \mathbb{D}$).  Hence the composition $J_\theta:=(1/G_{\tilde{\theta}}) \circ M_\theta$ is a conformal map of the left component $\Omega_1$ of $\mathbb{H} \bs \gamma$ to $\mathbb{H}$, sending $e^{i\theta}$ to 0.  
    
    We claim that $J_\theta'$ is continuous and non-zero in a neighborhood $B_\epsilon(e^{i\theta}) \cap \overline{\Omega}_1$ of $e^{i\theta}$ in the closure of $\Omega_1$ (note that in this section, the ``bar'' of a set always denotes closure).  We have
    \begin{align*}
        G_{\tilde{\theta}}'(z) = \frac{z^2-2i\sin(\tilde{\theta})z-1}{2z^2},
    \end{align*}
    which vanishes at $e^{i\theta}$ and $-e^{-i\theta}$.  Since $M_\theta$ is \Mob, $J_\theta'$ is as claimed on a punctured neighborhood of $e^{i\theta}$ in $\overline{\Omega}_1$.  For the potentially problematic point $e^{i\theta}$, corresponding to $z=0$ for $G_{\tilde{\theta}}$, we compute
    \begin{align}\label{Eq:GExpansion0}
        \Big( \frac{1}{G_{\tilde{\theta}}} \Big)'(z) = 2 + 8i\sin(\theta)z \log(z) + O(z), \qquad \overline{\tilde{\Omega}}_1 \ni z \rightarrow 0,
    \end{align}
    and we conclude the claim about $J_\theta$ holds.  
    
    The conformal parametrization of $\gamma$ near $e^{i\theta}$ is given by the inverse mapping $K_\theta(x) := J^{-1}_\theta(x)$ for $x \in \mathbb{R}$ near $0$.  By the above, we see $K_\theta$ has continuous, non-zero derivative at $x \in B_{\epsilon'}(0) \cap \mathbb{R}$, where 
    \begin{align*}
        K'_\theta(x) = \lim_{\overline{\mathbb{H}} \ni w \rightarrow x} \frac{K_\theta(w)-K_\theta(x)}{w-x}.    
    \end{align*}
    By \eqref{Eq:GExpansion0} and the fact that $M_\theta$ is M\"{o}bius, we have that $J_\theta$ is sharply $C^{1,1}_w$ at $e^{i\theta}$ in $\overline{\Omega}_1$ (with the obvious adjustment of definitions for functions of a complex variable).  Similar to Lemma \ref{Lemma:C11w}($\ref{Lemma:C11wInverse}$), we conclude $K_\theta$ likewise is on $\mathbb{R}$ near $0$.  In fact, a computation shows
 \begin{align}
        K_\theta'(z) &= e^{i\theta}\sin(\theta) - ie^{i\theta}\sin(2\theta)z \log(z) + o(z\log(z)), \qquad \overline{\mathbb{H}} \ni z \rightarrow 0. \label{Eq:K_theta'}
    \end{align}
    We conclude the conformal parametrization of $\gamma$ is sharply $C^{1,1}_w$ at the point $x=0$ corresponding to $e^{i\theta}$.

    \bigskip
    \noindent $(ii)$  We claim this also carries over to the arc-length parametrization $\eta$, which we normalize to satisfy $\eta(0) = e^{i\theta}$.  Indeed, setting
    \begin{align}\label{Eq:ArcLengthSigma}
        \sigma(t) := \int_{0}^t |K'_\theta(x)|dx,
    \end{align}
    we have $\eta(u)= K_\theta(\sigma^{-1}(u))$ nearby $u=0$ (we allow $t$ to be negative in \eqref{Eq:ArcLengthSigma}), with
    \begin{align*}
        \eta'(u) = \frac{K'_\theta(\sigma^{-1}(u))}{|K'_\theta(\sigma^{-1}(u))|}.
    \end{align*}
    Since as $u \rightarrow 0$,
    \begin{align*}
        \sigma^{-1}(u) = \frac{u}{\sigma'(0)}(1+o(1)) = \frac{u}{\sin(\theta)}(1+o(1)),
    \end{align*}
    from \eqref{Eq:K_theta'} we find
    \begin{align*}
        K'_\theta(\sigma^{-1}(u)) = e^{i\theta}\sin(\theta) - ie^{i\theta}\frac{\sin(2\theta)}{\sin(\theta)}u \log|u| + o(u\log|u|),
    \end{align*}
    and therefore
    \begin{align}
        \eta'(u)&=\frac{K'_\theta(\sigma^{-1}(u))}{\sqrt{K'_\theta(\sigma^{-1}(u)) \overline{K'_\theta(\sigma^{-1}(u))}}}\label{Eq:ArcLengthC11wsStep1}\\ & = \frac{e^{i\theta}\sin(\theta) - ie^{i\theta}\frac{\sin(2\theta)}{\sin(\theta)}u \log|u| + o(u\log|u|)}{\sin(\theta)\sqrt{1+o(u\log|u|)}}\label{Eq:ArcLengthC11wsStep2}\\  
        &= e^{i\theta} - ie^{i\theta}\frac{\sin(2\theta)}{\sin^2(\theta)}u \log|u| + o(u\log|u|).\label{Eq:ArcLengthC11ws}
    \end{align}
    Hence $\eta$ is sharply $C^{1,1}_w$ at the point corresponding to $e^{i\theta}$, as claimed.  
    \bigskip
    
    \noindent $(iii)$ We built $\eta$ from the conformal parametrization $K_\theta$, but we can also do so from the capacity parametrization of  $\gamma$, which, abusing notation, we also write as $\gamma$.  Indeed, as in \eqref{Eq:ArcLengthSigma} we set
    \begin{align*}
        \tilde{\sigma}(t) := \int_{\tau_\theta}^t |\gamma'(s)|ds.
    \end{align*}
    By \cite[Thm. 5.2]{Wong}, $|\gamma'(s)|$ has a positive lower bound in a neighborhood of $\tau_\theta$, and we can thus write the arc-length parametrization as $\eta(u) = \gamma \circ \tilde{\sigma}^{-1}(u)$ nearby $u=0$.  Noting that
    \begin{align}\label{Ineq:sigmatilde}
        |\tilde{\sigma}'(s)-\tilde{\sigma}'(\tau_\theta)| = \big||\gamma'(s)|-|\gamma'(\tau_\theta)| \big| \leq |\gamma'(s) - \gamma'(\tau_\theta)|,
    \end{align}
    we see that $\gamma \in C^{1,1}_w$ implies that $\tilde{\sigma}$ is also $C^{1,1}_w$ at $\tau_\theta$, and hence by Lemma \ref{Lemma:C11w}($\ref{Lemma:C11wInverse}$), $\tilde{\sigma}^{-1}$ is $C^{1,1}_w$ at $0$. Since $\gamma \circ \tilde{\sigma}^{-1}$ is sharply $C^{1,1}_w$ at 0 by \eqref{Eq:ArcLengthC11ws}, Lemma \ref{Lemma:C11w}($\ref{Lemma:C11wComposition}$) says that either $\gamma$ or $\tilde{\sigma}^{-1}$ is sharply $C^{1,1}_w$ at $\tau_\theta$ or $0$, respectively.  If $\gamma$ is not, then \eqref{Ineq:sigmatilde} implies that $\tilde{\sigma}$ also is not, and so $\tilde{\sigma}^{-1}$ also is not by Lemma \ref{Lemma:C11w}($\ref{Lemma:C11wInverse}$).  Thus neither $\gamma$ or $\tilde{\sigma}^{-1}$ is sharply $C^{1,1}_w$, a contradiction, forcing us to conclude $\gamma$ is sharply $C^{1,1}_w$ at $\tau_\theta$. 
\end{proof}
Our above approach leads to a lemma on the boundary behavior of conformal maps, which originated from comparing the parametrizations of the geodesic pair above to the case of \cite[Example 7.1]{LindTran}.  We state the lemma and then compare these two curves in light of it in Example \ref{Eg:WongSharp} below. 

Let $F$ be a conformal map from $\mathbb{H}$ to one component of $\mathbb{H} \bs \gamma$, taking $x_0 \in \mathbb{R}$ to some point on $\gamma$.  It is not always the case that if $F$ is sharply $C^{1,1}_w$ at $x_0$, then the arc-length parametrization $\eta$ of $\gamma$ is also sharply $C^{1,1}_w$ at $F(x_0)$.  We can read this off, however, from the expansion of $F'$ near $x_0$.  We state a result which covers $C^{1,1}_w$ regularity as a special case.
\begin{lemma}\label{Lemma:ArcLengthCoeff}
    Let $F:\mathbb{H} \rightarrow \Omega$ be conformal with $F'(z)$ defined in some neighborhood $B_\epsilon(x_0) \cap \overline{\mathbb{H}}$, where, for some $a_1, a_2 \neq 0$ and $0 < \alpha \leq 1$, $0 \leq \beta$,
    \begin{align}\label{Eq:ConformalMapExpansion}
        F'(z) = a_1 + a_2(z-x_0)^\alpha \log^\beta(z-x_0) +o\big( (z-x_0)^\alpha\log^\beta(z-x_0) \big), \qquad \overline{\mathbb{H}} \ni z \rightarrow x_0.
    \end{align}
    Then $\partial \Omega$ has an arc-length parametrization $\eta$ that satisfies
    \begin{align}\label{Eq:ArcLengthEtaRegularity1}
        \eta'(u) = \eta'(u_0) +O\big( (u-u_0)^\alpha \log^\beta|u-u_0| \big) 
    \end{align}
    in a neighborhood of $\eta(u_0) := F(x_0)$.  Furthermore, 
    \begin{align}\label{Eq:ArcLengthEtaRegularity2}
        \eta'(u) = \eta'(u_0) +o\big( (u-u_0)^\alpha \log^\beta|u-u_0| \big) 
    \end{align}
    as $u \rightarrow u_0$ if and only if $\overline{a_1} a_2 \in \mathbb{R}$. 
\end{lemma}
\noindent In other words, $\eta$ is more regular than the conformal parametrization if and only if $a_1$ and $a_2$ are linearly dependent as vectors in $\mathbb{R}^2$; greater symmetry in the $a_j$ leads to greater regularity in $\eta$. Note that in the statement, $\overline{z}$ is the complex conjugate of a single complex number $z$, while $\overline{\mathbb{H}}$ refers to the closure of the upper half plane.  Also, $\log^\beta(z) := \big(\log(z)\big)^\beta$.


\noindent 

\begin{proof}
    Since pre-composing with the shift $z+x_0$ does not change the coefficients in \eqref{Eq:ConformalMapExpansion}, $x_0=0$ without loss of generality.  
    
    We first note that the condition on the $a_j$ is independent of the choice of conformal map.  Indeed, if $F$ satisfies \eqref{Eq:ConformalMapExpansion} and $\tilde{F}: \mathbb{H} \rightarrow \Omega$ is also conformal with $\tilde{F}(0) = F(0)$, then $\tilde{F} = F \circ M$ for some \Mob transformation $M(z) = \frac{bz}{z+c}$ with $b,c \in \mathbb{R}$, $bc \neq 0$, and we find
    \begin{align*}
        \tilde{F}'(z) = a_1 \frac{b}{c} + a_2 \frac{b^{\alpha+1}}{c^{\alpha+1}}z^\alpha \log^\beta(z) + o\big(z^\alpha\log^\beta(z)\big), \qquad \overline{\mathbb{H}} \ni z \rightarrow 0.
    \end{align*}
    Thus $\tilde{F}$ satisfies the criterion if and only if $F$ likewise does, and so the condition may be stated in terms of any conformal map.  Note also that \eqref{Eq:ConformalMapExpansion} implies 
    \begin{align*}
        F'(z) = a_1 + a_2z^\alpha \log^\beta|z| +o\big( z^\alpha\log^\beta|z| \big), \qquad z \rightarrow 0.
    \end{align*}
    
    We proceed to construct the arc-length parametrization $\eta:(-\epsilon', \epsilon') \rightarrow B_{\epsilon}(F(z))\cap \partial \Omega$ as in part $(ii)$ of the proof of Corollary \ref{Cor:WongSharp}.  That is, we set
    \begin{align*}
        \sigma(t) := \int_{0}^t |F'(x)|dx
    \end{align*}
    and thus have $\eta(u) = F(\sigma^{-1}(u))$.  Then $\sigma^{-1}(u) = \frac{1}{|a_1|}u + o(u)$ as $u \rightarrow 0$, and we subsequently find
    \begin{align*}
        F'(\sigma^{-1}(u)) = a_1 + \frac{a_2}{|a_1|^\alpha }u^\alpha \log^\beta|u| + o\big( u^\alpha \log^\beta|u| \big), \qquad u \rightarrow 0.
    \end{align*}
    It follows that
    \begin{align*}
        |F'(\sigma^{-1}(u))| &= \sqrt{F'(\sigma^{-1}(u))\overline{F'(\sigma^{-1}(u))}}\\
        &= \sqrt{|a_1|^2 + \frac{a_1\overline{a_2} + \overline{a_1}a_2}{|a_1|^\alpha}u^\alpha \log^\beta|u| + o\big( u^\alpha \log^\beta|u| \big) }\\
        &= |\alpha_1| + \frac{a_1\overline{a_2} + \overline{a_1}a_2}{2|a_1|^{\alpha+1}}u^\alpha \log^\beta|u| + o\big( u^\alpha \log^\beta|u| \big), \qquad u \rightarrow 0.
    \end{align*}
    Combining these previous two computations yields
    \begin{align*}
        \eta'(u) &= \frac{F'(\sigma^{-1}(u))}{\sqrt{F'(\sigma^{-1}(u))\overline{F'(\sigma^{-1}(u))}}} = \frac{a_1}{|a_1|} + \frac{i \Imag(\overline{a_1}a_2)}{\overline{a_1}|a_1|^{\alpha+1}} u^\alpha \log^\beta|u| +o\big(u^\alpha \log^\beta|u|\big)
    \end{align*}
    as $u \rightarrow 0$.  The conclusions on the regularity of $\eta'$ immediately follow.
\end{proof}

{
\renewcommand{\arraystretch}{2.1}
\renewcommand{\tabcolsep}{5pt}
\begin{table}
    \centering
	\begin{tabular}{|c|c|c|c|c|}
		\hline
		\emph{} & \emph{Driver} & \makecell{\emph{Capacity}\\ \emph{parametrization}} & \makecell{\emph{Arc-length}\\\emph{parametrization}} & \makecell{\emph{Conformal}\\ \emph{map}}  \\ \hline
		\makecell{$\Gamma_1 =$ EMP geodesic pair\\ $\gamma_\theta \cup \gamma_\theta^*$} & $C^{3/2}$ & Sharply $C^{1,1}_w$ & Sharply $C^{1,1}_w$ & Sharply $C^{1,1}_w$\\ \hline
		\makecell{$\Gamma_2=$ EMW curve followed\\by vertical line (also \cite[Ex. 7.1]{LindTran})} & $C^{3/2}$ & $C^{1,1}$ & $C^{1,1}$ & Sharply $C^{1,1}_w$\\ \hline
	\end{tabular}
	\bigskip
	\caption{{\small Regularities associated to the two curves in Figure \ref{Fig:WongSharp} and Example \ref{Eg:WongSharp}.}}
	\label{Table:WongSharp}
\end{table}
}
\begin{example}\label{Eg:WongSharp}
    Let us compare the Wang-minimizer geodesic pair $\Gamma_1 = \gamma_\theta \cup \gamma_\theta^*$ with the curve $\Gamma_2$ from \cite[Example 7.1]{LindTran}, as in Figure \ref{Fig:WongSharp}.  For the former, we saw in \eqref{Eq:K_theta'} that the coefficients $a_1,a_2$ for the derivative $K_\theta'$ of the conformal map are not co-linear, and so by Lemma \ref{Lemma:ArcLengthCoeff} we know that $\eta$ has the same regularity as the conformal parametrization, namely sharply $C^{1,1}_w$.  This is, of course, what we discovered in the course of the proof of Corollary \ref{Cor:WongSharp}.
    
    For $\Gamma_2$, the non-smooth point $\zeta_2$ is $F(0)$ for the conformal map $F$ given in \cite{LindTran}, for which we explicitly compute
\begin{align*}
    F'(z) = \frac{ic\pi}{2} - ic^2 \pi z \log(z) + O(z), \qquad \overline{\mathbb{H}} \ni z \rightarrow 0,
\end{align*}
where $0 \neq c \in \mathbb{R}$.  By Lemma \ref{Lemma:ArcLengthCoeff}, the arc-length parametrization $u \mapsto \eta(u)$ is not sharply $C^{1,1}_w$ at the point corresponding to $\zeta_2$.  

In fact, we claim that $\eta$ is $C^{1,1}$.  This follows from Lind-Tran's result that the half-plane capacity parametrization $\gamma$ is $C^{1,1}$.  Indeed, write $\eta(u) = \gamma \circ \tilde{\sigma}^{-1}(u)$ as in part $(iii)$ of the proof of Corollary \ref{Cor:WongSharp}, and note by \eqref{Ineq:sigmatilde} that $\tilde{\sigma} \in C^{1,1}$.  Similarly to Lemma \ref{Lemma:C11w}, it follows that $\tilde{\sigma}^{-1} \in C^{1,1}$, and thus that the composition $\eta$ is $C^{1,1}$ as well.

Table \ref{Table:WongSharp} summarizes these two examples.

As it turns out, the $\Gamma_2$ curve, after mapping down the initial straight-line portion, will be the focus of our attention in the next section, where we will show that it solves the Loewner energy minimization problem for welding.  While we will not mention \cite[Example 7.1]{LindTran} in the sequel, we note Theorem \ref{Thm:EMW}($\ref{Thm:EMWSingleCurve}$) shows the EMW curve is the same, and also provides the explicit driving function.

\end{example}

\subsubsection{Questions}\label{Sec:SharpnessQuestions}
We close this section by highlighting two open questions on the fine properties of the regularity correspondence between $\gamma$ and $\lambda$. 
\begin{enumerate}[$(i)$]
    
    \item\label{Q:Coefficients} The driver $\lambda$ contains all the information of the curve $\gamma$.  For $\lambda \in C^{3/2}$, is it possible to see, e.g., from the coefficients of the expansion of $\dot{\lambda}$, when $\gamma$ will be sharply $C^{1,1}_w$?  Write $\lambda_j$ for the downwards driver of $\Gamma_j$, with $\Gamma_j$ as in Figure \ref{Fig:WongSharp} and Example \ref{Eg:WongSharp}.  From \eqref{Eq:WangDriver} we find
    \begin{align*}
        \dot{\lambda}_1(t) = 16 \frac{\cos(\theta)}{\sin^3(\theta)}(\tau_\theta - t)^{1/2} - \frac{896}{3} \frac{\cos^3(\theta)}{\sin^9(\theta)}(\tau_\theta - t)^{5/2} + O (\tau_\theta - t)^{9/2}, \qquad t \nearrow \tau_\theta,
    \end{align*}
    while in the normalization of $\Gamma_2$ given by Theorem \ref{Thm:EMW}$(\ref{Thm:EMWSingleCurve})$ below, with $\tau$ the time corresponding to $\zeta_2$,
    \begin{align*}
        \dot{\lambda}_2(t) = -\frac{16}{\pi} (t - \tau)^{1/2} - \frac{896}{3\pi^3} (t-\tau)^{5/2} + O (t-\tau)^{9/2}, \qquad t \searrow \tau.
    \end{align*}
    This suggests the alternating signs in $\dot{\lambda}_1$'s coefficients lead to the faster growth of $\gamma_1'$ near $\tau_\theta$.  Can this be made rigorous and classify when $\gamma \in C^{1,1}_w$?
    
    \item\label{Q:LindTran} We also note that Lind and Tran \cite[Thm. 1.1]{LindTran} extended Proposition \ref{Prop:Wong} to all $\beta>2$, although their conclusion for $\gamma$ when $\beta + 1/2 = n \in \{3,4,\ldots\}$ is that $\gamma^{(n)}$ is in the Zgymund class $\Lambda_*$, which is to say, for each $0 < t_0$ there exists $C<\infty$ such that
\begin{align*}
    |\gamma^{(n)}(t+\delta) + \gamma^{(n)}(t-\delta) - 2\gamma^{(n)}(t)| \leq C\delta
\end{align*}
for $0 < t_0 \leq t \leq T$ and $0 < \delta$.\footnote{Recall that, for functions from a compact interval to $\mathbb{C}$, $C^1 \subset C^{0,1} \subset \Lambda_* \subset \cap_{0 < \alpha < 1} C^{\alpha} \subset C$ \cite[p.57]{garnett_marshall}.}  Their proof \cite[Thm. 4.1]{LindTran}, however, shows that the modulus of continuity of $\gamma^{(n)}$ is controlled by $C\delta \log \frac{1}{\delta}$ (see \cite[Lemma 4.3]{LindTran}). Thus Lind-Tran's result is similar to Wong's in the $\lambda \in C^{n-1/2}$ case.  Our example only shows sharpness for $n=2$ (Wong's theorem).  Is the Lind-Tran extension also sharp for larger values of $n$?  Answering question $(\ref{Q:Coefficients})$ would likely produce an answer here, too.

\end{enumerate}

\section{The energy minimizer for weldings (EMW) family}\label{Sec:EMW}
We now turn our attention to the second question in the introduction: what is the infimal energy needed to weld a given $x<0<y$, and what is the nature of minimizing curves, if they exist?  Our approach is entirely parallel to that for the EMP curves in \S\ref{Sec:Wang}; we first prove existence of minimizers (Lemma \ref{Lemma:WeldingEnergyMinimizerExists}), then an ``even approach'' property for finite-energy curves (Lemma \ref{Lemma:EvenApproach}), and then collect our main results (Theorem \ref{Thm:EMW}). 

\begin{lemma}\label{Lemma:WeldingEnergyMinimizerExists}
    For any $x_0<0<y_0$, there exists a driver $\xi$ with $\xi(0)=0$ which welds $x_0$ to $y_0$ under its upwards Loewner flow and satisfies 
    \begin{align*}
        I(\xi) = \inf_{\eta \in \mathcal{W}} I(\eta),
    \end{align*}
    where $\mathcal{W}$ is the family of all drivers $\eta$ that weld $x_0$ to $y_0$, have $\eta(0)=0$, and generate a simple curve $\gamma^\eta$.
\end{lemma}
Note that the infimum would be the same if we allowed drivers in $\mathcal{W}$ that did not generate simple curves, as these drivers have infinite energy (finite-energy curves are quasi-arcs and thus simple \cite{WangReverse}).

    Conceptually, the proof is nearly identical to the argument for Lemma \ref{Lemma:AngleEnergyMinimizerExists}; we again use compactness and lower semicontinuity of energy.  A slight twist is why a limiting driver $\xi$ still welds $x_0$ and $y_0$, and for this we use a result which states uniform driver convergence implies a form of uniform welding convergence \cite[Thm. 4.3]{TVY}.
\begin{proof}
  As we note below in Lemma \ref{Lemma:CircularArc}, the orthogonal circular arc segment which welds $x_0$ to $y_0$ to its base has finite energy, and so the infimum is finite.  For $\eta \in \mathcal{W}$, let $\tau_\eta$ be the ``hitting time'' of $x_0,y_0$ under $\eta$, where $x(t), y(t)$ are the images of $x_0,y_0$ under the upwards flow generated by $\eta$.  That is, $\tau_\eta$ is the first time when $x(\tau_\eta)=\eta(\tau_\eta)=y(\tau_\eta)$.  Note that the $\tau_\eta$ are uniformly bounded over $\mathcal{W}$: as any curve $\gamma^\eta$ generated by $\eta \in \mathcal{W}$ has $\diam(\gamma^\eta) \asymp y_0-x_0$ \cite[top of p.74]{Lawler}, there exists $R>0$ such that $\gamma^\eta \subset B_R(0) \cap \mathbb{H}$ for all $\eta$, implying
    \begin{align*}
        \tau_\eta = \hcap(\gamma^\eta) \leq \hcap(B_R(0)\cap \nH) = R^2 \hcap(B_1(0) \cap \nH) = R^2
    \end{align*}
    by the monotonicity of hcap and explicit calculation. Thus, if $\{\eta_n\}$ is a sequence such that
    \begin{align}\label{Lim:MinimalEnergyEMW}
        I(\eta_n) \rightarrow \inf_{\eta \in \mathcal{W}} I(\eta) =:L,
    \end{align}
    then by flowing upwards, if necessary, with the constant driver $\eta_n(\tau)$ from the moment $\tau = \tau(\eta_n)$ that $x_0$ and $y_0$ are welded together (which adds no energy), we may assume that each $\eta_n$ is defined on the same interval $[0,T]$.  Hence $\{\eta_n\}$ is a bounded subset of $W^{1,2}([0,T])$ and so is precompact in $C([0,T])$ (recall the discussion in \S\ref{Sec:BackgroundLoewnerEnergy}).  If $\eta_{n} \rightarrow \xi$ is any sub-sequential limit, by lower semicontinuity
    \begin{align*}
       I(\xi)\leq \liminf_{n \rightarrow \infty} I(\eta_{n}) =:L,
    \end{align*}
    and thus $I(\xi) = L$ if $\xi \in \mathcal{W}$.  As noted above, $\xi$ generates a simple curve $\gamma$ since $I(\xi) < \infty$, and so we just have to ensure that $\xi$ welds $x_0$ to $y_0$.  By flowing up with the constant driver $\xi(\tau)$, if necessary, we may assume that the welding $\varphi$ for $\gamma$ is defined for $x_0$.  Similarly extending the $\xi_n$, we still have $\xi_n \rightarrow \xi$ uniformly on a fixed time interval $[0,T']$.  By further flowing up, if necessary, we have that $x_0$ is in the interior of the intervals welded by $\xi_n$ and $\xi$.  The uniform driver convergence on the resulting interval $[0,T'']$ then yields $y_0 = \varphi_{n}(x_0) \rightarrow \varphi(x_0)$ by \cite[Thm. 4.3]{TVY}.
\end{proof}

\begin{lemma}[``Even welding approach'']\label{Lemma:EvenApproach}
    Let $\xi$ be a driver with $\xi(0)=0$ and finite energy which welds $x_0 < 0 <y_0$ together at time $\tau$ under its upwards Loewner flow.  Then if $x(t)$ and $y(t)$ are the positions of $x_0$ and $y_0$ under the centered upwards flow, 
    \begin{align*}
        \lim_{t \rightarrow \tau^-}\frac{y(t)}{y(t)-x(t)} = \frac{1}{2}.
    \end{align*}
    Equivalently, $r(t):=-x(t)/y(t) \rightarrow 1$.
\end{lemma}
\begin{proof}
    We claim that there is a positive lower bound $m(r)$ for the energy to weld points with ratio $r=r(0) = -x_0/y_0 \neq 1$ that is non-decreasing in $|r-1|$. By symmetry we may suppose $r>1$.  
    
    Note that if we initially flow \emph{downwards} with the constantly-zero driver for time $t$, the images of $x_0,y_0$ have ratio $\sqrt{x_0^2+4t}/\sqrt{y_0^2+4t}$, which monotonically decreases to 1 as $t \rightarrow \infty$ (recall the explicit solution to the Loewner equation in this case, see \S\ref{Sec:BackgroundLoewner}).  So we may weld $-\sqrt{x_0^2+4t}$ to $\sqrt{y_0^2+4t}$ in the upwards flow by starting with the zero driver for time $t$ followed by $\xi$, which uses the same amount energy as $\xi$, thus showing
    \begin{align*}
        m\Big( \frac{\sqrt{x_0^2+4t}}{\sqrt{y_0^2+4t}} \Big) \leq m\Big(\frac{-x_0}{y_0}\Big). 
    \end{align*}
    We conclude $m(r)$ is non-decreasing for $r>1$. 
    
    Furthermore, $m(r)>0$ when $r \neq 1$, since minimizers exist by Lemma \ref{Lemma:WeldingEnergyMinimizerExists}, and a minimizer with zero energy can only weld symmetric points.  Hence if $\xi$ welds $x_0$ to $y_0$ with finite energy, as \begin{align*}
        m(r(t)) \leq \frac{1}{2}\int_t^\tau \dot{\xi}(s)^2ds \rightarrow 0
    \end{align*} 
    as $t \rightarrow \tau^-$, we have $r(t) \rightarrow 1$.
\end{proof}
We provide an exact formula for $m(r)$ in \eqref{Eq:EMWEnergyFormula} below.

As an aside, we note that Lemmas \ref{Lemma:WangEvenApproach} and \ref{Lemma:EvenApproach} state that points welded by finite-energy $\xi$ ``evenly'' approach $\xi$ and then move up into $\mathbb{H}$ perpendicularly from the real line.  While both properties hold for finite-energy curves, it is instructive to note that they are not equivalent, as the following example shows.\footnote{The fact that both properties hold for finite-energy curves is reminiscent of the fact that $I(\gamma) = I(-1/\gamma)$, Wang's reversibility of the Loewner energy \cite{WangReverse}.  Reversibility gives a global sense in which finite-energy curves ``look the same'' in both directions, and, informally speaking, lemmas \ref{Lemma:WangEvenApproach} and \ref{Lemma:EvenApproach} give a local sense in which this also holds.}
\begin{figure}
    \centering
    \subfloat[][]{
        \includegraphics[width=0.2\textwidth]{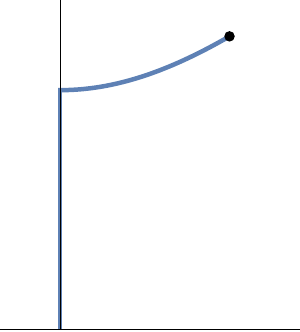}
        \label{Fig:SubfigA}
   }\qquad
        \subfloat[][]{
        \includegraphics[width=0.2\textwidth]{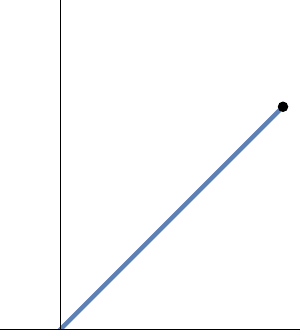}
        \label{Fig:SubfigB}
   }
    \caption{\small{A curve $\gamma$ in subfigure \ref{Fig:SubfigA} which shows the properties in lemmas \ref{Lemma:WangEvenApproach} and \ref{Lemma:EvenApproach} are not equivalent. Subfigure \ref{Fig:SubfigB}} shows the image of $\gamma$ after mapping down the vertical line segment.  While $\gamma$ satisfies the ``orthogonal approach'' property in Lemma \ref{Lemma:WangEvenApproach} for all times $\tau$, it does not satisfy the ``even welding approach'' property of Lemma \ref{Lemma:EvenApproach} at its corner.}
    \label{Fig:NotEquiv}
\end{figure}
\begin{example}\label{Eg:NotEquiv}
    For $c>0$, consider the curve $\gamma$ whose downward driver is
    \begin{align*}
        \lambda(t) = \begin{cases}
            0 & 0 \leq t \leq 1,\\
            c\sqrt{t} & 1 < t \leq 2.
        \end{cases}
    \end{align*}
That is, the base of $\gamma$ is a vertical line segment, and when this is mapped down, what remains is a line segment which meets $\mathbb{R}_{\geq 0}$ at angle $\alpha \pi \in (0,\pi/2)$ for some $\alpha = \alpha(c)$ \cite[Example 4.12]{Lawler}, as in Figure \ref{Fig:NotEquiv}.  We claim $\gamma$ satisfies the ``even angle approach'' property of Lemma \ref{Lemma:WangEvenApproach} but not the ``even welding approach'' property of Lemma \ref{Lemma:EvenApproach}. 
    
    With respect to the former, it is not hard to see that for each $0<\tau \leq 2$, 
    \begin{align*}
        \lim_{t \rightarrow \tau^-} \Arg(G_t(\gamma(\tau)) = \frac{\pi}{2}.
    \end{align*}
    This is obvious for $0 < \tau \leq 1$.  If $1 < \tau \leq 2$, set $s:=1+(\tau-1)/2$ and first map down with $G_{s}$.  The remaining curve $\tilde{\gamma}(t) := G_{s}(\gamma(t + s))$ on $0 \leq t \leq 2-s$ has driver $\tilde{\lambda}(t) = c\sqrt{t + s} - c \sqrt{s}$, which has finite energy, and so by Lemma \ref{Lemma:WangEvenApproach}, the image of $\gamma(\tau)$ must approach the imaginary axis as we continue to flow down.
    
    
    On the other hand, the property in Lemma \ref{Lemma:EvenApproach} does not hold for the points $x_0,y_0$ which weld under the upwards flow to $\gamma(1)$, the corner of the curve.  To see this, map down the vertical line segment; since the remaining curve $G_1(\gamma([1,2]))$ is a line segment with angle $\alpha \pi$ to $\mathbb{R}$, we claim that further pulling down any small portion $G_1(\gamma([1,1+\epsilon]))$ with the centered mapping down function sends the base to points $x(\epsilon) < 0 < y(\epsilon)$ that satisfy 
    \begin{align}\label{Eq:RatioBad}
        \frac{y(\epsilon)}{y(\epsilon) - x(\epsilon)} = \alpha \neq \frac{1}{2}.
    \end{align}
    One can see this explicitly; the conformal map $F:\mathbb{H} \rightarrow \mathbb{H} \bs L_\alpha$, 
    \begin{align*}
        L_\alpha: = \{ re^{i \alpha \pi} \, : \, 0 \leq r \leq \alpha^\alpha(1-\alpha)^{1-\alpha} \},
    \end{align*} 
    which satisfies $F(z) = z +O(1)$ as $z \rightarrow \infty$ is $F(z) = (z-\alpha)^\alpha(z+1-\alpha)^{1-\alpha}$ (see the construction in \cite[\S1 ``The Slit Algorithm'']{MRZip}, for instance).  This sends $\alpha-1$ and $\alpha$ to the base of the curve, which by scaling shows \eqref{Eq:RatioBad}.  Alternatively, the ratio of the harmonic measures of either side of $G_1(\gamma([1,1+\epsilon]))$ as seen from $\infty$ is always $r \neq 1$, independent of $\epsilon$.  So by conformal invariance of harmonic measure, the two intervals $[x(\epsilon),0]$, $[0, y(\epsilon)]$ one obtains upon mapping down also have the same ratio of lengths, yielding \eqref{Eq:RatioBad} again.  
\end{example}


We call the curve family whose existence is given by the follow theorem the \emph{energy minimizers for welding (EMW) family}.
\begin{theorem}\label{Thm:EMW}
    Fix $x_0<0<y_0$.
    \begin{enumerate}[$(i)$]
    \item (Uniqueness, SLE$_0$, and energy)\label{Thm:EMWEnergy} There exists a unique driving function $\xi$ with $\xi(0)=0$ which welds $x_0$ to $y_0$ at its base and minimizes the Loewner energy among all such drivers. Furthermore, $\xi$ is upwards SLE$_0(-4,-4)$ starting from $(\xi(0), V^1(0), V^2(0)) = (0, x_0, y_0)$, and satisfies
     \begin{align}\label{Eq:EMWEnergyFormula}
            I(\xi) = \frac{1}{2}\int_0^\tau \dot{\xi}(t)^2dt = -8 \log(2\sqrt{\alpha(1-\alpha)}) = -8\log \Big( \frac{2\sqrt{r}}{1+r} \Big)
        \end{align}
     where $\alpha := \frac{y_0}{y_0-x_0}$, $r:=-x_0/y_0$, and $\tau$ is the time $\xi$ takes to weld $x_0$ to $y_0$.  The driver $\xi$ is $C^\infty([0,\tau))$ and monotonic.

        \item \label{Thm:EMWDriver}(Driver and welding) When $y_0 \neq -x_0$, set $B = B(x_0,y_0):= \frac{(y_0-x_0)^3}{24|y_0+x_0|}$. The downwards driving function $\lambda = \lambda_{x_0,y_0}$ for the minimizer is explicitly
        \begin{align}\label{Eq:EMWDriverFormula}
            \lambda(t) &= -\frac{16}{\sqrt{3}}\, \sgn(y_0+x_0)t^{3/2}\left( B^{2/3} + 2 \Real \Big( \big(\sqrt{B^2-t^2} +it \big)^{2/3}\Big) \right)^{-3/2}
        \end{align}
        for $0 \leq t \leq \frac{1}{24}(x_0^2-4x_0y_0+y_0^2) =: \tau(x_0,y_0) = \tau$, where $z \mapsto z^{2/3}$ is defined with the principal branch of the logarithm.
        In particular, $\lambda(0)=0$ and $\lambda(\tau) = -\frac{2}{3}(x_0+y_0)$.
        
        When $r = -x_0/y_0 <1$, the conformal welding $\varphi:[x_0,0] \rightarrow [0,y_0]$ corresponding to $\lambda$ satisfies the implicit equation
        \begin{align}\label{Eq:EMWWeldImplicit}
            W(r,T \circ \varphi \circ T^{-1}(x)) = W(r,x)
        \end{align}
        for all $-\infty< x \leq y_0/x_0$, where
        \begin{align*}
            W(r,x) := rx + \frac{1}{x} + (1-r)\log(x),
        \end{align*}
        and $T(x) = T_{x_0,y_0}(x) := \frac{x-y_0}{x-x_0}$.  When $r>1$, one obtains a similar implicit equation by reflecting across the imaginary axis and using $\varphi^{-1}$.     
        
        \item \label{Thm:EMWSingleCurve}(Universal curve) The curve $\Gamma \subset \mathbb{H}\cup\{0\}$ satisfying 
        \begin{align}\label{Eq:EMWUniversalVariety}
            (x^2+y^2)^2 = -4xy,
        \end{align}
        continuously parametrized $t \mapsto \Gamma(t)$ so that its base is perpendicular to $\mathbb{R}$, is universal for the EMW curves, in the sense that for each ratio $r\in(0,1)$, there exists $t_r$ such that the segment $\Gamma([0,t_r])$ is an EMW curve which welds points $x_r<0<y_r$ with ratio $r = -x_r/y_r$ to its base. $\Gamma$'s downwards driving function is
        \begin{align*}
            \lambda_\Gamma(t) = -\frac{16}{\sqrt{3}}\,t^{3/2}\Big((\pi/6)^{2/3} + 2 \Real \Big( \big(\sqrt{(\pi/6)^2 -t^2} + it\big)^{2/3}\Big) \,\Big)^{-3/2}, \qquad 0 \leq t \leq \pi/6.
        \end{align*}
        Furthermore, $\Gamma^2 = \{z^2 \, : \, z \in \Gamma\}$ is the circle $x^2 + (y+1)^2 =1$.  The reflection $(x^2+y^2)^2 = 4xy$ of $\Gamma$ across the imaginary axis gives a similar universal curve $\Gamma'$ for ratios $r \in (1,\infty)$ with driver $\lambda_{\Gamma'} = -\lambda_{\Gamma}$.
        \item\label{Thm:DistinctFamilies} The EMP curve $\gamma_\theta$ coincides with an EMW curve if and only if $\theta =\pi/2$.
    \end{enumerate}
\end{theorem}
\begin{remark}\label{Remark:DualFam}
We precede the proof with several comments.
\begin{enumerate}[$(a)$] 
    \item In part $(\ref{Thm:EMWEnergy})$, $\xi$ welds $x_0$ and $y_0$ together at time $\tau$.  The meaning of welding these points ``at its base'' is that the lifetime of $\xi$ is $[0,\tau]$; there is no further curve generated after time $\tau$.  (Allowing a further curve eliminates uniqueness because we could continue flowing up with the constant driver.)   
    \item\label{Remark:RemainingEMW} Fix $x_0 < 0 <y_0$ with $-x_0 \neq y_0$, and flow up with the corresponding EMW driver $\xi$ for some time $t<\tau$.  Theorem \ref{Lemma:Inf98}($\ref{Thm:LocalCompareEMW}$) below shows that this initial portion of the curve is generally not itself an EMW curve, as we may take $\eta = \xi$.  However, the remaining curve that $\xi$ generates on $[t,\tau]$ must be a minimizer: if not, then we could replace $\xi$ with the corresponding EMW driver for  $x(t)$ and $y(t)$ and lower the energy.  Hence the symmetry of the EMW family is with respect to the ``base,'' or what remains to flow up, rather than with respect to the ``top,'' or the segment already flown up into $\mathbb{H}$. In terms of the downwards flow, this says $\gamma[0,t]$ always is an EMW curve, whereas $G_t(\gamma[0,t])$ is not. Thus it is natural to write \eqref{Eq:EMWDriverFormula} in terms of the downwards driver $\lambda$ rather than its reversal $\xi$.  Note also that this symmetry is the opposite of what we saw with the EMP curves in Remark \ref{Remark:RemainingWang}($\ref{Remark:RestOfWang}$).  
    \item The discussion around \eqref{Ineq:AlwaysNegative} below shows that $a \geq 0$ in the $\Real ((\sqrt{a}+it)^{2/3})$ term of the driver formula \eqref{Eq:EMWDriverFormula}.
\end{enumerate}
\end{remark}


To simplify notational clutter, we will interchangeably use $\xi(t)$ and $\xi_t$ and similarly for other functions of $t$.

\begin{proof}
    By Lemma \ref{Lemma:WeldingEnergyMinimizerExists} we may start with a minimizer $\xi$ for the welding problem, which is absolutely continuous since $I(\xi) =:L < \infty$.  We first show 
    \begin{align}\label{Eq:EMWDeriv}
        \dot{\xi}(t) = -4\Big( \frac{1}{x(t)} + \frac{1}{y(t)} \Big)
    \end{align}
    at all times $t$ where $\dot{\xi}$ exists and where $t$ is a Lebesgue point of $\dot{\xi}^2$, and where $x(t)$ and $y(t)$ are the images of $x_0$ and $y_0$ under the centered upwards-flow maps $H_t(z) := h_t(z)-\xi(t)$.  By \eqref{Eq:LoewnerEqUp}, we have that 
    \begin{align}\label{Eq:LoewnerEqUpCentered}
        \dot{H}_t(z) = - \frac{2}{H_t(z)} - \dot{\xi}(t)
    \end{align}
    at points $t$ of differentiability of $\xi$. By Lemma \ref{Lemma:EvenApproach}, $r(t):=-x(t)/y(t) \rightarrow 1$, and since $\xi$ minimizes energy, $r$ must (strictly) monotonically approach 1.  Indeed, if $r(t_0) = 1$ for some $t_0 < \tau$, then we must have $r(t) \equiv 1$ for all $t_0 \leq t \leq \tau$.  If $1 \neq r(t_1)=r(t_2)$ for some $t_1 < t_2$, then $\xi$ cannot be constant on $[t_1,t_2]$, because a ratio distinct from 1 changes under the constant driver.  So energy would be unnecessarily wasted on $[t_1,t_2]$, and we conclude that the claimed monotonicity holds.
    
    To be a minimizer, $\xi$ must expel as little energy as possible to move the ratio to $1$.  For $I(t) := \frac{1}{2}\int_0^t \dot{\xi}(s)^2ds$, we thus need to optimize $dr/dI$, in the sense of maximizing it when $r <1$ and minimizing it when $r>1$.

    Suppose first that $\xi$ is right-differentiable at $t=0$ and that $t=0$ is a Lebesgue point for $\dot{\xi}^2$.  Then the energy expelled on a small interval $[0,\Delta t]$ is $\frac{1}{2}\dot{\xi}^2(0)\Delta t + o(\Delta t)$, and $r$ is right-differentiable at $t=0$ with
    \begin{align}\label{Lim:drdI}
        \Delta r = \dot{r}(0)\Delta t + o(\Delta t) = \Big( \frac{2(y^2-x^2)}{xy^3} + \frac{\dot{\xi}(y-x)}{y^2} \Big)\Delta t + o(\Delta t)
    \end{align}
    by the Loewner equation \eqref{Eq:LoewnerEqUpCentered}, where $x,y$ and $\dot{\xi}$ are all evaluated at $t=0$, and $\dot{\xi}$ is the right derivative of $\xi$. We thus have
    \begin{align*}
        \frac{\Delta r}{\Delta I} \rightarrow \frac{4(y^2-x^2)}{\dot{\xi}^2 xy^3}+ \frac{2(y-x)}{\dot{\xi}y^2}
    \end{align*}
    as $\Delta t \rightarrow 0$. Calculus shows that this expression is optimized with respect to $\dot{\xi}$ when $\dot{\xi}(0)$ satisfies \eqref{Eq:EMWDeriv} at $t=0$, in the sense of yielding a unique global maximum if $r(0)<1$ and a unique global minimum if $r(0)>1$ (and satisfying $\dot{\xi}=0$ if $r=1$, as needed).  Thus any minimizer for which $\dot{\xi}(0)$ exists and where $t=0$ is a Lebesgue point of $\dot{\xi}^2$ must satisfy  \eqref{Eq:EMWDeriv} at $t=0$.
     
     More generally, let $t_0$ be a point of differentiability of $\xi$ and a Lebesgue point of $\dot{\xi}^2$.  Note that the remaining driver $\xi|_{[t_0,\tau]}$ must be an energy minimizer for welding $x(t_0)$ to $y(t_0)$, as noted above in Remark \ref{Remark:DualFam}. By the previous paragraph, we see $\dot{\xi}(t_0)$ satisfies \eqref{Eq:EMWDeriv}.  Since $\xi$ is differentiable at a.e. $t$ and a.e. $t$ is a Lebesgue point of $\dot{\xi}^2 \in L^1([0,\tau])$, we have \eqref{Eq:EMWDeriv} at a.e. $t$.
    
    Plugging this formula into \eqref{Eq:LoewnerEqUpCentered} yields that $\xi(t),x(t)$ and $y(t)$ are absolutely continuous a.e. solutions to the system of differential equations
    \begin{align}\label{EMWDiffEqSystem}
        \dot{\xi}(t) = -\frac{4}{x(t)} - \frac{4}{y(t)}, \qquad \dot{x}(t) = \frac{2}{x(t)} + \frac{4}{y(t)}, \qquad \dot{y}(t) = \frac{4}{x(t)} + \frac{2}{y(t)}
        \end{align}
    with $\xi(0)= x(0)-x_0  = y(0)-y_0=0$, $0 \leq t < \tau$.  Since each function is absolutely continuous, replacing the derivatives with these continuous expressions for all $t$ does not change the values of the functions, and we see that each function is actually $C^1$.  Then each of the right-hand sides in \eqref{EMWDiffEqSystem} is $C^1$, and so each of the three functions is at least $C^2$. Continuing to bootstrap we see the functions are smooth on $[0,\tau)$.
    
    Note also that classical solutions to \eqref{EMWDiffEqSystem} are unique: by scale invariance $\min\{-x,y\} =1$ without loss of generality, and for small times $0 \leq t \leq t_0$, $|x(t)|$ and $|y(t)|$ are therefore both bounded below since $\xi$ is $\sqrt{L}$-\Hol-1/2 continuous.  Hence the map $f(t,\xi,x,y) = (-4(x^{-1}+y^{-1}), 2x^{-1}+4y^{-1}, 4x^{-1}+2y^{-1})$ is Lipschitz in $(\xi,x,y)$ on $[0,t_0]$, giving uniqueness.  Note that uniqueness also gives that if $r(0) \neq 1$, $r(t) \neq 1$ for all $t< \tau$, and so \eqref{Eq:EMWDeriv} also shows that $\xi$ is strictly monotone.

    Note that \eqref{EMWDiffEqSystem} states
    \begin{align*}
        \dot{\xi}(t) = - \frac{4}{h_t(x_0) - \xi(t)} - \frac{4}{h_t(y_0) - \xi(t)} =: - \frac{4}{V^1(t) - \xi(t)} - \frac{4}{V^2(t) - \xi(t)},
    \end{align*}
    which says that $\xi$ is upwards SLE$_0(-4,-4)$ starting from $(\xi(0), V^1(0), V^2(0)) = (0,x_0,y_0)$.
    
    For the energy formula, we note from \eqref{Lim:drdI} and our formula \eqref{Eq:EMWDeriv} for $\dot{\xi}$ that we have the ODE
    \begin{align*}
        \frac{dI}{dr} = \frac{4y}{x}\cdot \frac{x+y}{x-y} = \frac{4}{r} \cdot\frac{1-r}{1+r} = \frac{4}{r} - \frac{8}{r+1},
    \end{align*}
    and therefore
    \begin{align*}
        I(r(t)) - 0 = 4\log\Big(\frac{r(t)}{r_0} \Big) - 8\log\Big( \frac{r(t)+1}{r_0+1} \Big).
    \end{align*}
    Since $r(t) \rightarrow 1$, sending $t \rightarrow \tau^-$ yields \eqref{Eq:EMWEnergyFormula}.
    \bigskip
    
    For the explicit formulas in $(ii)$, we first derive the hitting time $\tau$, which follows from observing from \eqref{EMWDiffEqSystem} that
    \begin{align}\label{Eq:EMWTime}
        \frac{d}{dt}((y_t-x_t)^2 -2x_ty_t) = -24
    \end{align}
    for $0 \leq t < \tau$.  Hence integrating from $0$ to $\tau-\epsilon$ and sending $\epsilon \rightarrow 0$ yields
    \begin{align*}
        - (y_0-x_0)^2 +2x_0y_0 = -24\tau,
    \end{align*}
    as claimed (recall $x_\tau = y_\tau = 0$).  Also, if $A_t = (x_t+y_t)/2$, we find $\dot{A}_t = -\frac{3}{4}\dot{\xi}_t$ and hence
    \begin{align}\label{Eq:EMWEndPoint}
        0 - \frac{1}{2}(x_0+y_0) = -\frac{3}{4}\xi_\tau -0 \qquad \text{or} \qquad \xi_\tau = \frac{2}{3}(x_0+y_0).
    \end{align}
   
    We next claim that, in addition to \eqref{Eq:EMWDeriv}, $\xi$ also satisfies the ODE
    \begin{align}\label{Eq:EMWDeriv2}
        \dot{\xi}_t = \frac{16(\xi_t - \xi_\tau)}{(\xi_t-\xi_\tau)^2 - \frac{32}{3}(\tau-t)}
    \end{align}
    for $0 \leq t < \tau$.  To obtain this, we first observe that after flowing up for some time $t_0 < \tau$, the driver $\tilde{\xi}_t:=\xi_{t+t_0} - \xi_{t_0}$ on $0 \leq t \leq \tau-t_0$ is the EMW driver for welding $x_{t_0}$ to $y_{t_0}$, as discussed in Remark \ref{Remark:DualFam}($\ref{Remark:RemainingEMW}$).  Thus, from \eqref{Eq:EMWEndPoint}, we have
    \begin{align*}
        \frac{2}{3}(x_{t_0}+y_{t_0}) = \tilde{\xi}_{\tau-t_0} = \xi_{\tau} - \xi_{t_0},
    \end{align*}
    or in other words 
    \begin{align*}
        \xi_t -\xi_\tau = -\frac{2}{3}(x_t+y_t)
    \end{align*}
    for any $0 \leq t \leq \tau$.  Note that by integrating \eqref{Eq:EMWTime} from $t$ to $\tau-\epsilon$ and sending $\epsilon \rightarrow 0$ we also have 
    \begin{align*}
        \tau - t = \frac{1}{24}\big((y_t-x_t)^2-2x_ty_t \big) = \frac{1}{24}\big((y_t+x_t)^2-6x_ty_t \big)
    \end{align*}
    for $0 \leq t \leq \tau$, and plugging in these last two formulas into the right-hand side of \eqref{Eq:EMWDeriv2} yields
    \begin{align*}
        \frac{16(\xi_t - \xi_\tau)}{(\xi_t-\xi_\tau)^2 - \frac{32}{3}(\tau-t)} = \frac{-\frac{32}{3}(x_t+y_t)}{\frac{4}{9}(x_t+y_t)^2-\frac{4}{9}\big((x_t+y_t)^2-6x_ty_t \big)} = -4 \Big( \frac{1}{x_t}+\frac{1}{y_t} \Big),
    \end{align*}
   thus verifying \eqref{Eq:EMWDeriv2} in light of \eqref{EMWDiffEqSystem}.
    
    We proceed to obtain $\xi$ by solving \eqref{Eq:EMWDeriv2}, and start by introducing the change of variables
    \begin{align*}
        \nu(t) := \xi\Big(\tau + \frac{3t}{32} \Big) - \frac{2}{3}(x_0 + y_0).
    \end{align*}
    We will momentarily assume that $-x <y$, so that $\nu$ is a non-positive, increasing function defined for $-\frac{32}{3}\tau \leq t \leq 0$, and which vanishes only at $t=0$.  In terms of $\nu$, \eqref{Eq:EMWDeriv2} says
    \begin{align*}
        \dot{\nu} = \frac{\frac{3}{2}\nu}{\nu^2+t},
    \end{align*}
    and so if we define $\mu(t) := \nu(t)^{-2/3}$, where $w \mapsto w^{-2/3}$ is the precalculus function mapping $\mathbb{R}\backslash\{0\}$ to $(0,\infty)$, we see
    \begin{align*}
        \dot{\mu} = \frac{-\mu}{\mu^{-3}+t}, \qquad \text{implying} \qquad \frac{d}{dt}\big( -\frac{1}{2} \mu^{-2} + t\mu \big) = 0,
    \end{align*}
    and thus $-\frac{1}{2} \mu(t)^{-2} + t\mu(t) = C$.  Evaluating at $t = -\frac{32}{3}\tau$ yields
    \begin{align*}
        C = -\Big( \frac{2}{3}\Big)^{1/3}\frac{(y_0-x_0)^2}{(y_0+x_0)^{2/3}},
    \end{align*}
    and we see that $\mu$ satisfies
    \begin{align}\label{Eq:muCubic}
         \mu^3 - \frac{C}{t} \mu^2 - \frac{1}{2t} =0.
    \end{align}
    The corresponding depressed cubic in $\tilde{\mu} := \mu - \frac{C}{3t}$ is
    \begin{align}\label{Eq:tmuCubic}
        \tilde{\mu}^3 - \frac{C^2}{3t^2}\tilde{\mu} - \frac{2C^3}{27t^3} - \frac{1}{2t} = 0,
    \end{align}
    which has discriminant
    \begin{align*}
        4\Big(\frac{C^2}{3t^2}\Big)^3-27 \Big(\frac{2C^3}{27t^3} + \frac{1}{2t}\Big)^2 = -\frac{2C^3}{t^4}- \frac{27}{4t^2}.
    \end{align*}
    Hence there are three real solutions $\tilde{\mu}$ to \eqref{Eq:tmuCubic} when $\frac{27}{8}t^2 < -C^3$, or
    \begin{align*}
        \frac{81}{16}t^2 < \frac{(y_0-x_0)^6}{(y_0+x_0)^2},
    \end{align*}
    and as $-\frac{32}{3}\tau \leq t < 0$, we have $81t^2/16 \leq 576 \tau^2 = (x_0^2-4x_0y_0+y_0^2)^2.$ The inequality
    \begin{align}\label{Ineq:AlwaysNegative}
        (x_0^2-4x_0y_0+y_0^2)^2 < \frac{(y_0-x_0)^6}{(y_0+x_0)^2}
    \end{align}
    is equivalent to $(3y_0-x_0)(y_0-3x_0)>0$, which is always true, and so \eqref{Eq:tmuCubic} always has three real roots.  Using Cardano's method, we see that the zeros are
    \begin{multline*}
        \tilde{\mu}_k = e^{2\pi ik/3}\sqrt[3]{\frac{1}{27t^3}\left(C^3 + \frac{27}{4}t^2 + \frac{27}{12}t\sqrt{\frac{8}{3}C^3 + 9t^2} \right)}\\
        +e^{-2\pi ik/3}\sqrt[3]{\frac{1}{27t^3}\left(C^3 + \frac{27}{4}t^2 - \frac{27}{12}t\sqrt{\frac{8}{3}C^3 + 9t^2} \right)},
    \end{multline*}
    $k =0,1,2$, where all roots are the principal branches.  Recalling that inputs $t$ are negative for $\nu$ (and hence also for $\mu$ and $\tilde{\mu}$), and also that $\frac{8}{3}C^3 +9t^2<0$ by \eqref{Ineq:AlwaysNegative}, this simplifies to
    \begin{align*}
        \tilde{\mu}_k = -\frac{2}{3t}\Real\Bigg(e^{2\pi ik/3}\sqrt[3]{-C^3 - \frac{27}{4}t^2 - \frac{27}{12}ti\sqrt{-\frac{8}{3}C^3 - 9t^2} }\,\Bigg).
    \end{align*}
    Near $t=0^-$, the three solutions $\mu_k$ for $\mu(t)=\tilde{\mu}(t) + \frac{C}{3t}$ thus satisfy
    \begin{align*}
        \mu_k = \frac{2C}{3t}\Big(\cos\Big( \frac{2\pi k}{3} \Big) + \frac{1}{2} \Big) + O(1).
    \end{align*}
    Since $\mu(t) = \nu(t)^{-2/3} \rightarrow +\infty$ as $t \rightarrow 0^-$, our desired solution is when $k=0$, yielding 
    \begin{align*}
        \mu(t) = \frac{C}{3t}-\frac{2}{3t}\Real\Bigg(\sqrt[3]{-C^3 - \frac{27}{4}t^2 - \frac{27}{12}ti\sqrt{-\frac{8}{3}C^3 - 9t^2} }\,\Bigg).
    \end{align*}
    (Since $\Big|\sqrt[3]{-C^3 - \frac{27}{4}t^2 - \frac{27}{12}ti\sqrt{-\frac{8}{3}C^3 - 9t^2} }\Big| = -C$, it is also easy to see that this solution satisfies $\mu(t)>0$, as needed.)  Undoing our changes of variable, we have
    \begin{align*}
        \xi(t) = \frac{2}{3}(x_0+y_0) -  \mu^{-3/2}\Big(\frac{32}{3}(t-\tau)\Big),
    \end{align*}
    and using \eqref{Eq:EMWEndPoint}, we thus see that the downwards driver is $\lambda(t) = \xi(\tau-t)-\xi(\tau) = -\mu^{-3/2}(-32t/3)$.  Following arithmetic, we thus have
    \begin{align*}
        \lambda(t) &= -128\sqrt{3}\,t^{3/2}\Bigg(\frac{(y_0-x_0)^2}{(y_0+x_0)^{2/3}}+2\Real \sqrt[3]{\frac{(y_0-x_0)^6}{(y_0+x_0)^2}-1152t^2+48it\sqrt{\frac{(y_0-x_0)^6}{(y_0+x_0)^2}-576t^2}} \,\Bigg)^{-3/2}\\
        &= -\frac{16}{\sqrt{3}}\,t^{3/2}\Bigg(\frac{(y_0-x_0)^2}{24^{2/3}(y_0+x_0)^{2/3}}+2\Real \sqrt[3]{\bigg( \sqrt{\frac{(y_0-x_0)^6}{576(y_0+x_0)^2}-t^2}+it  \bigg)^2}\,\Bigg)^{-3/2}.
    \end{align*}
    Since the complex number $w=\sqrt{(y_0-x_0)^6/576(y_0+x_0)^2-t^2}+it$ is always in the first quadrant, we may interchange the order of the square and principal cube root and simply write $\sqrt[3]{w^2} = w^{2/3}$. 
    
    We assumed $-x_0<y_0$, and by symmetry the driver $\lambda_{x_0,y_0}(t)$ for $y_0<-x_0$ is $-\lambda_{-y_0,-x_0}(t)$.  We conclude that the proposed formula \eqref{Eq:EMWDriverFormula} holds in either case.
    
    We prove the implicit equation \eqref{Eq:EMWWeldImplicit} for the welding below at the end of our proof for part $(\ref{Thm:EMWSingleCurve})$.
    \bigskip

        For $(\ref{Thm:EMWSingleCurve})$, we start by arguing for the existence of a (initially non-explicit) universal curve $\Gamma$ and compute its driver.  Note that if we have two ratios $0 < r_1 < r_2<1$ and we flow up with an energy-minimizing driver $\xi_1$ for $r_1$ (defined uniquely up to scaling), by Lemma \ref{Lemma:EvenApproach} the ratio $-x(t)/y(t)$ of its welding endpoints will monotonically approach 1 and hence be $r_2$ at some time $t_2$.  At this point, as explained in Remark \ref{Remark:DualFam}($\ref{Remark:RemainingEMW}$), the remaining curve $\gamma_{r_2}$ generated by $\xi_1$ on $[t_2,\tau_1]$ is the energy-minimizing curve for $r_2$ (again, defined uniquely up to scaling).  Thus the resulting curve $\gamma_{r_1}$ for $r_1$ contains $\gamma_{r_2}$ as a sub-curve.  Similarly, by starting with lower ratios $0 < r < r_1$ and re-scaling the curve $\gamma_r$ to some $c_r\gamma_r$ so that the sub-curve of $c_r\gamma_r$ corresponding to ratio $r_1$ is $\gamma_{r_1}$, we see that we obtain a simple curve $\bigcup_{0<r<1} c_r\gamma_r$, where each $c_r\gamma_r$ is a curve generated by an EMW driver with welding endpoints of ratio $r$, and where 
        \begin{align}\label{Inclusion:UniversalEMW}
            c_{r_2}\gamma_{r_2} \subset c_{r_1}\gamma_{r_1} \qquad \text{whenever} \qquad r_1 < r_2.
        \end{align}
        Hence such a curve $\Gamma$ exists, and is defined uniquely up to scaling.  We proceed to pick the representative among all scaled versions that will yield the formula \eqref{Eq:EMWUniversalVariety}.
        
        Note that the inclusion \eqref{Inclusion:UniversalEMW} shows that the downwards driving function $\lambda_{r_1}$ of $c_{r_1}\gamma_{r_1}$ is simply an extension of the driver $\lambda_{r_2}$ for  $c_{r_2}\gamma_{r_2}$, and hence the driver $\lambda_{\Gamma}$ for $\Gamma$ is the limit of the $\lambda_r$'s.  Begin by choosing the nested family $\{c_r\gamma_r\}_{0 < r < 1}$ such that right-end of the welding interval $y_r$ for $c_r\gamma_r$ satisfies
        \begin{align}\label{Lim:EMWUNiversaly}
            y_r \rightarrow 2\sqrt{\pi} \qquad \text{ as } \qquad r \rightarrow 0,
        \end{align}
         and set $\Gamma:= \bigcup_{0 < r <1}c_r\gamma_r$.\footnote{Such a choice of $\gamma_r$ is possible.  For instance, as $r$ decreases to zero, we may re-scale the curves so that the right endpoint $y_r$ is always $2\sqrt{\pi}$.  In the end the curve does not degenerate to the origin, since the diameter of each curve is comparable to $2\sqrt{\pi}$ \cite[Prop.4.4]{LMR}.  Thus any choice of $\Gamma$ is bounded, and we may choose the scale for the limiting curve such that \eqref{Lim:EMWUNiversaly} holds.}    Writing $\lambda(t) = \lambda(x,y,t)$, we have 
        \begin{align}
            \lambda_{\Gamma}(t) = \lim_{r \rightarrow 0} \lambda(x_r,y_r,t) &= \lambda(0,2\sqrt{\pi},t) \notag \\
            &= -\frac{16}{\sqrt{3}}\,t^{3/2}\Big((\pi/6)^{2/3} + 2 \Real \big(\sqrt{(\pi/6)^2 -t^2} + it\big)^{2/3} \,\Big)^{-3/2} \label{Eq:EMWUniversalDriver}
        \end{align}
        for $0 \leq t \leq \pi/6 = \tau(0,2\sqrt{\pi})$, where the second equality is by the continuity in \eqref{Eq:EMWDriverFormula} at $(x,y) = (0,2\sqrt{\pi})$.  Thus the claimed driver formula holds.
        
        Now let $\tilde{\Gamma}$ be the intersection of $\{\, (x,y) \in \mathbb{R}^2 \;: \; (x^2+y^2)^2 = -4xy\,\}$ with $\mathbb{H}\cup\{0\}$, parametrized by half-plane capacity.  We wish to show $\tilde{\Gamma}$ has \eqref{Eq:EMWUniversalDriver} for its driving function.  To do so, we explicitly compute the conformal map $F_t$ from $\mathbb{H}$ to the complement of a portion $\tilde{\Gamma}_t$ of $\tilde{\Gamma}$.  Appropriately normalizing $F_t$ at infinity, we will read off its driver and see that it is  \eqref{Eq:EMWUniversalDriver}.
\begin{figure}
    \centering
    \includegraphics[scale=0.12]{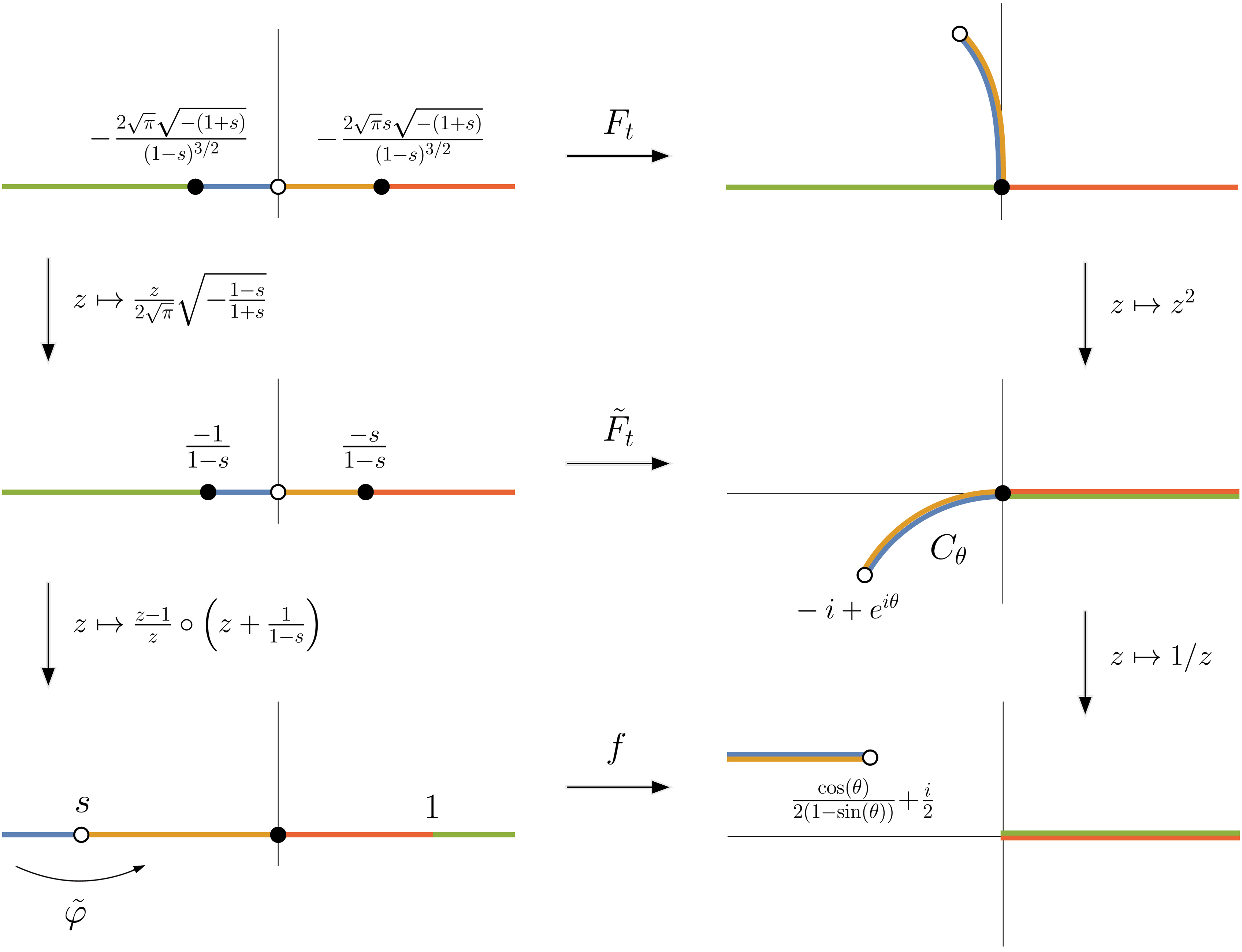}
    \caption{\small Constructing the conformal map for the EMW curves.}
    \label{Fig:EMWUniversalConstruction}
\end{figure}  

First, note that if $z = u+iv \sim (u,v) \in \tilde{\Gamma}$, then for $z^2 \sim (u^2 -v^2, 2uv) =: (x,y)$, we have
        \begin{align*}
            x^2 + (y+1)^2 = (u^2+v^2)^2+4uv + 1 = 1,
        \end{align*}
        and thus $\tilde{\Gamma}^2 \subset \{x^2+(y+1)^2 =1\}$.  By continuity the two sets are identical.  As circles are more tractable, we find $F_t$ by first computing the conformal map $\tilde{F}_t:\mathbb{H} \rightarrow \mathbb{C} \bs C_\theta$, where for $\pi/2 < \theta < 5\pi/2$,
        \begin{align*}
            C_\theta := \{\, -i + e^{i \alpha} \; : \; \pi/2 < \alpha \leq \theta \,\} \cup \mathbb{R}_{\geq 0},
        \end{align*}
        and $\theta = \theta(t)$.  See Figure \ref{Fig:EMWUniversalConstruction}.  Note that $1/C_\theta$ consists of two half-infinite horizontal line segments, 
        \begin{align*}
            \Big\{\, \frac{\cos(\alpha)}{2(1-\sin(\alpha))} + \frac{i}{2} \; : \; \frac{\pi}{2} < \alpha \leq \theta \,\Big\} \cup \mathbb{R}_{\geq 0}.
        \end{align*}
        For a parameter $s = s(\theta)<0$, consider the conformal map $f = f_\theta$ from $\mathbb{H}$ to $\mathbb{C} \bs (1/C_\theta)$ which takes $(-\infty,s]$ to the top half of the upper horizontal line, $[s,0)$ to its lower side, and then $(0,1]$ and $[1,\infty)$ to the bottom and top halves of $\mathbb{R}_{\geq 0}$, respectively.  That is, $f$ ``folds'' $(-\infty,s]\cup[s,0)$ and $(0,1]\cup[1,\infty)$ to the upper and lower line segments, respectively. The idea behind the Schwarz-Christoffel method (see \cite{MarshallComplex}, for instance) allows us to compute $f$, as we observe that $f'(x)$ is defined and has constant argument on each of the four intervals $(-\infty,s), (s,0), (0,1), (1,\infty)$.  Futhermore, as $x$ increases, $f'$ has changes of argument of $-\pi, 2\pi$ and $-\pi$, respectively, at the interfaces of the intervals (the $2\pi$ change is clear through considering small half circles centered at $z=0$).  Hence by Lindel\"{o}f's maximum principle \cite{MarshallComplex},
        \begin{align*}
            \Imag \log f'(z) &= \arg(z-1) - 2 \arg(z) +\arg(z-s) =\Imag \log \frac{(z-1)(z-s)}{z^2} + C_1,
        \end{align*}
        and so $f'(z) = C_2(z-1)(z-s)/z^2$, whence integration yields $f(z) = C_2\big(z - \frac{s}{z} - (1+s)\log(z) \big) + C_3$.  Since $f(1)=0$ (recall $f$ ``folds'' there), we see $C_3 = -C_2(1-s)$, and as $\Imag f(x) = 1/2$ for $x<0$, we find $C_2 = -1/(2\pi(1+s))$, yielding
        \begin{align}\label{Eq:EMWUniversalf}
            f(z) = \frac{-1}{2\pi(1+s)}\Big(z - \frac{s}{z} - (1+s)\log(z) \Big) + \frac{1-s}{2\pi(1+s)}.
        \end{align}
        Upon reciprocation back to $\mathbb{C} \bs C_\theta$, we want a map from $\mathbb{H}$ which fixes infinity and sends zero to the tip $-i + e^{i\theta}$, and we thus consider
        \begin{align}
            \tilde{F}_t(z) :=& \frac{1}{f(z)}\circ \frac{z-1}{z} \circ \Big( z+ \frac{1}{1-s} \Big) \notag\\
            =& 2\pi \frac{1+s}{1-s} \bigg(\frac{s-(1-s)^2z^2}{\big((1-s)z+s \big) \big((1-s)z+1 \big)}+ \frac{1+s}{1-s}
            \log \big( \frac{(1-s)z+s}{(1-s)z+1} \big) + 1 \bigg)^{-1} \label{Eq:TildeFt}\\
            =& -4\pi \frac{1+s}{1-s}z^2 - \frac{16\pi}{3} \Big(\frac{1+s}{1-s}\Big)^2 z - \frac{2\pi(1+s)(5s^2 + 28s+5)}{9(1-s)^3} + O(z^{-1}), \qquad z \rightarrow \infty. \notag
        \end{align}
    The expansion of the square $F_t^2$ of the centered mapping up function $F_t = g_t^{-1}(z+\lambda_t)$ for $\sqrt{C_\theta}$ is $z^2 + O(z)$ as $z \rightarrow \infty$ (here the square root for  $\sqrt{C_\theta}$ has branch cut along $\mathbb{R}_{\geq 0}$, sending the semi-circle to $\tilde{\Gamma}_t$).  Noting that $\tilde{F}_t$ maps $(-1/(1-s),0)$ to the inside of the circular arc and $(0,-s/(1-s))$ to the outside, we must have
    \begin{align*}
        \frac{1}{1-s} < \frac{-s}{1-s} \qquad \Leftrightarrow \qquad -1>s
    \end{align*}
    since the harmonic measure of the inside of the circle, as seen from $x\ll 0$, is less than that of the outside.  Thus $-4\pi(1+s)/(1-s)>0$, and we pre-compose with the appropriate dilation to obtain
    \begin{align*}
        F_t^2(z) &= \tilde{F}_t \circ \frac{1}{2\sqrt{\pi}} \sqrt{-\frac{1-s}{1+s}} \, z\\
        &= z^2 - \frac{8 \sqrt{\pi}}{3}\Big(- \frac{1+s}{1-s} \Big)^{3/2} z - 2\pi \frac{1+s}{(1-s)^3}(5s^2 + 28s+5) + O(z^{-1}), \qquad z \rightarrow \infty.
    \end{align*}
    Taking the square root, we thus have that the centered mapping-up function $F_t:\mathbb{H} \rightarrow \mathbb{H} \bs \tilde{\Gamma}_t$ is
    \begin{align*}
        F_t(z) = \sqrt{F_t^2(z)} = z - \frac{4\sqrt{\pi}}{3} \Big(- \frac{1+s}{1-s} \Big)^{3/2} + \frac{\pi}{3}\,\frac{1+s}{(1-s)^3}(s^2-4s+1) z^{-1} + O(z^{-2}), \qquad z \rightarrow \infty.
    \end{align*}
    Note that $F_t$ maps $0$ to the tip of the curve, and $x_t<0<y_t$ to the base, where
    \begin{align*}
        x_t = \frac{-2\sqrt{\pi}\sqrt{-(1+s)}}{(1-s)^{3/2}}, \qquad y_t = \frac{-2\sqrt{\pi} s \sqrt{-(1+s)}}{(1-s)^{3/2}},
    \end{align*}
    and hence the ratio $r = r_t := -x/y = -1/s<1$.  The expansion in terms of this (more natural) parameter $r$ is
    \begin{align*}
        F_t(z) = z - \frac{4\sqrt{\pi}}{3} \Big( \frac{1-r}{1+r} \Big)^{3/2} - \frac{\pi}{3}\,\frac{1-r}{(1+r)^3}(r^2+4r+1) z^{-1} + O(z^{-2}), 
    \end{align*}
    from which we see the Loewner time of $\tilde{\Gamma}_t$ is 
    \begin{align}\label{Eq:TildeTau}
        t = t(r) = \frac{\pi}{6}\,\frac{1-r}{(1+r)^3}(r^2+4r+1) < \frac{\pi}{6},
    \end{align}
    and the ending driver position is 
    \begin{align}\label{Eq:TildeLambda}
        \tilde{\lambda}(t) = - \frac{4\sqrt{\pi}}{3} \Big( \frac{1-r}{1+r} \Big)^{3/2}.
    \end{align}
    We proceed to solve for $r$ in terms of $t$ in \eqref{Eq:TildeTau} and then show  $\tilde{\lambda}(r(t))$ is the same as our other formula \eqref{Eq:EMWUniversalDriver}.
        
    From \eqref{Eq:TildeTau} we see that $r^3 + 3r^2 +3\alpha r + \alpha =0$ for $\alpha = \frac{t-\pi/6}{t+\pi/6}$, for which the corresponding depressed cubic in $q:=r+1$ is
    \begin{align}\label{Eq:TimeCubic}
        q^3-3\beta q + 2\beta =0, \qquad \beta = \frac{\pi/3}{t+\pi/6}
    \end{align}
    which has discriminant $108\beta^3 -108\beta^2 >0$ by \eqref{Eq:TildeTau}, and hence three real solutions.  As $r \nearrow 1$, $t \searrow 0$, and at $t=0$ we have $q^3-3\beta q + 2\beta = (q-2)(q^2+2q-2)$, for which $q=2 \sim r=1$ is the largest solution.  By continuity (of roots of $q$ with respect to $t$) we thus seek the largest solution to \eqref{Eq:TimeCubic}. By Cardano's method the solutions are
    \begin{align*}
        q = e^{2\pi i k/3} \sqrt[3]{-\beta + i\beta \sqrt{\beta-1}}+ e^{-2\pi i k/3}\sqrt[3]{-\beta -i\beta \sqrt{\beta-1}},
    \end{align*}
    $k =0,1,2$, where all roots are principal branches.  From the $t=0$ case we see we need $k=0$, and thus find
    \begin{align*}
        r(t) &= -1 + \sqrt[3]{\frac{\pi/3}{\pi/6 + t}} \bigg( \sqrt[3]{-1 + i \sqrt{\frac{\pi/6 - t}{\pi/6 + t}}} + \sqrt[3]{-1 - i \sqrt{\frac{\pi/6 - t}{\pi/6 + t}}}\, \bigg)\\
        &= -1 + \frac{\sqrt[3]{\pi/3}}{\sqrt{\pi/6+t}}\Big( \sqrt[3]{-\sqrt{\pi/6+t}+i\sqrt{\pi/6-t}} + \sqrt[3]{-\sqrt{\pi/6+t}-i\sqrt{\pi/6-t}} \, \Big).
    \end{align*}
    Writing $\mu_1 := \sqrt{\pi/6+t}$, $\mu_2:=\sqrt{\pi/6-t}$ and noting $\sqrt[3]{\pi/3} = \sqrt[3]{-\mu_1+i\mu_2}\sqrt[3]{-\mu_1-i\mu_2}$, we have that our expression \eqref{Eq:TildeLambda} for the driver becomes
    \begin{align*}
        \tilde{\lambda}(t) = - \frac{4}{\sqrt{3}}\Bigg( \frac{2\mu_1 - \sqrt[3]{-\mu_1+i\mu_2}^2\sqrt[3]{-\mu_1-i\mu_2}-\sqrt[3]{-\mu_1+i\mu_2}\sqrt[3]{-\mu_1-i\mu_2}^2}{\sqrt[3]{-\mu_1+i\mu_2} + \sqrt[3]{-\mu_1-i\mu_2}} \Bigg)^{3/2}.
    \end{align*}
    To show this is equivalent to \eqref{Eq:EMWUniversalDriver}, we first observe
    \begin{align*}
        \sqrt[3]{-\mu_1+i\mu_2}^2 &= e^{i2\pi/3}\sqrt[3]{(-\mu_1+i\mu_2)^2}\\
        &= 2^{1/3}e^{i2\pi/3}\sqrt[3]{t-i\sqrt{(\pi/6)^2-t^2}}\\
        &= 2^{1/3}i \sqrt[3]{\sqrt{(\pi/6)^2-t^2} +it},
    \end{align*}
    and so, by conjugating, 
    \begin{align*}
        \sqrt[3]{-\mu_1-i\mu_2}^2 = -2^{1/3}i \sqrt[3]{\sqrt{(\pi/6)^2-t^2} -it}.
    \end{align*}
    The universal driver formula \eqref{Eq:EMWUniversalDriver} is thus
    \begin{align*}
         \lambda_{\Gamma}(t) &= -\frac{16}{\sqrt{3}}\,t^{3/2}\Big((\pi/6)^{2/3} + 2 \Real \big(\sqrt{(\pi/6)^2 -t^2} + it\big)^{2/3} \,\Big)^{-3/2}\\
         &= -\frac{32}{\sqrt{3}}\,t^{3/2}\Big((\pi/3)^{2/3} -\sqrt[3]{-\mu_1-i\mu_2}^4 - \sqrt[3]{-\mu_1+i\mu_2}^4 \,\Big)^{-3/2}\\
         &= -\frac{32}{\sqrt{3}}\,t^{3/2}\Big(\sqrt[3]{-\mu_1+i\mu_2}^2\sqrt[3]{-\mu_1-i\mu_2}^2 -\sqrt[3]{-\mu_1-i\mu_2}^4 - \sqrt[3]{-\mu_1+i\mu_2}^4 \,\Big)^{-3/2}
    \end{align*}
    Writing $\nu:= \sqrt[3]{-\mu_1+i\mu_2} = \sqrt[3]{-\sqrt{\pi/6 +t}+i\sqrt{\pi/6-t}}$, we thus wish to show
    \begin{align*}
        -\frac{32}{\sqrt{3}}\,t^{3/2}\big( \nu^2\Bar{\nu}^2 - \Bar{\nu}^4-\nu^4 \big)^{-3/2} = - \frac{4}{\sqrt{3}}\Big( \frac{2\mu_1 - \nu^2 \Bar{\nu} - \nu \Bar{\nu}^2}{\nu + \Bar{\nu}} \Big)^{3/2}
    \end{align*}
    or
    \begin{align*}
        8t^{3/2} = \Big(\frac{( \nu^2\Bar{\nu}^2 - \Bar{\nu}^4-\nu^4 )
    (2\mu_1 - \nu^2 \Bar{\nu} - \nu \Bar{\nu}^2)}{\nu + \Bar{\nu}}\Big)^{3/2},
    \end{align*}
    which follows from cancellations after expanding the numerator on the right-hand side.  Since the driving function uniquely identifies the curve, we have proven that the curve described by \eqref{Eq:EMWUniversalVariety} is the universal EMW curve.

    \bigskip
    
    With the conformal maps in hand, we return to prove the implicit relation \eqref{Eq:EMWWeldImplicit} for the welding.  Indeed, from our computation \eqref{Eq:EMWUniversalf} above, we see that the welding $\tilde{\varphi}$ generated by $f$ for the EMW curve $\gamma$ satisfies $W(r, \tilde{\varphi}(x)) = W(r,x)$ for all $-\infty<x \leq s = y_0/x_0$ (the identifications established by $f$ do not change when we take the square root of the reciprocal of the image of $f$ to obtain $\gamma$, as these are conformal maps). See Figure \ref{Fig:EMWUniversalConstruction}.  However, $\tilde{\varphi} = T \circ \varphi \circ T^{-1}$, where $T$ is the \Mob automorphism of $\mathbb{H}$ sending the triple $(x_0,0,y_0)$ to $(-\infty, s,0)$.  In terms of our maps in Figure \ref{Fig:EMWUniversalConstruction}, we have
    \begin{align*}
        T(x) = \frac{x-1}{x} \circ \Big(x + \frac{1}{1-s} \Big) \circ \frac{x}{2\sqrt{\pi}}\sqrt{- \frac{1-s}{1+s}} \circ \frac{-2\sqrt{\pi}\, x \, \sqrt{-(1+s)}}{x_0(1-s)^{3/2}} = \frac{x-y_0}{x-x_0},
    \end{align*}
    as claimed.
    
    \bigskip
    
    Lastly, for part $(\ref{Thm:DistinctFamilies})$, it is clear that if $\theta=\pi/2$, then the EMP curve is the vertical line segment $[0,i]$, which is the same as the EMW curve for $x=-1,y=1$.  Both use zero energy.  For $\theta \neq 0$, we see from \eqref{Eq:EMWDriverFormula} that near $t=0$, the EMW driver satisfies $\lambda(t) = Ct^{3/2} + O(t^{5/2})$ for some $C      \neq 0$.  However, from \eqref{Eq:WangDriverDown} we see that the downwards Wang driver $\lambda_\theta$ is always smooth near $t=0$ (no radicands vanish at $t=0$), and so the drivers are never identical.
\end{proof}


\section{Energy comparisons}\label{Sec:Compare}
In this section we explore how the close sufficiently-smooth curves are to be locally energy minimizing, and how the energies of our two minimizing families compare.

\begin{theorem}\label{Lemma:Inf98}
Let $\epsilon>0$.
\begin{enumerate}[$(i)$]
    \item \label{Thm:LocalCompareWang} If $\gamma$ is any curve driven by $\lambda \in C^{3/2+\epsilon}([0,T))$ with $\lambda(0)=0$, $\dot{\lambda}(0) \neq 0$, then
    \begin{align*}
        \lim_{\delta \rightarrow 0^+} \frac{I(\gamma[0,\delta])}{I(\text{Wang   minimizer from 0 to }\gamma(\delta))} = \frac{9}{8}.
    \end{align*}

    \item \label{Thm:LocalCompareEMW} Let $\eta \in C^{3/2+\epsilon}([0,T))$ be a driver with $\eta(0)=0$, $\dot{\eta}(0) \neq 0$, and $\tilde{\gamma}[0,\delta]$ the curve generated by $\eta$ by means of the upwards Loewner flow on $[0,\delta]$.  Let $u(\delta)<0<v(\delta)$ be the points welded to the base of $\tilde{\gamma}[0,\delta]$ at time $\delta$.  Then
    \begin{align*}
       \lim_{\delta \rightarrow 0^+} \frac{I(\eta[0,\delta])}{I(\text{EMW minimizer for }u(\delta), v(\delta))} = \frac{9}{8}.
    \end{align*}
\end{enumerate}
\end{theorem}
Our proof will use infinitesimal expansions for the parametrized curve and the points welded together.  The former says that for a driver $\lambda \in C^{3/2+\epsilon}$, the associated curve $\gamma = \gamma^\lambda$ satisfies
    \begin{align}\label{Eq:InfinitesimalCurve}
        \gamma(\delta) = 2i\sqrt{\delta} + \frac{2}{3}\dot{\lambda}(0)\delta - \frac{i}{18}\dot{\lambda}(0)^2\delta^{3/2} + O(\delta^{3/2 + \epsilon})
    \end{align}
    as $\delta \rightarrow 0$ \cite[Prop. 6.2]{LindTran}.  The latter is a similar expansion for the conformal welding, and says that for a driver $\eta \in C^{3/2+\epsilon}$, the points $x(\delta) < 0 < y(\delta)$ welded together at time $\delta$ by the upwards flow generated by $\eta$ are
    \begin{align}\label{Eq:InfinitesimalWelding}
    \begin{split}
		x(\delta) &= -2 \sqrt{\delta} + \frac{2}{3}\dot{\xi}(0)\delta - \frac{1}{18}\dot{\xi}(0)^2 \delta^{3/2} + O(\delta^{3/2+\epsilon}),\\
		y(\delta) &= 2 \sqrt{\delta} + \frac{2}{3}\dot{\xi}(0)\delta + \frac{1}{18}\dot{\xi}(0)^2 \delta^{3/2} + O(\delta^{3/2+\epsilon})
	\end{split}
    \end{align}
    as $\delta \rightarrow 0$ \cite[Thm. 3.5]{Mesikepp}.

\begin{proof}
     The energy used by $\gamma[0,\delta]$ is $\frac{1}{2}\dot{\lambda}(0)^2\delta + o(\delta)$ as $\delta \rightarrow 0$.  By symmetry we may assume $\dot{\lambda}(0)>0$, and thus by \eqref{Eq:InfinitesimalCurve} the angle of the tip of the curve is 
    \begin{align*}
        \Arg(\gamma(\delta)) &= \frac{\pi}{2} - \arctan \left( \frac{\frac{2}{3}\dot{\lambda}(0)\delta  + O(\delta^{3/2+\epsilon})}{2\sqrt{\delta} - \frac{1}{18}\dot{\lambda}(0)^2 \delta^{3/2}+ O(\delta^{3/2+\epsilon})} \right) = \frac{\pi}{2}-\frac{1}{3}\dot{\lambda}(0)\sqrt{\delta} +O(\delta^{1+\epsilon}).
    \end{align*}
    Noting that $-8\log(\sin(\theta)) = 4(\theta-\pi/2)^2 + O(\theta - \pi/2)^4$, we see that the minimal-energy curve through this angle has energy
    \begin{align*}
        \frac{4}{9}\dot{\lambda}(0)^2\delta + O(\delta^{3/2 + \epsilon}),
    \end{align*}
    yielding the ratio
    \begin{equation*}
        \frac{\frac{1}{2}\dot{\lambda}(0)^2\delta + o(\delta)}{\frac{4}{9}\dot{\lambda}(0)^2\delta + O(\delta^{3/2+\epsilon})} \rightarrow \frac{9}{8}.
    \end{equation*}
    
    For the welding claim, we use the expansions in \eqref{Eq:InfinitesimalWelding} to see that the ratio $\alpha(\delta) := \frac{y(\delta)}{y(\delta)-x(\delta)}$ of the welded points $x(\delta)$ and $y(\delta)$ is
    \begin{align*}
        \alpha(\delta) = \frac{1}{2}+\frac{1}{6}\dot{\eta}(0)\sqrt{\delta} + O(\delta^{1+\epsilon}).
    \end{align*}
    Recalling from \eqref{Eq:EMWEnergyFormula} that the minimal energy for ratio $\alpha$ is
    \begin{align*}
        -4 \log(4\alpha(1-\alpha)) = 16(\alpha- 1/2)^2 + O(\alpha- 1/2)^4,
    \end{align*}
    we thus see the minimal energy to weld $x(\delta)$ and $y(\delta)$ is
    \begin{align*}
        \frac{4}{9}\dot{\xi}(0)^2\delta + O(\delta^{3/2+\epsilon}).
    \end{align*}
    As the actual energy is $\frac{1}{2}\dot{\xi}(0)^2\delta + o(\delta)$, we find the same asymptotic ratio.
\end{proof}

Lastly,  we wish to compare the energy usage between the two minimizing families themselves.  We have the following asymptotic results.  
\begin{theorem}\label{Thm:EMWWang}
     Let $x_r$ and $y_r$ satisfy $x_r<0<y_r$ with $-x_r/y_r=r$.  There is a  unique EMP curve which welds $x_r$ to $y_r$ at its base, and as $r \rightarrow 1$,
    \begin{align*}
         I(\text{EMW $\gamma$ welding $x_r$ and $y_r$ to its base}) = O (r-1)^2  = I(\text{Wang $\gamma$ welding $x_r$ and $y_r$ to its base}).
    \end{align*}
    Furthermore,
    \begin{align}\label{Lim:EnergySameBase}
        \lim_{r \rightarrow 1} \frac{I(\text{EMW $\gamma$ welding $x_r$ and $y_r$ to its base})}{I(\text{Wang $\gamma$ welding $x_r$ and $y_r$ to its base})} = \Big(\frac{\pi}{4}\Big)^2.
    \end{align}
    
    For each $\theta \in (0,\pi)$ there is a unique EMW curve with tip $e^{i\theta}$, and as $\theta \rightarrow \pi/2$,
    \begin{align*}
        I(\text{EMW $\gamma$ with tip at $e^{i\theta}$}) = O \Big(\theta - \frac{\pi}{2}\Big)^2 = I(\text{Wang $\gamma$ with tip at $e^{i\theta}$}).
    \end{align*}
    Furthermore,
    \begin{align}\label{Lim:EnergySameTip}
        \lim_{\theta \rightarrow \pi/2} \frac{I(\text{EMW $\gamma$ with tip at $e^{i\theta}$})}{I(\text{Wang $\gamma$ with tip at $e^{i\theta}$})} = \Big(\frac{4}{\pi}\Big)^2.
    \end{align}
\end{theorem}
\noindent See Remark \ref{Remark:Reciprocals} after the proof for a discussion of the fact that the constants in \eqref{Lim:EnergySameBase} and \eqref{Lim:EnergySameTip} are reciprocals.  We note in Remark \ref{Remark:EnergiesNotAlwaysSame} that these ratios do not hold for all $\theta$ and all $r$.
\begin{proof}
    In Theorem \ref{Thm:Wang}($\ref{Thm:WangDriver}$) we saw that the welding $\varphi_\theta$ for the EMP curve to $e^{i\theta}$ satisfies $\varphi_\theta:[x_\theta,0] \rightarrow [0,y_\theta]$, where
    \begin{align*}
        x_\theta = -\sqrt{\frac{\sin ^3(\theta )}{\sin (\theta )-\theta  \cos (\theta )}} \qquad \text{and} \qquad y_\theta = \sqrt{\frac{\sin ^3(\theta )}{\sin (\theta )-\theta  \cos (\theta )+\pi  \cos (\theta )}}.
    \end{align*}
    By inspection the corresponding ratio
    \begin{align*}
        r(\theta) := -\frac{x_\theta}{y_\theta} = \sqrt{1+\frac{\pi \cos(\theta)}{\sin(\theta) - \theta \cos(\theta)}}
    \end{align*}
    maps $(0,\pi)$ bijectively to $(0,\infty)$, showing that for each $x_r < 0 < y_r$, there is a unique EMP curve which welds these points to the base of the curve.  Asymptotically,
    \begin{align}\label{Eq:rFromThetaWang}
        r(\theta) = 1 - \frac{\pi}{2}\Big(\theta - \frac{\pi}{2}\Big) + O\Big(\theta - \frac{\pi}{2}\Big)^2
    \end{align}
    as $\theta \rightarrow \pi/2$, which inverts as
    \begin{align*}
        \theta(r) = \frac{\pi}{2} - \frac{2}{\pi}(r - 1) + O(r - 1)^2, \qquad r \rightarrow 1.
    \end{align*}
   The associated energy is thus
    \begin{align*}
        -8\log(\sin(\theta(r))) = \frac{16}{\pi^2}(r- 1)^2 + O(r-1)^3, \qquad r \rightarrow 1.
    \end{align*}
    For the EMW curves, on the other hand, \eqref{Eq:EMWEnergyFormula} says the energy at ratio $r$ is
    \begin{equation}\label{Eq:EMWEnergyTaylor}
        -8 \log\Big( \frac{2\sqrt{r}}{1+r} \Big) = (r-1)^2 + O(r-1)^3
    \end{equation}
    as $r \rightarrow 1$, and so we find \eqref{Lim:EnergySameBase} holds.
    
    For the second claim, first observe from the universal curve property of Theorem \ref{Thm:EMW}$(\ref{Thm:EMWSingleCurve})$ that for each $\theta$ is there is an EMW curve with tip $e^{i\theta}$.  We proceed to use the conformal map constructed in the proof to show there is a unique such curve and to find $\theta$ as a function of $r =-x/y$.  By symmetry it suffices to consider the case $r<1$, as there.  Recall the map $\tilde{F}_t$ in \eqref{Eq:TildeFt} sends zero to $-i+e^{i\theta}$, where $\Arg(i+e^{i\theta}) = 2\beta$ is twice the argument $\beta$ of the tip $\Gamma(t_r)$ of the corresponding point on the universal curve $\Gamma([0,t_r]) = \sqrt{C_\theta}$.  See Figure \ref{Fig:EMWUniversalConstruction}.  We thus have
    \begin{align}\label{Eq:EMWBeta}
        2\beta = \frac{3\pi}{2} - \arctan\Big( \frac{\Real \tilde{F}_t(0)}{\Imag \tilde{F}_t(0)} \Big) = \frac{3\pi}{2} - \arctan\Big( \frac{2(1+r)+(1-r)\log(r)}{\pi(1-r)} \Big).
    \end{align}
    By inspection, the function $r \mapsto \beta(r)$ is a bijection of $(0,1)$ to $(\pi/2,\pi)$, showing that for each $\beta$ there is a unique EMW curve with tip $e^{i\beta}$.  Furthermore,
    \begin{align*}
        \beta(r) = \frac{\pi}{2} - \frac{\pi}{8}(r-1) + O(r-1)^2, \qquad r \rightarrow 1^-,
    \end{align*}
    which inverts as
    \begin{align}\label{Eq:rFromThetaEMW}
        r(\beta) = 1 - \frac{8}{\pi}\Big(\beta- \frac{\pi}{2} \Big) + O\Big(\beta- \frac{\pi}{2} \Big)^2, \qquad \beta \rightarrow \frac{\pi}{2},
    \end{align}
    and hence from \eqref{Eq:EMWEnergyTaylor} the energy is
    \begin{align*}
        -8\log \Big( \frac{2\sqrt{r(\beta)}}{1+r(\beta)}\Big) = \frac{64}{\pi^2}\Big(\beta - \frac{\pi}{2}\Big)^2 + O\Big(\beta - \frac{\pi}{2}\Big)^3.
    \end{align*}
    The EMP curves $\gamma_\beta$, on the other hand, satisfy
    \begin{align*}
        I(\gamma_\beta) = -8\log(\sin(\beta)) = 4\Big(\beta - \frac{\pi}{2}\Big)^2 + O\Big(\beta - \frac{\pi}{2}\Big)^4
    \end{align*}
    as $\beta \rightarrow \pi/2$, yielding the claimed ratio in the limit.
\end{proof}

\begin{remark}\label{Remark:Reciprocals}
    It does not appear to be \emph{a priori} obvious that the two limits in Theorem \ref{Thm:EMWWang} should be reciprocals, as we proceed to explain.  Note our minimal-energy formulas have expansions \eqref{Eq:EMWEnergyTaylor} and
    \begin{align*}
        -8\log(\sin(\theta)) &= 4 \left(\theta -\frac{\pi
  }{2}\right)^2+O\left(\theta -\frac{\pi
  }{2}\right)^4,
  \qquad \theta \rightarrow \frac{\pi}{2}.
    \end{align*}
    
    Let $\theta$ be the argument of the tip of a curve welding $x$ to $y$, and suppose that as $\theta \rightarrow \pi/2$, the ratio $-x/y$ approaches unity $C$ times as fast for the EMW curves as for the EMP family.  That is, suppose the EMP curves' ratio $r_W$ satisfies $r_W(\theta) = 1 + D(\theta-\pi/2) + O(\theta-\pi/2)^2$, while the EMW's $r_E$ satisfies $r_E(\theta) = 1 + CD(\theta-\pi/2) + O(\theta-\pi/2)^2$.  Then the energy of the EMP curve as a function of $r$ would be 
    \begin{align*}
        -8\log(\sin(\theta(r))) = \frac{4}{D^2}(r-1)^2 + O(r-1)^3, 
    \end{align*}
    which combined with \eqref{Eq:EMWEnergyTaylor} shows the first ratio in Theorem \ref{Thm:EMWWang} would be $D^2/4$.  Similarly, the second ratio would be $C^2D^2/4$, and hence these are reciprocals if and only if  $C^2D^4=16$.  From \eqref{Eq:rFromThetaWang} we see $D=-\pi/2$, and so the EMW curves' ratio would need to approach $1$ precisely $16/\pi^2$ as fast as for the EMP's as $\theta \rightarrow \pi/2$.  From \eqref{Eq:rFromThetaEMW} this is the case, but from the outset it is not obvious that it should be so.   
\end{remark}

\begin{remark}\label{Remark:EnergiesNotAlwaysSame}
\begin{figure}
    \centering
    \includegraphics[scale=0.75]{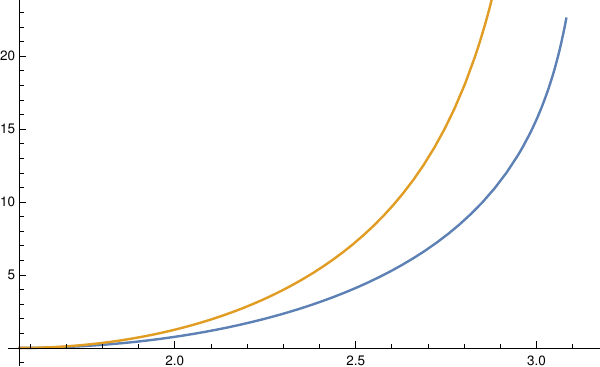}
    \caption{\small Energies for the Wang and EMW curves with the tip $e^{i\theta}$ in blue and orange, respectively, where the $x$-axis is $\theta \in (\pi/2, \pi)$ and the $y$-axis the energy.}
    \label{Fig:EnergySameTip}
\end{figure}
We also remark that the asymptotic comparisons in \eqref{Lim:EnergySameBase} and \eqref{Lim:EnergySameTip} do not hold for all ratios $r$ and all angles $\theta$.  With respect to \eqref{Lim:EnergySameTip}, for instance, it is easy to intuitively see that the energy of an EMW curve whose tip has angle $\theta$ close to $\pi^-$ (or $0^+$) is much larger than that for the corresponding EMP curve.  This is because the Wang asymptotically hits the real line with angle $\pi/3$ as $\theta \rightarrow \pi^-$ (reflect the plots in Figure \ref{Fig:BoundaryGeodesics} over the imaginary axis), while the EMW curve turns in to intersect $\mathbb{R}$ tangentially (see Figure \ref{Fig:UniversalEMW}).  At an intuitive level, the latter requires vastly more energy than the former.  This is confirmed in Figure \ref{Fig:EnergySameTip}, where we plot the EMP energy $-8\log(\sin(\beta))$ and the composition of the energy formula $-8\log(2\sqrt{r}/(1+r))$ for the EMW curves with the numerical inverse $r(\beta)$ of the EMW angle formula $\beta(r)$ given by \eqref{Eq:EMWBeta}.   
\end{remark}

\section{Appendix 1: other instances of upwards SLE$_0(\rho,\rho)$}\label{Sec:Appendix}
Theorem \ref{Thm:EMW}($\ref{Thm:EMWEnergy}$) says that the welding minimizers are upwards SLE$_0(-4,-4)$ starting from 
\begin{align*}
    (\xi(0), V_1(0),V_2(0)) = (0,x,y).
\end{align*}
The following lemma says that if we replace $-4$ with $-3$ or $-2$, we obtain orthogonal circular arc segments or straight line segments, respectively.

\begin{lemma}\label{Lemma:UpwardsSLE}
    Fix $x<0<y$. 
    \begin{enumerate}[$(i)$]
    \item\label{Lemma:StraightSlitUpwardsSLE} The hull $\gamma_\tau$ generated on $[0,\tau]$ by upwards SLE$_0(-2,-2)$ started from $(\xi(0), V_1(0), V_2(0)) = (0, x,y)$, where $\tau$ is the first hitting time of $\xi$ and one of the $V_j$, is a (straight) line segment whose angle of intersection with the right-hand side of $\mathbb{R}$ is $\frac{\pi y}{y-x}$, and the images of the force points $V_1(\tau)$ and $V_2(\tau)$ are the two prime ends corresponding to the base of $\gamma_\tau$. 
    \item\label{Lemma:CAUpwardsSLE} The hull $\tilde{\gamma}_t$ generated by upwards SLE$_0(-3,-3)$ started from $(\xi(0), V_1(0), V_2(0)) = (0, x,y)$ is a circular arc which meets $\mathbb{R}$ orthogonally, where $t$ is any time $0 < t \leq \tau$, the first hitting time of $\xi$ and one of the $V_j$.  At time $\tau$, $V_1(\tau)$ and $V_2(\tau)$ are the two prime ends corresponding to the base of $\tilde{\gamma}_\tau$.
    \end{enumerate}
\end{lemma}
\begin{remark}
    Krusell \cite{KrusellComm} has shown that the orthogonal circular arc which intersects $\mathbb{R}$ at 0 and $y$ is also both downwards SLE$_0(-6)$ and downwards SLE$_0(-3)$, with forcing starting at $y$ and $y/2$, respectively.
\end{remark}
\begin{proof}
    The upwards SLE$_0(-2,-2)$ claim in ($\ref{Lemma:StraightSlitUpwardsSLE}$) is equivalent to showing that $\gamma$'s downward Loewner flow is downwards SLE$_0(-2,-2)$ with forcing starting at the two prime ends at the base of $\gamma$.  That is, we need to show that for any $t>0$,
    \begin{align}\label{Eq:SLE-2NTS}
       \dot{\lambda}(t) = \frac{2}{g_t(0-) -\lambda(t)} + \frac{2}{g_t(0+)-\lambda(t)} =: \frac{2}{x_t} + \frac{2}{y_t}.
    \end{align}
    We verify this by explicit computation using the centered mapping up function $F_t$ (which is explicit, in contrast to the centered mapping down function $G_t$).

    For fixed $u < 0 < v$, recall that the conformal map
    \begin{align}\label{Eq:SlitMapUp}
    F(z) := (z-v)^{\frac{v}{v-u}}(z-u)^{\frac{-u}{v-u}} = z - (u+v) + \frac{uv}{2}\frac{1}{z} + O\Big( \frac{1}{z^2}\Big), \qquad z \rightarrow \infty,
    \end{align}
    maps $\mathbb{H}$ to the complement of a straight line segment $\tilde{\gamma}$ with base at the origin, taking the two intervals $[u,0]$ and $[0,v]$ to the two sides of $\tilde{\gamma}$.  The line segment $\tilde{\gamma}$ angle 
    \begin{align}\label{Eq:SlitAngle}
        \frac{\pi v}{v-u}
    \end{align}
    with $\mathbb{R}_{\geq 0}$ and capacity time $2t = -uv/2$ (see \cite[``The Slit Algorithm'']{MRZip} for a derivation).  Thus the centered mapping-up function $F_t:\mathbb{H} \rightarrow \mathbb{H} \backslash \gamma[0,t]$ for $\gamma$ is of this form, with $u=x_t$ and $v=y_t$.  Writing $\alpha := \frac{y_t}{y_t-x_t}$ as usual, we see $x_t = \frac{\alpha-1}{\alpha} y_t$ by \eqref{Eq:SlitAngle} and 
    \begin{align}\label{Eq:StraightSlitWeldEndpoints}
        t = -\frac{x_ty_t}{4} = \frac{1-\alpha}{4\alpha}y_t^2 \qquad \Rightarrow \qquad y_t = 2 \sqrt{\frac{\alpha}{1-\alpha}}\sqrt{t}, \quad x_t = -2\sqrt{\frac{1-\alpha}{\alpha}}\sqrt{t}.
    \end{align}
    From \eqref{Eq:SlitMapUp} we also see
    \begin{align*}
        \lambda(t) = -(x_t+y_t) = \frac{2(1-2\alpha)}{\sqrt{\alpha(1-\alpha)}}\sqrt{t}. 
    \end{align*}
    Differentiating and using the formulas for $x_t$ and $y_t$ in \eqref{Eq:StraightSlitWeldEndpoints} shows \eqref{Eq:SLE-2NTS} holds.
    
    For $(\ref{Lemma:CAUpwardsSLE})$ we start by considering the orthogonal circular arc growing from the origin towards $x=2\sqrt{2}$, and then argue this covers all possible cases.  This circular arc has downwards driver $\lambda(t) = 3\sqrt{2} - 3\sqrt{2}\sqrt{1-t}$ for $0 \leq t \leq 1$ by \eqref{Eq:CADriver1} (scale by $r=2\sqrt{2}$ and recall \eqref{Eq:ScaledDriver}).  We want to show that $(a)$ $\lambda$ is generated by downwards SLE$_0(-3,-3)$ with forcing starting at the two prime ends corresponding to the base, and that $(b)$ after any amount of time flowing down $t < 1$, what remains is an orthogonal circular arc segment.  Combined, this is equivalent to upwards SLE$_0(-3,-3)$ claim in the lemma statement.
    
    For $(a)$ we must show
    \begin{align}\label{Eq:UpwardsSLE33}
        \dot{\lambda}(t) = \frac{3}{g_t(0-)-\lambda(t)} + \frac{3}{g_t(0+)-\lambda(t)}
    \end{align}
    for any $0 <t <1$.  Scaling \eqref{Eq:CAWeldEndpoints} by $2\sqrt{2}$, we find
    \begin{align*}
        g_t(0-)-\lambda(t) &= -2\sqrt{2} +2\sqrt{2}\sqrt{1-t}-2\sqrt{2}\sqrt{1-\sqrt{1-t}},\\
        g_t(0+)-\lambda(t) &= -2\sqrt{2} +2\sqrt{2}\sqrt{1-t}+2\sqrt{2}\sqrt{1-\sqrt{1-t}},
    \end{align*}
    from which arithmetic shows that \eqref{Eq:UpwardsSLE33} holds.  Furthermore, $(b)$ immediately follows from the self-similarity of the circular arc family: what remains upon mapping down is still a member of this family \cite[Theorem 3.2]{LoewnerCurvature}.  
    
    Note that a scaled SLE$_0(\rho,\rho)$ is still an SLE$_0(\rho,\rho)$, as is an SLE$_0(\rho, \rho)$ reflected in the imaginary axis.  Now, since the ratio $\frac{\lambda(t)-V_1(t)}{V_2(t)-\lambda(t)}$ for the above circular arc attains all values in the interval $(1,\infty)$ as $t$ varies in $(0,1)$, after scaling and reflection this case covers all possible initial conditions $(0,x,y)$. 
\end{proof}

\section{Appendix 2: the orthogonal circular arc family}\label{Sec:Appendix2}
In this appendix we compare hyperbolic geodesic segments in $\mathbb{H}$, or the orthogonal circular arc family, to the EMP and EMW curves. The geodesic segments are upwards SLE$_0(-3,-3)$ by Lemma \ref{Lemma:UpwardsSLE}, and the next lemma says they have a distinguished place in terms of energy, both with respect to the EMP and EMW families.
\begin{lemma}\label{Lemma:CircularArc} If $\theta \neq \pi/2$, the unique orthogonal circular arc from 0 to $z_0 = re^{i \theta}$ uses exactly $9/8$th of the energy of the EMP curve to $z_0$. If $x<0<y$ with $-x \neq y$, the unique orthogonal circular arc which welds $x$ to $y$ uses exactly $9/8$th the energy of the EMW curve which welds these points.
\end{lemma}
Of course, if $\theta = \pi/2$ or $-x=y$, then all the curves in question are the same and use zero energy.  The surprising aspect here, perhaps, is that there is instant and constant disagreement as soon as we deviate from the vertical line segment.  Given this lemma, we can guess the constants in Theorem \ref{Lemma:Inf98}, because sufficiently-smooth curves should locally be like circular arcs.\footnote{One could make this precise by comparing the expansion of the circular arc from $0$ to $\gamma(\delta)$ to \eqref{Eq:InfinitesimalCurve}, and the expansion of the circular arc which welds $x(\delta)$ and $y(\delta)$ to \eqref{Eq:InfinitesimalWelding}.  Theorem \ref{Lemma:Inf98} is saying the ``local similarity'' holds at least in the energy sense.}

\begin{proof}
By reflection and scaling invariance it suffices to consider the orthogonal circular arc $\gamma_1$ from $0$ to $x=1$, and here the results follows from explicit computations.  Indeed, let $\tau$ be the time such that $\gamma_1(\tau)=1$, and consider the segment $\gamma_1([0,t])$ for some $t <\tau$, and write $\gamma_1(t) = re^{i\theta}$.  The centered mapping down function $G_t$ for $\gamma_1([0,t])$ is a the composition of a \Mob automorphism $M_ 1( z ) = \frac{z}{1-z}$ of $\mathbb{H}$ sending 1 to $\infty$, followed by $S(z)  :=  \sqrt{z^2+\cot^2(\theta)}$ which pulls down a segment of the imaginary axis (we take the branch cut on the real line), followed by a second \Mob map   $M_2(z) = \frac{\sin^2(\theta)z}{z+\csc(\theta)}$ renormalizing at $\infty$. We find
\begin{align*}  
    G_t(z) = \frac{\sin^2(\theta) \sqrt{\big( \frac{z}{1-z} \big)^2+\cot^2(\theta)}}{\sqrt{\big( \frac{z}{1-z} \big)^2+\cot^2(\theta)}+\csc(\theta)} = z - \frac{3}{2}\cos^2(\theta) + \frac{1}{4}\big( 1 -\sin^4(\theta) \big)\frac{1}{z} + O\Big( \frac{1}{z^2} \Big), \qquad z \rightarrow  \infty.
\end{align*}
Recalling \eqref{Eq:LoewnerNormalize} and \eqref{Eq:hcapNormalize}, we see the time $t(\theta)$ needed to achieve argument $\theta$ is 
\begin{align}\label{Eq:CATime}
    t(\theta) = \frac{1}{8}\big(1-\sin^4(\theta)\big),
\end{align}
and hence the downwards driver is
\begin{align}\label{Eq:CADriver1}
    \lambda_1(t) = \frac{3}{2}\cos^2(\theta(t)) = \frac{3}{2}(1-\sqrt{1-8t}), \qquad 0 \leq t \leq \frac{1}{8},
\end{align}
and the energy to a given $\theta$ is
    \begin{align*}
        \frac{1}{2} \int_0^{t(\theta)} \frac{36}{1-8s}ds = -9 \log(\sin(\theta)).
    \end{align*}
The first claim holds in light of \eqref{Eq:WangEnergy}.
    
    For the welding result, we note from the above conformal map that the centered conformal welding $\varphi_t$ for the arc segment $\gamma_1[0,t]$, $0<t<1/8$, satisfies $\varphi_t:[x_t,0] \rightarrow [0,y_t]$, where
    \begin{equation}\label{Eq:CAWeldEndpoints}
        \begin{aligned}
        x_t = M_2(-\cot(\theta)) &= -1 + \sqrt{1-8t} - \sqrt{1-\sqrt{1-8t}},\\
        y_t =M_2(\cot(\theta) &= -1 + \sqrt{1-8t} + \sqrt{1-\sqrt{1-8t}}.
        \end{aligned}
    \end{equation}
    In fact, the welding $\varphi_t$ is simply $x \mapsto -x$ conjugated by $M_2$, which is explicitly 
    \begin{align}\label{Eq:CircularArcWelding}
        \varphi_t(x) = - \frac{\sqrt{1-8t}\,x}{\sqrt{1-8t} -2x}.
    \end{align}
    By inspection, the ratio $\alpha_t = \frac{y_t}{y_t-x_t}$ is monotonically decreasing, and inverting we have the unique time $t(\alpha)$ to obtain a given $\alpha \in (0,1/2)$ is
    \begin{align}\label{Eq:CATForAlpha} 
        t(\alpha) = \frac{1}{8} -2\alpha^2(1-\alpha)^2,
    \end{align}
    and thus the corresponding energy is
    \begin{align*}
        \frac{1}{2} \int_0^{t(\alpha)} \frac{36}{1-8s}ds = -9\log(2\sqrt{\alpha(1-\alpha)}),
    \end{align*}
     exactly $9/8$ of \eqref{Eq:EMWEnergyFormula}. 
\end{proof}

The hyperbolic geodesic segments also share the ``universality'' properties of both the EMP and EMW families.
\begin{lemma}
    The orthogonal circular arc family satisfies the following:
    \begin{enumerate}[$(i)$]
        \item\label{Lemma:CAUniversalDriver} For any fixed $\theta \in (0,\pi)\bs \{\pi/2\}$, both the upwards driver and the conformal welding for the single orthogonal circular arc to $e^{i\theta}$ generate \emph{all} orthogonal circular arcs for angles $\alpha \in (0,\pi)\bs \{\pi/2\}$ up to translation, scaling, and reflection in the imaginary axis.
        \item\label{Lemma:CAUniversalCurve} There is an explicit algebraic curve which generates all orthogonal circular arcs up to translation, scaling and reflection in the imaginary axis.
    \end{enumerate}
\end{lemma}
\begin{proof}
    Take any half circle centered on $\mathbb{R}$ for $(\ref{Lemma:CAUniversalCurve})$.  For $(\ref{Lemma:CAUniversalDriver})$, let $\gamma_\theta$ be the circular arc segment from the origin to $e^{i\theta}$, where without loss of generality $0 < \theta < \pi/2$.  Its capacity time is $\tau = \frac{1}{8}\big(1 - \sin^4(\theta) \big)$ by \eqref{Eq:CATime}, and by \eqref{Eq:CAWeldEndpoints} and \eqref{Eq:CircularArcWelding} its welding is
    \begin{align}\label{Eq:CircularArcWelding2}
        \varphi(x) = - \frac{\sqrt{1-8\tau}\,x}{\sqrt{1-8\tau} -2x}, \qquad  -1 + \sqrt{1-8\tau} - \sqrt{1-\sqrt{1-8\tau}} \leq x \leq 0.
    \end{align}
    If we scale $\gamma_\theta$ by a factor of $r := 1/\sqrt{1-8\tau}$, the resulting welding is 
    \begin{align*}
        r \varphi(x/r) = - \frac{x}{1-2x}, \qquad -\frac{1}{\sqrt{1-8\tau}} + 1 - \sqrt{\frac{1}{\sqrt{1-8\tau}}-1} \leq x \leq 0.
    \end{align*}
    Since the welding formula does not depend on $\tau$, we see that we may generate \emph{all} circular arcs, up to translation, scaling and reflection in the imaginary axis, from the single welding \eqref{Eq:CircularArcWelding2}.
    
    By \eqref{Eq:CADriver1} the downwards driver for $\gamma_\theta$ is $\lambda(t) = \frac{3}{2}(1 -\sqrt{1-8t})$, $0 \leq t \leq \tau$, and so the downwards driver for $r \gamma_\theta$ is 
    \begin{align*}
        r \lambda(t/r^2) = \frac{3}{2\sqrt{1-8\tau}}\big(1-\sqrt{1-8t(1-8\tau)}  \,\big), \qquad 0 \leq t \leq \frac{\tau}{1-8\tau}.
    \end{align*}
    The corresponding (centered) upwards driver is
    \begin{align*}
        \xi(t) = r \lambda\Big(\frac{1}{r^2}\big( \frac{\tau}{1-8\tau} - t\big) \Big) - r \lambda\Big(\frac{\tau}{r^2(1-8\tau)}  \Big) = \frac{3}{2}\big(1 - \sqrt{1+8t} \,\big).
    \end{align*}
    Since this formula is also independent of $\theta$, we conclude the upwards driver corresponding to $\lambda$ generates all orthogonal circular arcs up to scaling, translation and reflection.    
\end{proof}

\bibliographystyle{acm}
\bibliography{Tp}

\end{document}